\documentclass{amsart}

\usepackage[usenames,dvipsnames]{xcolor}

\usepackage{enumerate}
\usepackage{xy}
\xyoption{matrix}
\xyoption{all}
\usepackage{array}
\usepackage{amssymb}
\usepackage{psfrag}
\usepackage{graphicx}
\usepackage{color}
\usepackage{verbatim}
\usepackage{pgf}
\usepackage{mathtools}
\usepackage{microtype}
\usepackage[colorlinks=true,linkcolor=blue,pagebackref,citecolor=violet]{hyperref}
\usepackage{enumitem}
\setenumerate{label=(\alph*)}
\usepackage[only,llbracket,rrbracket]{stmaryrd}

\usepackage{tikz,ifthen,tikz-cd}
\usetikzlibrary{backgrounds}
\usetikzlibrary{calc,decorations.markings,arrows}
\usetikzlibrary{positioning} 
\usetikzlibrary{petri} 
\usetikzlibrary{decorations.pathmorphing} 
\usetikzlibrary{fit} 
\usetikzlibrary{backgrounds}

\theoremstyle{plain}
\newtheorem{lemma}{Lemma}[section]
\newtheorem{theorem}[lemma]{Theorem}
\newtheorem{corollary}[lemma]{Corollary}

\newtheorem{proposition}[lemma]{Proposition}

\newtheorem{thm*}{Theorem}

\theoremstyle{definition}
\newtheorem{definition}[lemma]{Definition}
\newtheorem{remark}[lemma]{Remark}
\newtheorem{example}[lemma]{Example}

\theoremstyle{remark}
\newtheorem*{note}{Note}

\numberwithin{equation}{section}

\newcommand{\Gr}{\operatorname{Gr}}
	\newcommand{\Grcone}{\operatorname{\widehat{\Gr}}}
\newcommand{\End}{\operatorname{End}}
	\newcommand{\stabEnd}{\operatorname{\underline{End}}}
\newcommand{\Hom}{\operatorname{Hom}}
	\newcommand{\stabHom}{\operatorname{\underline{Hom}}}
\newcommand{\Ext}{\operatorname{Ext}}

\newcommand{\Sub}{\operatorname{Sub}}
	\newcommand{\stabSub}{\operatorname{\underline{Sub}}}
\newcommand{\ind}{\operatorname{ind}}
\newcommand{\add}{\operatorname{add}}
\newcommand{\module}{\operatorname{mod}}
	\newcommand{\projmod}{\operatorname{proj}}

\newcommand{\opp}{\mathrm{opp}}

\newcommand{\RHom}{\operatorname{\mathbf{R}Hom}}
\newcommand{\CM}{\operatorname{CM}}
\newcommand{\GP}{\operatorname{GP}}

\newcommand{\bdd}{\mathrm{b}}

\newcommand{\F}{\mathcal{F}}

\newcommand{\PR}{\mathcal{P}}

\newcommand{\C}{\mathcal{C}}
\newcommand{\CC}{\mathbb{C}}
\newcommand{\B}{\mathcal{B}}
\newcommand{\D}{\mathcal{D}}
\newcommand{\E}{\mathcal{E}}
\newcommand{\K}{\mathrm{K}}

\newcommand{\QQ}{\mathbb{Q}}
\newcommand{\T}{\mathcal{T}}
\newcommand{\U}{\mathcal{U}}

\newcommand{\X}{\mathcal{X}}
\newcommand{\ZZ}{\mathbb{Z}}

\newcommand{\M}{\ensuremath{M^{\perp_1}}}

\newcommand{\Ical}{\mathcal{I}}

\newcommand{\bSigma}{\mathbf{\Sigma}}

\newcommand{\A}{\mathcal{A}}

\newcommand{\field}{\mathbb{K}}

\colorlet{shadecolor}{gray!70}

\newcommand{\mc}[1]{\ensuremath{\mathcal{#1}}}   %
\newcommand{\col}[1]{{\color{lightgray} #1}}
\newcommand{\bl}[1]{{\color{blue} #1}}
\newcommand{\gr}[1]{{\color{Purple} #1}}
\newcommand{\ye}[1]{{\color{red} #1}}

\newcommand{\dual}{\mathrm{D}}
\newcommand{\EE}{\mathbb{E}}

\newcommand{\FKcc}[2]{\Phi_{#1}^{#2}}
\newcommand{\clucha}[2]{\Phi_{#1}^{#2}}

\newcommand{\QGr}[2]{\mathrm{Gr}_{#1}(#2)}
\newcommand{\s}{\mathfrak{s}}

\newcommand{\infl}{\rightarrowtail}
\newcommand{\defl}{\twoheadrightarrow}
\newcommand{\confl}{\dashrightarrow}
\tikzcdset{
infl/.style={rightarrowtail},
defl/.style={twoheadrightarrow},
confl/.style={dashed},
exists/.style={dotted}}
\newcommand{\lperp}[1]{{}^{\perp_1}{#1}}
\newcommand{\rperp}[1]{{#1}^{\perp_1}}

\newcommand{\powser}[2]{#1\llbracket#2\rrbracket}

\begin{document}

\title{Reduction of Frobenius extriangulated categories}

\author[Faber]{Eleonore Faber}
\address[EF]{
School of Mathematics,
University of Leeds,
Leeds LS2 9JT, U.K.
}
\email{E.M.Faber@leeds.ac.uk}

\author[Marsh]{Bethany Rose Marsh}
\address[BRM]{
School of Mathematics,
University of Leeds,
Leeds LS2 9JT, U.K.
}
\email{B.R.Marsh@leeds.ac.uk}

\author[Pressland]{Matthew Pressland}
\address[MP]{
School of Mathematics \& Statistics,
University of Glasgow,
Glasgow G12 8QQ, U.K.
}
\email{matthew.pressland@glasgow.ac.uk}

\begin{abstract}
We describe a reduction technique for stably $2$-Calabi--Yau Frobenius extriangulated categories $\F$ with respect to a functorially finite rigid subcategory $\X$.
The reduction of such a category is another category $\rperp{\X}\subseteq\F$ of the same kind, whose cluster-tilting subcategories are those cluster-tilting subcategories $\T\subseteq\F$ such that $\X\subseteq\T$.
This reduction operation generalises Iyama--Yoshino's reduction for $2$-Calabi--Yau triangulated categories, which is recovered by passing to stable categories.
Moreover, for a certain class of categories $\F$ and rigid objects $M$, we show that the relationship between $\F$ and $\M$ may also be expressed in terms of internally Calabi--Yau algebras, in the sense of the third author.
As an application, we give a conceptual proof of a result on frieze patterns originally obtained by the first author with Baur, Gratz, Serhiyenko, and Todorov.
\end{abstract}

\thanks{
All three authors would like to thank Mikhail Gorsky and Yann Palu for useful conversations.
This work was partially supported by a grant from the Simons Foundation, supporting the second author.
This work was supported by the Engineering and Physical Sciences Research Council [grant numbers EP/R014604/1, EP/T001771/2, EP/W007509]. We thank the Isaac Newton Institute for Mathematical Sciences in Cambridge for support and hospitality during the programme \emph{Cluster algebras and representation theory} in 2021, where work on this paper was undertaken.
This work was also supported by the Centre for Advanced Study, Oslo, which the second author visited as part of the programme \emph{Representation theory: combinatorial aspects and applications} in 2023.
}

\maketitle

\section{Introduction}

Cluster categories, as introduced by Buan, Marsh, Reineke, Reiten, and Todorov \cite{BMRRT06} and generalised by Amiot \cite{Amiot}, are $2$-Calabi--Yau triangulated categories which provide a powerful framework for connecting the combinatorics of cluster algebras \cite{FZ02} to the representation theory of quivers and finite-dimensional algebras.
These constructions cover cluster algebras defined from quivers admitting a Jacobi-finite non-degenerate potential (for example, acyclic quivers) and having no frozen variables.

Many of the most important families of examples of cluster algebras are those isomorphic to the coordinate rings of algebraic varieties, and these typically do have frozen variables.
These examples include coordinate rings of Grassmannians \cite{Scott06} and other (partial) flag varieties \cite{GLS08} and open pieces within them, Schubert cells and other open positroid varieties \cite{GL,SSBW}, and more general braid varieties \cite{CGGLSS}.
Many examples of this kind also have categorifications, now not by $2$-Calabi--Yau triangulated categories but rather by stably $2$-Calabi--Yau Frobenius exact categories \cite{DI,GLS08,JKS16,Pre22}.

Motivated in part by the many similarities in the theories of triangulated and of exact categories, Nakaoka and Palu \cite{NakaokaPalu} recently introduced a common generalisation, the notion of an extriangulated category.
Subsequently, Yilin Wu \cite{Wu} has introduced a cluster category associated to any Jacobi-finite ice quiver with potential, in the style of Amiot's construction, which is a stably $2$-Calabi--Yau Frobenius extriangulated category, but typically neither triangulated nor exact.

A significant result in the theory of triangulated cluster categories is that of Iyama--Yoshino reduction \cite{IY08}.
This construction takes as input a $2$-Calabi--Yau triangulated category $\C$ and a rigid object $M\in\C$ and outputs a second $2$-Calabi--Yau triangulated category whose cluster-tilting subcategories (corresponding to seeds in a cluster algebra) are in bijection with those of $\C$ which contain $M$.
On the cluster algebra side, this provides information on the collection of seeds containing a chosen fixed set of cluster variables.
Iyama--Yoshino reduction is closely related to other reduction theorems in representation theory, such as $\tau$-tilting reduction \cite{Jasso} and silting reduction \cite{AI12}.

In this paper, we describe an Iyama--Yoshino-style reduction procedure for stably $2$-Calabi--Yau Frobenius extriangulated categories in general.
Precisely, we show the following.
\begin{thm*}
\label{t:subcategory}
If $\F$ is a stably $2$-Calabi--Yau Frobenius extriangulated category, and $\X\subseteq\F$ is a functorially finite rigid subcategory, then the full extension-closed subcategory
\[\rperp{\X}=\{M\in\F:\text{$\EE_{\F}(X,M)=0$ for all $X\in\X$}\}\]
is itself a stably $2$-Calabi--Yau Frobenius extriangulated category (Proposition~\ref{Prop:MperpisFrobenius-extri}), and its cluster-tilting subcategories are those of $\F$ which contain $\X$ (Proposition~\ref{p:ct-bij-extri}).
\end{thm*}

Note that the article~\cite{ZZ18} gives results related to Iyama--Yoshino reduction for extriangulated categories. However, these results involve taking additive factor categories or modifying the extriangulated structure in the original category, rather than taking a subcategory as we do here.

In Section~\ref{s:examples}, we give an example illustrating this theorem in the case of a Grassmannian cluster category~\cite{JKS16} by computing the Auslander--Reiten quiver of the subcategory of Grassmannian cluster category of $\M$ for $M$ an indecomposable rigid object in the Grassmannian cluster category for the Grassmannian $\Gr(2,6)$.

We relate this reduction operation to Iyama--Yoshino's by showing that the triangulated stable categories $\underline{\F}$ and $\underline{\rperp{\X}}$ are related by Iyama--Yoshino reduction at $\X$.

\begin{thm*}[Theorem~\ref{thm:triangleequivalence-extri}]
If $\F$ is a stably $2$-Calabi--Yau Frobenius extriangulated category and $\X\subseteq\F$ is a functorially finite rigid subcategory, then there is a triangle equivalence
\[\underline{\rperp{\X}}=\rperp{\X}_{\underline{\F}}/\add(\X)\]
between the stable category of the reduction of $\F$ at $\X$, and the Iyama--Yoshino reduction of the stable category $\underline{\F}$ at $\X$.
\end{thm*}

If $\F$ has a cluster-tilting object $T$ and $\X\subseteq\add(T)$, then both $\F$ and $\rperp{\X}$ admit a cluster character with respect to $T$, thanks to work of Wang, Wei and Zhang. We show in Section~\ref{s:CC} that, as one would expect, the cluster character on $\rperp{\X}$ is the restriction of that on $\F$.

In Section~\ref{s:icy}, we consider the construction of stably $2$-Calabi--Yau Frobenius exact categories $\F=\GP(B)$ from an internally Calabi--Yau pair $(A,e)$, in the sense of \cite{Pre17}.
Here $A$ is a Noetherian algebra and $e=e^2\in A$ an idempotent element, satisfying several conditions, most notably a Calabi--Yau symmetry property (Definition~\ref{def:bi3cy}).
These conditions imply that the category $\GP(B)$ of Gorenstein projective modules over $B=eAe$ is a stably $2$-Calabi--Yau Frobenius exact category containing the cluster-tilting object $eA$.
It is immediate from the definition that if $e'\in A$ is an idempotent such that $ee'=e=e'e$, then $(A,e')$ is another internally Calabi--Yau pair, and we obtain a second Frobenius exact category $\GP(B')$, for $B'=e'Ae'$. 

\begin{thm*}[Theorem~\ref{t:GP-reduction}]
In the context of the previous paragraph, the Frobenius exact category $\GP(B')$ is naturally equivalent to the reduction $\M$ of $\GP(B)$ at the rigid object $M=eAe'$.
\end{thm*}

Propp~\cite{propp20} and Cuntz--Holm--J{\o}rgensen \cite[Defn.~2.1]{CuntzHolmJorgensen} introduced the notion of a frieze pattern with coefficients. Such friezes are defined in a way similar to Conway--Coxeter friezes~\cite{CCFrieze1,CCFrieze2}, but allow the entries in the top and bottom rows to be arbitrary non-zero values. Conway--Coxeter friezes are closely related to cluster algebras (a relationship first pointed out in~\cite{CC06}) and, as mentioned in~\cite[\S1]{CuntzHolmJorgensen}, passing to a frieze pattern with coefficients corresponds to allowing the frozen variables to take values other than $1$.

In~\cite[\S4]{CuntzHolmJorgensen}, it is pointed out that frieze patterns with coefficients can be obtained from Conway--Coxeter frieze patterns by cutting out subpolygons from the corresponding triangulation of a regular polygon. Part of the motivation for this paper was to understand this procedure from a categorical point of view (i.e.\ using the Grassmannian cluster category~\cite{JKS16}). We give an example of this relationship in Section~\ref{s:friezes}.

Moreover, we apply our reduction method to mesh friezes coming from Grassmannian cluster categories of finite type and use it to give an alternative proof of the following result \cite[Prop.~5.3]{BFGST}, which was proved in an ad hoc way in loc.~cit.

\begin{thm*}[Theorem~\ref{Thm:friezered}] Let $F$ be a mesh frieze coming from a Grassmannian cluster category of finite type $\F$ and let $M$ be a rigid indecomposable non-projective object in $\F$. If $F(M)=1$, then $F|_{M^{\perp_1}}$ is a mesh frieze for the category $M^{\perp_1}$.
\end{thm*}

\section{Background and definitions}

Here we briefly discuss the features of
extriangulated categories that we need, while introducing notation.
We do not recall the full definition here, but note that full details can be found in Nakaoka and Palu's original article~\cite{NakaokaPalu} or in Palu's survey \cite{Palu-Extri}.

Let $\field$ be a field.
A $\field$-linear \emph{extriangulated category} $\C=(\C,\EE_{\C},\s)$ is a $\field$-linear additive category $\C$, a $\field$-bilinear bifunctor $\EE_{\C}\colon\C\times\C\to\module{\field}$, and a family $\s$ of \emph{realisation maps} assigning to each element $\delta\in\EE_{\C}(C,A)$, for objects $A,C\in\C$, an equivalence class of diagrams
\[\begin{tikzcd}
A\arrow{r}{f}&B\arrow{r}{g}\arrow{r}&C
\end{tikzcd}\]
with $g\circ f=0$.
The diagrams in classes appearing in this way are called \emph{conflations} (or sometimes extriangles), and we write
\[\begin{tikzcd}
A\arrow[infl]{r}{f}&B\arrow[defl]{r}{g}\arrow{r}&C\arrow[confl]{r}{\delta}&\phantom{}
\end{tikzcd}\]
to indicate that the pair of morphisms $(f,g)$ is a conflation in the equivalence class $\s(\delta)$.
An \emph{inflation} is a morphism which appears as the first map in a conflation, and a \emph{deflation} is a morphism which appears as the second map in a conflation.
In such a conflation, $g$ is called a \emph{cone} of $f$ and $f$ is called \emph{cocone} of $g$.
We also apply these terms to the objects $C$ and $A$ respectively.
The triple $\C=(\C,\EE_{\C},\s)$ must satisfy various axioms which can be found in \cite{NakaokaPalu}.

\begin{example}
\label{eg:extri-cats}
While we establish our abstract results for general extriangulated categories, our applications are primarily in the following two extremal cases.
\begin{enumerate}
\item An exact category $\C$ is canonically extriangulated by taking $\EE_{\C}(C,A)=\Ext^1_{\C}(C,A)$, and $\s$ to be the Yoneda realisation function associating a short exact sequence
\[\begin{tikzcd}
\s(\delta)\colon\ 0\arrow{r}&A\arrow{r}{f}&B\arrow{r}{g}\arrow{r}&C\arrow{r}&0
\end{tikzcd}\]
to each $\delta\in\EE_{\C}(C,A)$.
In the terminology of \cite{Buehler}, the conflations are the admissible short exact sequences, the inflations the admissible monomorphisms, and the deflations the admissible epimorphisms.
\item A triangulated category $\C$ is canonically extriangulated by taking $\EE_{\C}(C,A)=\Hom_\C(C,\Sigma A)$, and $\s$ defined by completing a morphism $\delta\in\EE_{\C}(C,A)$ to a distinguished triangle
\[\begin{tikzcd}
A\arrow{r}{f}&B\arrow{r}{g}\arrow{r}&C\arrow{r}{\delta}&\Sigma A
\end{tikzcd}\]
and then forgetting the map $\delta$.
The conflations are these truncated triangles, and every morphism is both an inflation and a deflation.
\end{enumerate}
The final class of examples appearing explicitly in this paper is the following; see \cite[Rem.~2.18]{NakaokaPalu}.
\begin{enumerate}[resume]
\item\label{eg:subcats} Let $\C$ be an extriangulated category and $\D\subseteq\C$ an extension-closed subcategory; that is, if
\[\begin{tikzcd}
A\arrow[infl]{r}&B\arrow[defl]{r}\arrow{r}&C\arrow[confl]{r}&\phantom{}
\end{tikzcd}\]
is a conflation in $\C$ such that $A,C\in\D$, then $B\in\D$.
Then $\D$ naturally inherits an extriangulated structure by restricting $\EE_{\C}$ and $\s$ from $\C$.
The conflations in $\D$ are conflations in $\C$ with all terms in $\D$.
A map in $\D$ is an inflation if and only if it is an inflation in $\C$ and its cone is in $\D$ and, dually, a map in $\D$ is a deflation if and only if it is a deflation in $\C$ and its cocone is in $\D$.

If $\C$ is an exact category then so is $\D$ (see e.g.~\cite[Lem.~10.20]{Buehler}), but if $\C$ is triangulated then $\D$ may be neither triangulated nor exact.
A standard example is to take $\C=K^\bdd(\projmod{A})$ the category of bounded complexes of projective modules over a (non-zero) finite-dimensional algebra $A$, and $\D=K^{[-1,0]}(\projmod{A})$ the extension closed subcategory of complexes concentrated in degrees $-1$ and $0$.
This category, with the extriangulated structure described above, is not triangulated because it has the non-zero projective object $0\to A$ (contradicting \cite[Cor.~7.4]{NakaokaPalu}), and is not exact because the morphism from $0\to A$ to the zero object is an inflation (since its cone $A\to 0$ lies in $K^{[-1,0]}(\projmod{A})$), but not a monomorphism.
\end{enumerate}
\end{example}

\begin{definition}
\label{d:Frobenius}
If $\C$ is an extriangulated category, an object $P\in\C$ is \emph{projective} if $\Hom_{\C}(P,f)$ is surjective for any deflation $f$, and \emph{injective} if $\Hom_{\C}(g,P)$ is surjective for any inflation $g$~\cite[Defn.\ 3.23]{NakaokaPalu}.
The category $\C$ has \emph{enough projectives} if for every object $X\in\C$ there is a deflation $P\defl X$ with $P$ projective, and it has \emph{enough injectives} if for every $X\in\C$ there is an inflation $X\infl P$ with $P$ injective~\cite[Defn.\ 3.25]{NakaokaPalu}.
We say $\C$ is a \emph{Frobenius extriangulated category}~\cite[Defn.\ 7.1]{NakaokaPalu} if it has enough projective and enough injective objects, and an object is projective if and only if it is injective.
\end{definition}

\begin{example}
An exact category $\E$ is Frobenius as an extriangulated category if and only if it is Frobenius as an exact category in the usual sense.
A triangulated category $\C$ is always Frobenius as an extriangulated category: only the zero object is projective or injective, but this is enough because of the conflations
\[\begin{tikzcd}[row sep=0pt]
\Sigma^{-1}X\arrow[infl]{r}&0\arrow[defl]{r}&X\arrow[confl]{r}&\phantom{}\\
X\arrow[infl]{r}&0\arrow[defl]{r}&\Sigma X\arrow[confl]{r}&\phantom{}
\end{tikzcd}\]
which exist for any $X\in\C$.
In fact, by \cite[Cor.~7.6]{NakaokaPalu}, the class of Frobenius extriangulated categories with full subcategory of projective-injectives $\PR=\{0\}$ is precisely the class of triangulated categories.
\end{example}

As is familiar from the theory of exact categories \cite[Thm.~I.2.6]{Happel}, a Frobenius extriangulated category has an associated triangulated stable category.

\begin{theorem}[{\cite[Cor.~7.4]{NakaokaPalu}}]
\label{t:stabcat}
If $\F$ is a Frobenius extriangulated subcategory and $\PR$ its full subcategory of projective-injective objects, then the stable category $\underline{\F}=\F/(\PR)$ is canonically triangulated.
\end{theorem}

We will adopt the usual notation $\stabHom_{\F}(-,-):=\Hom_{\underline{\F}}(-,-)$ for morphism spaces in $\underline{\F}$.
Some more details on the triangulated structure of $\underline{\F}$ may be found in Section~\ref{s:IYred}.
We recall in particular that if
\[\begin{tikzcd}
X\arrow[infl]{r}&P\arrow[defl]{r}\arrow{r}&Y\arrow[confl]{r}&\phantom{}
\end{tikzcd}\]
is a conflation in $\F$ with $P\in\PR$, then $Y\cong\Sigma X$ in $\underline{\F}$.
Moreover, $\stabHom_{\F}(Z,\Sigma X)=\EE_{\F}(Z,X)$ for any $Z\in\F$.

\section{Reduction of extriangulated categories}

Throughout this section, we let $\C=(\C,\EE_\C,\s)$ be an extriangulated category.
If $\X\subset\C$ is a subcategory, we write
\begin{align*}
\X^{\perp_1}&=\{Y\in\C:\text{$\EE_{\C}(X,Y)=0$ for all $X\in\X$}\},\\
{}^{\perp_1}\X&=\{Y\in\C:\text{$\EE_{\C}(Y,X)=0$ for all $X\in\X$}\},
\end{align*}
and abbreviate $X^{\perp_1}=\add(X)^{\perp_1}$ and ${}^{\perp_1}{X}=\prescript{\perp_1}{}{\smash{\add(X)}}$ when $X\in\C$ is an object. 
We say that a subcategory $\X\subseteq\C$ is \emph{rigid} if $\X\subseteq\lperp{\X}$ (or equivalently $\X\subseteq\rperp{\X}$).
An object $X\in\C$ is rigid if $\add(X)$ is rigid.

\begin{lemma} \label{l:closed-extri}
For any subcategory $\X$, the subcategories $\X^{\perp_1}$ and ${}^{\perp_1}\X$ are extension closed in $\C$.
\end{lemma}
\begin{proof}
Let $A\infl B\defl C\confl$ 
be a conflation in $\C$, with $A$ and $C$ in $\X^{\perp_1}$.
Then by \cite[Cor.~3.12]{NakaokaPalu} we have an exact sequence
\[\begin{tikzcd}
\EE_\C(X,A)\arrow{r}& \EE_\C(X,B)\arrow{r}& \EE_\C(X,C)\end{tikzcd}\]
for any $X\in\X$. Since $\EE_\C(X,A)=0$ and $\EE_\C(X,C)=0$, we have $\EE_\C(X,B)=0$,
so $B\in \X^{\perp_1}$.
The proof for ${}^{\perp_1}\X$ is dual.
\end{proof}

As a consequence of Lemma~\ref{l:closed-extri}, both $\X^{\perp_1}$ and ${}^{\perp_1}\X$ become extriangulated categories in their own right by Example~\ref{eg:extri-cats}\ref{eg:subcats}.

\begin{definition}
An object $P$ in $\C$ is \emph{Ext-projective} if
$\EE_\C(P,X)=0$ for all $X$ in $\C$, and \emph{Ext-injective} if
$\EE_\C(X,P)=0$ for all $X$ in $\C$.
\end{definition}

Recall the notions of projectivity and injectivity from Definition~\ref{d:Frobenius}, in terms of lifting properties for deflations and inflations respectively.
We recall also the following from~\cite{NakaokaPalu}.

\begin{lemma}[{\cite[Prop.\ 3.24]{NakaokaPalu}}]
\label{l:proj-extri}
An object in $\C$ is
Ext-injective (respectively, Ext-projective) if and only if it is injective (respectively, projective).
\end{lemma}

Recall that an additive category is said to be \emph{weakly idempotent complete} if every retraction in $\C$ has a kernel (or, equivalently, every section has a cokernel)~\cite[Defn.\ 7.2]{Buehler}.

\begin{remark}
\label{rem:WIC}
By \cite[Prop.~2.7]{Klapproth}, weak idempotent completeness of $\C$ is equivalent to Nakaoka and Palu's WIC condition \cite[Cond.~5.8]{NakaokaPalu}:
for morphisms $f\colon A\to B$ and $g\colon B\to C$ in $\C$, if $gf$ is a deflation then $g$ is a deflation, and if $gf$ is an inflation then $f$ is an inflation.
\end{remark}

\begin{lemma}
\label{l:Mperpwic}
Assume $\C$ is weakly idempotent complete, and let $\X\subseteq\C$ be a subcategory.
Then $\rperp{\X}$ and $\lperp{\X}$ are also weakly idempotent complete.
\end{lemma}
\begin{proof}
We must show that the kernel in $\C$ of any retraction in $\rperp{\X}$ lies in $\rperp{\X}$.
Let $f\colon Y\rightarrow Z$ be a retraction in $\rperp{\X}$. Since $\rperp{\X}$ is a full subcategory, any retraction in $\rperp{\X}$ is also a retraction in $\C$, so $f$ has a kernel in $\C$ 
and we have a conflation
\begin{equation}
\label{eq:kernelretraction}
\begin{tikzcd}
 K \arrow[infl]{r} & Y \arrow[defl]{r}{f} & Z \arrow[confl]{r}&\phantom{},
\end{tikzcd}
\end{equation}
which is split because $f$ is a retraction.
Let $X\in\X$.
Since $\Ext^1_{\C}(X,Y)=0$ and $K$ is a direct summand of $Y$, it follows that $\Ext^1_{\C}(X,Y)=0$, so $K$ lies in $\rperp{\X}$ as required.
The proof for $\lperp{\X}$ is analogous.
\end{proof}

Recall that a map $p\colon X\rightarrow Y$ is said to be \emph{right minimal} if, whenever $\varphi\in \End(X)$ and
$p\varphi=p$, then $\varphi$ is an isomorphism.
If $\X$ is a subcategory, then a morphism $f\colon X\rightarrow A$
is a \emph{right $\X$-approximation} of $A\in\C$ if
$X$ lies in $\X$, and for any $X'$ in $\X$ and $g\colon X'\rightarrow
A$ there is $h\colon X'\rightarrow X$ such that $fh=g$.
Left minimal maps and approximations are defined dually.
If every object in the category has a right $\X$-approximation (respectively, left approximation) then $\X$ is said to be \emph{contravariantly} (respectively, \emph{covariantly}) \emph{finite}.
If $\X$ is both contravariantly and covariantly finite, it is said to be \emph{functorially finite}.

If, for every object $A\in\C$, there is a right $\X$-approximation $f\colon X\to A$ which is also a deflation, then $\X$ is said to be \emph{strongly contravariantly finite}. The notions of strongly covariantly finite and strongly functorially finite are defined analogously~\cite[Defn.\ 3.19]{ZZ18}.

\begin{remark}
\label{r:weak-to-strong}
Recall that $\C$ is said to have enough projectives if, for each object $X$ in $\C$, there is a deflation $P\defl X$ where $P$ is projective.
If $\C$ has enough projectives, a contravariantly finite subcategory $\X$ containing all projective objects in $\C$ is automatically strongly contravariantly finite \cite[Rem.~3.22]{HLN2} (see also \cite[Rem.~2.9]{ZZ19}).
If additionally $\C$ is weakly idempotent complete then even more is true: every right $\X$-approximation is a deflation \cite[Rem.~5.5]{FGPPP}.
\end{remark}

\begin{lemma}
\label{l:weakwakamatsu-extri}
Let $\C$ be an extriangulated category, 
let
\[\begin{tikzcd}
K\arrow[infl]{r}{f}&X\arrow[defl]{r}{g}&A\arrow[confl]{r}&\phantom{}
\end{tikzcd}\]
be a conflation in $\C$, and let $\X$ be a rigid subcategory of $\C$.
If $g$ is a right $\X$-approximation of $A$ (so that, in particular, $X\in \X$), then $K\in\X^{\perp_1}$.
\end{lemma}
\begin{proof}
The argument is standard but we include it for completeness. Let $X'$ be an object in $\X$. By~\cite[Prop.\ 3.3]{NakaokaPalu}, there is an exact sequence
\[\begin{tikzcd}
\Hom_{\C}(X',X) \arrow{r}{f_*} &
\Hom_{\C}(X',A)\arrow{r}{} &
\EE_{\C}(X',K) \arrow{r}{} &
\EE_{\C}(X',X). \end{tikzcd}\]
Since $g$ is a right $\X$-approximation, $f_*$ is surjective.
Since $\EE_{\C}(X',X)=0$ because $\X$ is rigid, we have $\EE_{\C}(X',K)=0$, giving the desired result.
\end{proof}

\begin{remark}
By (the dual of) \cite[Lem.~3.1]{LiuZhou}, the conclusion of Lemma~\ref{l:weakwakamatsu-extri} still holds if we replace the assumption that $\X$ is rigid by the assumption that $g$ is a minimal approximation; this is the extriangulated version of Wakamatsu's lemma.
\end{remark}

We denote by $\PR$ the subcategory of $\C$ consisting of the projective objects and by $\add(\X,\PR)$ the smallest full additive subcategory of $\C$ containing $\X$ and $\PR$.
In particular, if $\X$ is additively closed and contains $\PR$, then $\add(\X,\PR)=\X$.

\begin{lemma} \label{l:MPRstrongly}
Suppose that $\C$ has enough projective objects, and let $\X\subseteq\C$ be contravariantly finite. Then $\add(\X,\PR)$ is strongly contravariantly finite.
\end{lemma}
\begin{proof}
Let $M\in \C$, and let $f\colon X\rightarrow M$ be a right $\X$-approximation of $M$. Since $\C$ has enough projectives, there is a right $\PR$-approximation
$g\colon P\rightarrow M$ (which is moreover a deflation).
Then $h=(f,g)\colon X\oplus P\rightarrow M$ is a right $\add(\X,\PR)$-approximation of $M$, and so $\add(\X,\PR)$ is covariantly finite. Since it contains $\PR$, it is also strongly contravariantly finite by \cite[Rem.~3.22]{HLN2} (see Remark~\ref{r:weak-to-strong}).
Indeed, the argument given in loc.~cit.\ demonstrates that $h$ is itself a deflation.
\end{proof}

We denote the full subcategory of injective objects in $\C$ by 
$\Ical$.
The dual to Lemma~\ref{l:MPRstrongly} is as follows.

\begin{lemma} \label{l:MPRstronglydual}
Suppose that $\C$ has enough injective objects and let $\X\subseteq\C$ be covariantly finite. Then $\add(M,\Ical)$ is strongly covariantly finite.
\end{lemma}

The following is a version of the argument in~\cite[Lem.~2.2]{buanmarsh}, adapted to extriangulated categories.

\begin{lemma} \label{l:perp-extri}
Suppose that $\C$ is has enough projective objects, and let $\X\subseteq\C$ be contravariantly finite and rigid. Then $\lperp{(\rperp{\X})}=\add(\X,\PR)$.
\end{lemma}
\begin{proof}
If $M\in\add(\X,\PR)$ and $N\in \rperp{\X}$ then $\EE_\C(M,N)=0$, so $\X\subseteq {}^{\perp_1}(\rperp{\X})$.
Hence $\add(\X,\PR)$ is contained in ${}^{\perp_1}(\rperp{\X})$.

Assume $M\in {}^{\perp_1}(\rperp{\X})$.
By Lemma~\ref{l:MPRstrongly}, there is a conflation
\[\begin{tikzcd}
K\arrow[infl]{r}& X\arrow[defl]{r}{f}& M\arrow[confl]{r}&\phantom{}
\end{tikzcd}\]
in which $f$ is a right $\add(\X,\PR)$-approximation.
Since $\X$ is rigid, it follows from Lemma~\ref{l:weakwakamatsu-extri} that $K\in \rperp{\X}$.
Since $M\in {}^{\perp_1}(\rperp{\X})$,
we have $\EE_\C(M,K)=0$, so the sequence splits and
$M\in \add(\X,\PR)$.
\end{proof}

We have a dual version of Lemma~\ref{l:perp-extri}, with a dual proof, as follows.

\begin{lemma}
\label{l:perp-dual}
Suppose that $\C$ has enough injective objects, and let $\X\subseteq\C$ be covariantly finite and rigid. Then
$({}^{\perp_1}\X)^{\perp_1}=\add(\X,\Ical)$.
\end{lemma}

\begin{lemma} \label{l:enough-extri}
Suppose that $\C$ has enough projectives, and let $\X\subseteq\C$ be rigid and contravariantly finite. Then the subcategory of projective objects in $\rperp{\X}$ is $\add(\X,\PR)$.

Dually, if $\C$ has enough injectives and $\X\subseteq\C$ is rigid and covariantly finite, then the subcategory of injective objects in $\lperp{\X}$ is $\add(\X,\Ical)$.
\end{lemma}
\begin{proof}
The Ext-projective objects in $\rperp{\X}$ (which coincide with the projective objects by Lemma~\ref{l:proj-extri}) are the objects of $\rperp{\X}$ which also lie in
${}^{\perp_1}(\rperp{\X})$.
By Lemma~\ref{l:perp-extri}, we have ${}^{\perp_1}(\rperp{\X})=\add(\X,\PR)$. Since $\X$ is rigid, we also have $\add(\X,\PR)\subseteq\rperp{\X}$, giving the first statement.
The dual is proved similarly, using Lemma~\ref{l:perp-dual} in place of Lemma~\ref{l:perp-extri}.
\end{proof}

\begin{lemma}
\label{l:Mperpff}
Suppose that $\C$ is Frobenius. Let $\X\subseteq\C$ be rigid and functorially finite, and assume that $\rperp{\X}=\lperp{\X}$.
Then $\X^{\perp_1}$ is functorially finite in $\C$.
\end{lemma}
\begin{proof}
We use the same approach as~\cite[Thm.~II.2.1(a)]{BIRS09}.
By Lemma~\ref{l:MPRstrongly}, the subcategory $\add(\X,\PR)$ is strongly contravariantly finite in $\C$.
Hence, by \cite[Prop.~3.4]{CZZ19}, we have that $(\add(\X, \PR),\add(\X,\PR)^{\perp_1})=
(\add(\X, \PR),\rperp{\X})$ is a cotorsion pair in $\C$.
Note that the assumption that $\X$ is rigid is sufficient for the proof of~\cite[Prop.~3.4]{CZZ19}, since it means that we can apply Lemma~\ref{l:weakwakamatsu-extri} in place of Wakamatsu's lemma: this means we do not need to refer to minimal approximations, and hence do not require $\C$ to be Krull--Schmidt.

It thus follows that $\rperp{\X}$ is covariantly finite in $\C$ by~\cite[Rem.\ 3.2]{CZZ19}.
Since $\rperp{\X}=\lperp{\X}=\prescript{\perp_1}{}{\smash{\add(\X,\Ical)}}$, Lemma~\ref{l:MPRstronglydual}
and the dual of~\cite[Prop.~3.4]{CZZ19} imply that $\rperp{\X}$ 
is also contravariantly finite in $\C$, and hence 
functorially finite in $\C$.
\end{proof}

\begin{remark}
\label{r:stronglyfinite}
In the situation of Lemma~\ref{l:Mperpff}, $\rperp{\X}$ is strongly functorially finite by \cite[Rem.~3.22]{HLN2}, as in Remark~\ref{r:weak-to-strong}, since $\rperp{\X}$ contains all of the projective-injective objects of $\C$.
\end{remark}

\begin{proposition} \label{Prop:MperpisFrobenius-extri}
Suppose that $\C$ is Frobenius, and let $\X\subseteq\C$ be a rigid and functorially finite subcategory such that $\rperp{\X}=\lperp{\X}$.
Then $\rperp{X}$ is also a Frobenius extriangulated category, with projective-injectives $\add(\X,\PR)$.
If $\C$ is weakly idempotent complete, then so is $\rperp{\X}$.
\end{proposition}
\begin{proof}
The subcategory $\rperp{\X}$ inherits an extriangulated structure from $\C$ by Lemma~\ref{l:closed-extri}.
Since $\C$ is Frobenius, we have $\PR=\Ical$.
Thus the projective and injective objects in $\rperp{\X}$ are given by $\add(\X,\PR)=\add(\X,\Ical)$, by Lemma~\ref{l:enough-extri}, and hence coincide.
It remains to show that there are enough projectives and injectives in $\rperp{\X}$.

By Lemmas~\ref{l:MPRstrongly} and~\ref{l:MPRstronglydual}, the subcategory $\add(\X,\PR)$ is strongly functorially finite. Let $M\in \rperp{\X}$.
Then there is a conflation
\[\begin{tikzcd}
K\arrow[infl]{r}& X\arrow[defl]{r}{p}& M\arrow[confl]{r}&\phantom{}
\end{tikzcd}\]
in $\C$, in which $p$ is a right $\add(\X,\PR)$-approximation of $M$.
We have $X\in\add(\X,\PR)\subseteq\rperp{\X}$ since $\X$ is rigid and $\PR=\Ical$, and $K\in\add(\X,\PR)^{\perp_1}=\rperp{\X}$ by Lemma~\ref{l:weakwakamatsu-extri}.
Thus the whole conflation lies in $\rperp{\X}$, so in particular $p$ is a deflation in $\rperp{\X}$, and hence a projective cover of $M$.
A dual argument shows that $\lperp{\X}=\rperp{\X}$ has enough injectives.

The statement concerning weak idempotent completeness is Lemma~\ref{l:Mperpwic}.
\end{proof}

\begin{remark}
\label{r:BIRSproof}
If $\C$ is a Frobenius exact category and $\X\subseteq\C$ is functorially finite and rigid with $\rperp{\X}=\lperp{\X}$, the fact that $\rperp{\X}$ is Frobenius exact also follows from Lemma~\ref{l:Mperpff} and~\cite[Thm.\ II.2.6]{BIRS09}.
We'd like to thank Yann Palu for helpful remarks concerning this.
\end{remark}

\begin{definition}
\label{d:reduction}
Given a Frobenius extriangulated category $\C$ and a rigid functorially finite subcategory $\X\subseteq\C$ such that $\rperp{\X}=\lperp{\X}$, we call $\rperp{\X}$ the \emph{reduction} of $\C$ with respect to $\X$.
\end{definition}

Let $\T$ be a $\field$-linear Hom-finite triangulated category.
Then $\T$ is said to be \emph{$2$-Calabi--Yau} if there is a functorial isomorphism $\Ext^1_\T(X,Y)\cong \dual\Ext^1_\T(Y,X)$ for all objects $X$ and $Y$ in $\T$, where $\dual$ denotes the duality $\dual=\Hom_{\field}(-,\field)$.
This is called weakly $2$-Calabi--Yau in~\cite{Keller08}.
We call a Frobenius extriangulated category $\C$ \emph{stably $2$-Calabi--Yau} if it is $\field$-linear and its stable category $\underline{\C}=\C/(\PR)$ (see Theorem~\ref{t:stabcat}) is $2$-Calabi--Yau.
Note that in a stably $2$-Calabi--Yau Frobenius extriangulated category we have $\rperp{\X}=\lperp{\X}$ for any $\X\subseteq\C$.

\begin{proposition}
\label{p:Mperp-2cy-extri}
Let $\C$ be a stably $2$-Calabi--Yau Frobenius extriangulated category, and let $\X\subseteq\C$ be rigid and functorially finite.
Then the reduction $\rperp{\X}$ is also stably $2$-Calabi--Yau.
\end{proposition}
\begin{proof}
Firstly, note that $\stabHom_{\rperp{\X}}(M,N)$ is a quotient of $\stabHom_{\C}(X,Y)$ for all objects $M,N$ in $\rperp{\X}$, since the objects of $\PR\subseteq\add(\X,\PR)\subseteq \rperp{\X}$ are projective-injective in $\rperp{\X}$ by Proposition~\ref{Prop:MperpisFrobenius-extri}.
It follows that $\underline{\rperp{\X}}$ is $\Hom$-finite.
By Lemma~\ref{l:closed-extri},
we have an equality $\EE_{\rperp{\X}}(-,-)=\EE_{\C}(-,-)|_{\rperp{\X}}$ of bifunctors.
Hence $\rperp{\X}$ is also stably $2$-Calabi--Yau.
\end{proof}

One justification for calling $\rperp{\X}$ a reduction
is the following proposition, which also explains why we restrict to rigid subcategories $\X$ in Definition~\ref{d:reduction}.
Another justification is given in Remark~\ref{r:compare-extri}.
Recall that a subcategory $\T\subseteq\C$ is called \emph{cluster-tilting} if
\[\lperp{\T}=\T=\rperp{\T}\]
and $\T$ is functorially finite.
In particular, a cluster-tilting subcategory is additively closed.
An object $T$ is cluster-tilting if $\add(T)$ is a cluster-tilting subcategory; note that $\add(T)$ is always functorially finite.

This matches the definition in \cite[Defn.\ 4.1]{CZZ19}, which is close to Iyama's original \cite[Defn.~2.2]{Iyama-HART}; in general there may be value in requiring cluster-tilting subcategories of extriangulated categories to be strongly functorially finite, but whenever the ambient category $\C$ has enough projectives and injectives this is automatic, as in Remark~\ref{r:stronglyfinite}.

\begin{proposition}
\label{p:ct-bij-extri}
Let $\rperp{\X}$ be the reduction of a Frobenius extriangulated category $\C$ at a rigid functorially finite subcategory $\X\subseteq\C$ such that $\rperp{\X}=\lperp{\X}$.
Then the cluster-tilting subcategories in $\rperp{\X}$ are precisely those cluster-tilting subcategories $\T\subseteq\C$ such that $\X\subseteq\T$.
\end{proposition}
\begin{proof}
Recall that our notation ${}^{\perp_1}\T$ and $\T^{\perp_1}$ refers to the perpendicular categories inside $\C$.
By Lemma~\ref{l:closed-extri}, when $\T\subseteq\rperp{\X}$ the analogous perpendicular categories to $T$ in $\rperp{\X}$ are simply given by ${}^{\perp_1}\T\cap\rperp{X}$ and $\T^{\perp_1}\cap\rperp{\X}$.

If $\T\subseteq\rperp{\X}$, then $\X\subseteq{}^{\perp_1}\T$.
Hence if $\T$ is cluster-tilting in $\rperp{\X}$, then $\X\subseteq{}^{\perp_1}\T\cap\rperp{\X}=\T$.
To see that $\T$ is cluster-tilting in $\C$, first observe that $\T=\T^{\perp_1}\cap\rperp{\X}\subseteq \T^{\perp_1}$, and similarly for the left perpendicular, so it remains to prove the reverse inclusions.
But if $M\in\T^{\perp_1}$, then in particular $M\in\rperp{\X}$ since $\X\subseteq\T$, and so $M\in \T^{\perp_1}\cap\rperp{\X}=\T$.
An entirely analogous argument, using that $\rperp{\X}=\lperp{\X}$, shows that ${}^{\perp_1}\T=\T$. Since $\T$ is functorially finite in $\rperp{\X}$ by definition, and $\rperp{\X}$ is functorially finite in $\C$ by Lemma~\ref{l:Mperpff}, it follows that $\T$ is also functorially finite in $\C$, and so is cluster-tilting in $\C$.

Conversely, if $\T\subseteq\C$ is cluster-tilting and $\X\subseteq\T$, then $\X\subseteq{}^{\perp_1}\T$, and hence $\T\subseteq\rperp{\X}$.
Moreover, $\rperp{\X}\subseteq \T^{\perp_1}$, and so $\T^{\perp_1}\cap\rperp{\X}=\T^{\perp_1}=\T$.
Similarly ${}^{\perp_1}\T\cap\rperp{\X}=\T$, using again that $\rperp{\X}=\lperp{\X}$.
Since $\T$ is functorially finite in $\C$, it is also functorially finite in $\rperp{\X}$, and hence it is also cluster-tilting in this subcategory.
\end{proof}

Summing up, we obtain the main theorem of this section, as follows.
\begin{theorem}
\label{t:reduction}
Let $\C$ be a stably $2$-Calabi--Yau Frobenius extriangulated category and $\X\subseteq\C$ a rigid functorially finite subcategory.
Then $\rperp{\X}$ is also a stably $2$-Calabi--Yau Frobenius extriangulated category, functorially finite in $\C$, with projective-injectives $\add(\X,\PR)$.
If $\C$ is weakly idempotent complete, so is $\rperp{\X}$.
The cluster-tilting subcategories of $\rperp{\X}$ are precisely those cluster-tilting objects $\X\subseteq\C$ such that $\X\subseteq\T$.
\end{theorem}
\begin{proof}
This follows from Lemmas~\ref{l:Mperpwic} and~\ref{l:Mperpff} and Propositions~\ref{Prop:MperpisFrobenius-extri},~\ref{p:Mperp-2cy-extri} and~\ref{p:ct-bij-extri}.
\end{proof}

Triangulated categories are a special class of Frobenius extriangulated categories (precisely, those for which $\PR=0$ \cite[Cor.~7.6]{NakaokaPalu}).
In the case that $\C$ is triangulated, the reduction $\rperp{\X}$ is exactly the subcategory $\mathcal{U}$ considered by Fu and Keller in \cite[\S6.2]{FuKeller} in the context of categorifying cluster algebras with frozen variables.
Moreover, the stable category of $\rperp{\X}$, which is well-defined and equal to $\rperp{\X}/\add(\X)$ by Proposition~\ref{Prop:MperpisFrobenius-extri}, coincides with the Iyama--Yoshino reduction of $\C$ at $\X$ \cite{IY08}.
We will explore this connection further, for more general Frobenius extriangulated categories, in the next section.

\section{Compatibility with Iyama--Yoshino reduction}
\label{s:IYred}
Assume that $\F$ is a stably $2$-Calabi--Yau Frobenius extriangulated category.
In particular, this means that $\rperp{\X}=\lperp{\X}$ for any $\X\subseteq\F$. We fix such an $\X$, assumed to be rigid and functorially finite.

By Proposition~\ref{Prop:MperpisFrobenius-extri}, the
stable category of $\rperp{\X}$ is $\underline{\rperp{\X}}=\rperp{\X}/\add(\X,\PR)$.
Since $\rperp{\X}$ is a Frobenius extriangulated category by Proposition~\ref{Prop:MperpisFrobenius-extri}, the quotient category
$\underline{\rperp{\X}}$ is naturally triangulated by \cite[Cor.~7.4]{NakaokaPalu} (see Theorem~\ref{t:stabcat}).

We set $\X_{\underline{\mc{F}}}^{\perp_1}$ to be the 
subcategory of $\underline{\mc{F}}$ consisting of objects $M$ in $\underline{\F}$ satisfying
\[\Ext^1_{\underline{\F}}(M,X):=\stabHom_{\F}(M,X[1])=0\]
for all $X\in\X$,
or equivalently, since $\F$ is stably $2$-Calabi--Yau,
satisfying $\Ext_{\underline{\F}}^1(X,M)=0$ for all $X\in\X$.
Because $\Ext^1_{\F}(M,N)= \Ext^1_{\underline{\F}}(M,N)$ for all $M,N\in\F$,
it follows that $\X_{\underline{\F}}^{\perp_1}$ is the image of
$\rperp{\X}$ under the quotient functor
$\F\rightarrow \underline{\F}$.

By \cite[Thm.~4.7]{IY08}, the category $\X_{\underline{\mc{F}}}^{\perp_1}/\add_{\underline{\F}}(\X)$ is triangulated and $2$-Calabi--Yau; this category is the \emph{Iyama--Yoshino reduction} of the triangulated category $\underline{\F}$ with respect to the functorially finite rigid subcategory $\X$.
In this section, we will show (see Theorem~\ref{thm:triangleequivalence-extri}) that there is a triangle equivalence
\[\underline{\rperp{\X}} \simeq\X_{\underline{\mc{F}}}^{\perp_1}/\add_{\underline{\F}}(\X).\]

To begin with, we will need some generalities on functors and equivalences.
The following terminology is not standard, but will provide us with helpful language. %

\begin{definition}
\label{d:tautological}
Let $\U$ be a category, and let $\A$ and $\B$ be categories whose objects and morphisms are (in preferred bijection with) equivalence classes of objects and morphisms from $\U$.
We denote these equivalence classes by $[x]_\A$ and $[x]_\B$, where $x$ is an object or morphism of $\U$.
We say a functor $F\colon\A\to\B$ is \emph{tautological} if $F[x]_{\A}=[x]_{\B}$ for any object or morphism $[x]_\A$ of $\A$.
\end{definition}

\begin{note}
In Definition~\ref{d:tautological}, we do not assume that every object or morphism $x\in\U$ has an associated equivalence class $[x]_\A$ or $[x]_\B$ in $\A$ or $\B$.
However, it is necessary for the existence of a tautological functor $\A\to\B$ that if $x\in\U$ is an object or morphism for which $[x]_\A$ exists, then $[x]_\B$ must also exist.
\end{note}

\begin{example}
If $\A\subseteq\B$ is a subcategory, then the inclusion $i\colon\A\to\B$ is tautological.
If $\B=\A/I$ for some ideal $I$ of morphisms, then the quotient functor $\pi\colon\A\to\B$ is tautological.
(If not otherwise specified, we may take $\U=\B$ in the first example, and $\U=\A$ in the second.)
\end{example}

If a tautological functor in the sense of Definition~\ref{d:tautological} exists, then it is unique: indeed, the definition completely specifies its value on all objects and morphisms of $\A$.
Similarly, if $F\colon\A\to\B$ and $G\colon\B\to\C$ are tautological, so is $G\circ F$.

\begin{proposition}
\label{p:taut-inverse}
If $F\colon\A\to\B$ and $G\colon\B\to\A$ are tautological functors, then $F$ and $G$ are inverse isomorphisms.
\end{proposition}
\begin{proof}
The identity functor $1_\A\colon\A\to\A$ is tautological, as is the composition $G\circ F$.
Hence by uniqueness, $G\circ F=1_\A$.
Similarly, $F\circ G=1_{\B}$.
(Note that we have equality of functors here, not just natural isomorphism.)
\end{proof}

\begin{proposition}
\label{p:tautological}
Let $\A$ and $\B$ be categories whose objects and morphisms are equivalence classes of objects and morphisms from some category $\U$, and let $F\colon\U\to\A$ and $G\colon\A\to\B$ be functors.
Then if $F$ and $G\circ F$ are tautological, so is $G$.
\end{proposition}
\begin{proof}
Let $[x]_\A$ be an object or morphism in $\A$, for some object or morphism $x\in\U$.
Since $F$ is tautological (and $[x]_\U=\{x\}$ is identified with $x$ for all $x\in\U$), we must have $[x]_\A=Fx$, and since $G\circ F$ is tautological we must have $[x]_\B=(G\circ F)x=G[x]_\A$.
Thus $G$ is tautological.
\end{proof}

Next, we recall the construction of the quotient of an additive category by an additive subcategory and its universal properties; we also show that a composition of two such quotients (under a certain assumption) is isomorphic to a single quotient.

\begin{proposition}
\label{prop:additivequotient-extri}
Let $\A$ be an additive category and $\C$ a full additive subcategory of $\A$.
Then there is an additive category $\A/\C$ with the following properties:
\begin{enumerate}
\item the objects of $\A/\C$ are the same as the objects of $\A$;
\item there is a full additive tautological functor $F\colon\A\rightarrow \A/\C$;
\item for all objects $C$ in $\C$, we have $F(C)\cong 0$;
\item if $\B$ is an additive category and $G\colon\A\rightarrow \B$ is an additive functor with the property that $G(C)\cong 0$ for all objects $C$ in $\C$, then there is a unique additive functor $F'\colon\A/\C\rightarrow \B$ such that $F'F=G$.
\end{enumerate}
\end{proposition}

\begin{note}
To define $\A/\C$, we take the objects of $\A/\C$ to be the same as the objects of $\A$, and define $\Hom_{\A/\C}(X,Y)=\Hom_\A(X,Y)/L$, where $L$ is the subspace of $\Hom_\A(X,Y)$ consisting of maps which factor through an object in $\C$.
The functor $F$ is the identity on objects, and the quotient map $\Hom_{\A}(X,Y)\to\Hom_{\A}(X,Y)/L$ for each pair of objects $X,Y\in\A$.

Since $\A$ and $\A/\C$ have the same objects, any full subcategory $\D$ of $\A$ determines a full subcategory of $\A/\C$ with the same objects as $\D$ (although with different morphisms).
Starting with the next lemma, we will reuse the notation $\D$ for this second subcategory, with the context making it clear which category we are viewing it is a subcategory of.
\end{note}

\begin{lemma}
\label{lem:doublequotient-extri}
Let $\A$ be an additive category, let $\C$ and $\D$ be full additive subcategories of $\A$, and let $\E=\add(\C,\D)$.
Then there is an isomorphism $R\colon \A/\E\rightarrow (\A/\C)/\D$, with inverse $S$, such that $R$ is tautological (for $\U=\A$) and the following diagram commutes:
\begin{equation}
\label{e:dqdoublediagram-extri}
\begin{tikzcd}
\mc{A}\arrow[r,equal] \arrow{dd}{F} &\mc{A}\arrow[r,equal] \arrow{d}{G}&\mc{A} \arrow{dd}{F} \\
& \mc{A}/\mc{C} \arrow{d}{H}& \\
\mc{A}/\mc{E} \arrow{r}{R} & (\A / \C)/ \D \arrow{r}{S} & \A / \E 
\end{tikzcd}
\end{equation}

\end{lemma}
\begin{proof}
Let $F\colon\A\rightarrow \A/\E$, $G\colon\A\rightarrow \A/\C$ and $H\colon\A/\C\rightarrow (\A/\C)/\D$ be the quotient functors given by Proposition~\ref{prop:additivequotient-extri}; all of these functors are tautological.

If $E\in\E$, then $G(E)\in\D\subseteq\A/\C$, and so $HG(E)\cong0$ in $(\A/\C)/\D$. Thus there is a unique functor $R\colon\C/\E\rightarrow (\A/\C)/\D$ such that $RF=HG$ \eqref{e:dqlefthandsquare-extri}. By Proposition~\ref{p:tautological}, this formula implies that $R$ is tautological.
\begin{equation}
\label{e:dqlefthandsquare-extri}
\begin{tikzcd}
\mc{A}\arrow[r,equal] \arrow{dd}{F} &\mc{A} \arrow{d}{G}&\\
& \mc{A}/\mc{C} \arrow{d}{H} \\
\mc{A}/\mc{E} \arrow[r,exists]{}{R} & (\A / \C)/ \D
\end{tikzcd}
\end{equation}

If $C\in\C\subseteq\E$, then $F(C)\cong0$ in $\A/\E$, so there is a unique functor $S'\colon\A/\C\rightarrow \A/\E$ such that $F=S'G$, and $S'$ is tautological by Proposition~\ref{p:tautological}.
Similarly, if $D\in\D\subseteq\E$ then $0\cong F(D)=S'G(D)=S'(D)$. Hence there is a unique functor $S\colon(\A/\C)/\D\rightarrow \A/\E$ such that $SH=S'$ \eqref{e:dqrighthandsquare-extri}. Thus $SHG=S'G=F$ and $S$ is tautological by Proposition~\ref{p:tautological} again.

\begin{equation}
\label{e:dqrighthandsquare-extri}
\begin{tikzcd}
\mc{A}\arrow[r,equal] \arrow{d}{G} &\mc{A}\arrow{dd}{F}  \\
\mc{A}/\mc{C} \arrow{d}{H} \arrow[rd,exists]{}{S'} &  \\
(\A / \C)/ \D   \arrow[r,exists]{}{S} &  \mc{A}/\mc{E}
\end{tikzcd}
\end{equation}

Putting the diagrams~\eqref{e:dqlefthandsquare-extri} and~\eqref{e:dqrighthandsquare-extri} together, we have the following commutative diagram, which is part of what we needed to prove.

\begin{equation}
\label{e:dqoverall-extri}
\begin{tikzcd}
\A \arrow[r,equal] \arrow{dd}{F} & \mc{A}\arrow[r,equal] \arrow{d}{G} &\mc{A}\arrow{dd}{F}
\\
& \mc{A}/\mc{C} \arrow{d}{H} \arrow[rd,exists]{}{S'} & 
\\
 \mc{A}/\mc{E} \arrow[r,exists]{}{R} & (\A / \C)/ \D   \arrow[r,exists]{}{S} &  \mc{A}/\mc{E}  
\end{tikzcd}
\end{equation}
Since $R$ and $S$ are both tautological, they are inverse isomorphisms by Proposition~\ref{p:taut-inverse}.
\end{proof}

We next recall the triangulated structure on the stable category of a Frobenius extriangulated category $\F$.
Let $\PR$ be the full subcategory of projective-injective objects of $\F$, and let $\pi_{\F}$ be the quotient functor $\pi_{\F}\colon\F\rightarrow \underline{\F}=\F/\PR$.
Recall that the objects of $\F$ are the same as the objects of $\underline{\F}$ and so $\pi_{\F}$, being tautological, is the identity on objects.
Despite this, it will sometimes be useful to write $\pi_{\F}X$ in place of $X$ to emphasise that we are viewing this object in the category $\underline{\F}$.

By~\cite[Cor.~7.4]{NakaokaPalu} (cf.~\cite[Thm.~I.2.6]{Happel}), the stable category $\underline{\F}$ is a triangulated category.
We make this construction explicit, following the approach in~\cite[\S 3.3]{krause22}.
Firstly, for each object $X$ in $\F$, we fix an inflation $\alpha^{\F}_X\colon X\to I_{\F}(X)$, where $I_{\F}(X)$ is an injective object in $\F$, and let $\bSigma_{\F}(X)$ be the cone of $\alpha^{\F}_X$, so that there is a conflation
\begin{equation}
\label{eq:inj-conf}
\begin{tikzcd}
E_{\F}(X)\colon\quad X \arrow[r,infl]{}{\alpha^{\F}_X} & I_{\F}(X) \arrow[r,defl]{}{\beta^{\F}_X} & \bSigma_{\F}(X) \arrow[r,confl]{}{\delta_X^\F} & \phantom{}
 \end{tikzcd}
\end{equation}
in $\F$. Moreover, given a morphism $f\colon X\to Y$ we may choose a map of conflations
\begin{equation}
\label{eq:induced-morph}
\begin{tikzcd}
X\arrow[infl]{r}{\alpha_X^{\F}}\arrow{d}{f}&I_{\F}(X)\arrow[defl]{r}{\beta_X^{\F}}\arrow{d}{}&\bSigma_{\F}(X)\arrow{d}{\bSigma_{\F}(f)}\arrow[confl]{r}{\delta_X^{\F}}&\phantom{}\\
Y\arrow[infl]{r}{\alpha_Y^{\F}}&I_{\F}(Y)\arrow[defl]{r}{\beta_Y^{\F}}&\bSigma_{\F}(Y)\arrow[confl]{r}{\delta_Y^{\F}}&\phantom{}
\end{tikzcd}
\end{equation}
from $E_{\F}(X)$ to $E_{\F}(Y)$, and so in particular choose a morphism $\bSigma_{\F}(f)\colon\bSigma_{\F}(X)\to\bSigma_{\F}(Y)$.
Given $f$, there exists a map $I_{\F}(X)\to I_{\F}(Y)$ making the left-hand square of \eqref{eq:induced-morph} commute, since $I_{\F}(Y)$ is injective and $\alpha_X^{\F}$ is an inflation, although this map is not unique.
Having chosen this map, we may complete it, again non-uniquely, to the morphism of conflations \eqref{eq:induced-morph} as in \cite[Defn.~2.12(ET3)]{NakaokaPalu}.

\begin{remark}
\label{r:comm-conf}
The fact that \eqref{eq:induced-morph} is a morphism of conflations means not only that the two squares commute, but also that $f_*\delta_X^\F= \bSigma_\F(f)^*\delta_Y^\F$, where $f_*\colon\EE_\F(\bSigma_\F(X),X)\to\EE_\F(\bSigma_\F(X),Y)$ and $\Sigma_\F(f)^*\colon\EE_\F(\bSigma_\F(Y),Y)\to\EE_\F(\bSigma_\F(X),Y)$ are, respectively, the pushout and pullback morphisms; see \cite[Defn.~2.3]{NakaokaPalu}.
\end{remark}

As a consequence of \cite[Cor.~3.5]{NakaokaPalu}, having fixed the conflations $E_{\F}(X)$ for each object $X\in\F$, the morphisms $\pi_\F\bSigma_F(f)$ in $\underline{\F}$ are independent of the choices involved in constructing the diagram \eqref{eq:induced-morph}.
Since $\pi_{\F}$ is surjective on objects and morphisms, we thus obtain a functor $\Sigma_{\F}\colon\underline{\F}\to\underline{\F}$ by defining $\Sigma_{\F}(\pi_{\F}X)=\pi_{\F}\bSigma_{\F}(X)$ and $\Sigma_{\F}(\pi_{\F}f)=\pi_{\F}\bSigma_{\F}(f)$.
This functor will be the suspension in the triangulated structure on $\underline{\F}$.

To define the distinguished triangles, let $E$ be an arbitrary conflation
\[%
\begin{tikzcd} E\colon\quad X \arrow[r, infl]{}{\alpha} & Y \arrow[r,defl]{}{\beta} & Z \arrow[r,confl]{}{\delta} & \phantom{} \end{tikzcd}\]
in $\F$.
Just as for \eqref{eq:induced-morph}, the fact that $I_{\F}(X)$ is injective and $\alpha$ is an inflation means that we may choose a map of conflations
\begin{equation}
\label{eq:induced-tri}
\begin{tikzcd}
X \arrow[r,infl]{}{\alpha} \arrow[d,equal] & Y \arrow[r,defl]{}{\beta} \arrow[d] & Z \arrow[r,confl]{}{\delta} \arrow[d]{}{\gamma} & \phantom{} \\
X \arrow[r,infl]{}{\alpha^{\F}_X} & I_{\F}(X) \arrow[r,defl]{}{\beta^{\F}_X} & \bSigma_{\F}(X) \arrow[r,confl]{}{\delta_X^{\F}} & \phantom{}
\end{tikzcd}
\end{equation}
from $E$ to $E_{\F}(X)$.
Then the sextuple
\begin{equation}
\label{e:standardtriangles-extri}
\begin{tikzcd} \pi_{\F}X \arrow[r]{}{\pi_{\F}\alpha} & \pi_{\F}Y \arrow[r]{}{\pi_{\F}\beta} & \pi_{\F}Z \arrow[r]{}{\pi_{\F}\gamma} & \pi_{\F}\bSigma_{\F}(X)=\Sigma_{\F}(\pi_{\F}X) \end{tikzcd}
\end{equation}
in $\underline{\F}$ is, up to isomorphism, independent of the choice of diagram \eqref{eq:induced-tri}, and we call it a \emph{standard triangle} in $\underline{\F}$.
 
The following theorem may then be proved using the above arguments, and following precisely the strategy of \cite[Thm.~I.2.6]{Happel} or \cite[Prop.~3.3.2]{krause22}, which cover the case that $\F$ is exact (cf.~\cite[Rem.~7.5]{NakaokaPalu} and \cite[Defn.-Prop.~1.25]{INP}).
\begin{theorem}
\label{thm:stabletriangulated-extri}
Let $\F$ be a Frobenius extriangulated category. Then the functor $\Sigma_{\underline{\F}}$ is an autoequivalence and $\underline{\F}$, together with this autoequivalence and the class of sextuples isomorphic in $\underline{\F}$ to the standard triangles~\eqref{e:standardtriangles-extri}, is a triangulated category.
Moreover, for any $X,Y\in\F$, we have
\[\EE_\F(X,Y)=\stabHom_{\F}(X,\Sigma_{\F}Y).\]
\end{theorem}

Comparing to \cite[Thm.~4.7]{IY08}, we see that this is precisely the strategy used by Iyama and Yoshino to exhibit a triangulated structure on the category $\X_{\underline{\F}}/\add(\X)$.
Indeed, $\rperp{\X}_{\underline{\F}}$ is a Frobenius extriangulated category with projective-injectives $\add(\X)$, as in Proposition~\ref{Prop:MperpisFrobenius-extri}, so $\rperp{\X}_{\underline{\F}}/\add(\X)$ is its stable category.
Another special case of Theorem~\ref{thm:stabletriangulated-extri}, with the same proof strategy as the general case, is given by Msapato \cite[Thm.~3.3 (proof of reverse direction)]{Msapato}.
Theorem~\ref{thm:stabletriangulated-extri} can also be deduced from the following more general construction by Nakaoka and Palu, combined with \cite[Cor.~7.6]{NakaokaPalu}.

\begin{remark}
\label{r:stabcat-well-def}
If for each $X\in\F$ we choose a second conflation
\[\begin{tikzcd}
E'_{\F}(X)\colon\quad X \arrow[r,infl]{}{\alpha'_X} & I'_{\F}(X) \arrow[r,defl]{}{\beta'_X} & \bSigma'_{\F}(X) \arrow[r,confl]{}{\delta'_X} & \phantom{}
\end{tikzcd}\]
in which $I'_{\F}(X)$ is injective, then for each $X$ there exists a map of conflations
\[\begin{tikzcd}
X \arrow[r,infl]{}{\alpha^{\F}_X}\arrow[equal]{d} & I_{\F}(X) \arrow[r,defl]{}{\beta^{\F}_X}\arrow{d} & \bSigma_{\F}(X) \arrow[r,confl]{}{\delta_X^\F}\arrow{d}{\sigma_X} & \phantom{}\\
X\arrow[infl]{r}{\alpha'_X}&I_\F'(X)\arrow[defl]{r}{\beta'_X}&\bSigma'_\F(X)\arrow[confl]{r}{\delta'_X}&\phantom{}
 \end{tikzcd}\]
from $E_\F(X)$ to $E'_{\F}(X)$.
Using \cite[Cor.~3.5]{NakaokaPalu} again, one may check that the morphism $\pi\sigma_X$ in $\underline{\F}$ is independent of the choice of this map of conflations, that $\pi\sigma_X$ is an isomorphism, and that these isomorphisms collectively define a natural isomorphism $\pi\sigma\colon\Sigma_\F\to\Sigma'_\F$, where $\Sigma'_\F=\pi_\F\bSigma'_\F$ is the endofunctor of $\underline{\F}$ obtained from the conflations $E'_\F(X)$
The standard triangles obtained from the conflations $E'_{\F}(X)$ are isomorphic to those obtained from $E_\F(X)$, and so lead to the same collection of distinguished triangles in $\underline{\F}$.
As a result, we see that the two triangulated structures on $\underline{\F}$ obtained from the two sets of chosen conflations are equivalent, via the identity functor on $\underline{\F}$ and the natural isomorphism $\pi\sigma\colon\Sigma_\F\to\Sigma'_\F$.
\end{remark}

\begin{proposition}[{\cite[Prop.~3.30]{NakaokaPalu}}]
\label{p:partial-stab}
Let $\C$ be an extriangulated category and let $\PR_0$ be a class of projective-injective objects in $\E$, closed under direct sums and summands.
Then there is a natural extriangulated structure on $\C/\PR_0$ such that $\EE_{\C/\PR_0}(\pi_{\PR_0}{-},\pi_{\PR_0}{-})=\EE_{\C}(-,-)$ and $\s_{\C/\PR_0}=\pi_{\PR_0}\s_{\C}$, for $\pi_{\PR_0}\colon\C\to\C/\PR_0$ the projection functor.
\end{proposition}

\begin{remark}
\label{r:extri-functor}
It follows directly from the construction that the projection functor $\pi_{\PR_0}\colon\C\to\C/\PR_0$, together with the identity natural transformation $\EE_{\C}(-,-)\to\EE_{\C/\PR_0}(\pi_{\PR_0}{-},\pi_{\PR_0}{-})$, becomes an extriangulated functor in the sense of Bennett-Tennenhaus and Shah \cite[Defn.~2.32]{BTShah} (see also \cite[Defn.~3.15]{BTHSS}) when $\C/\PR_0$ is equipped with the extriangulated structure from Proposition~\ref{p:partial-stab}.
Indeed, the extension functor $\EE_{\C/\PR_0}(-,-)$ and realisation $\s_{\C/\PR_0}$ are defined in they only way they can be if this statement is to be true, and the content of \cite[Prop.~3.30]{NakaokaPalu} is in checking that they do in fact form part of an extriangulated structure on $\C/\PR_0$.
\end{remark}

\begin{proposition}
\label{p:extri-equiv}
Let $\F$ be a Frobenius extriangulated category and let $\X\subseteq\F$. The quotient functor $\pi\colon\F\to\underline{\F}=\F/\PR$ induces a tautological equivalence of extriangulated categories $\underline{\pi}\colon\rperp{\X}/\PR\overset{\sim}{\to}\rperp{\X}_{\underline{\F}}$.
\end{proposition}
\begin{proof}
Let $\underline{\pi}$ be the unique functor making the diagram
\begin{equation}
\label{eq:induct}
\begin{tikzcd}\rperp{\X}\arrow{r}\arrow{d}&\F\arrow{d}{\pi}\\
\rperp{\X}/\PR\arrow{r}{\underline{\pi}}&\underline{\F}\end{tikzcd}
\end{equation}
commute.
Here the unlabelled arrows are the inclusion $\rperp{\X}\to\F$ and the projection $\rperp{\X}\to\rperp{\X}/\PR$.
One can check directly that $\underline{\pi}$ is tautological (or apply Proposition~\ref{p:tautological} with $\U=\rperp{\X}$, $\A=\rperp{\X}/\PR$ and $\B=\pi(\rperp{\X})\subset\underline{\F}$, taking $F$ to be the canonical projection and $G=\underline{\pi}$).

By the construction of the stable category, we have
\begin{equation}
\label{eq:stable-exts}
\EE_{\underline{\F}}(X,Y):=\stabHom_{\F}(X,\Sigma_{\underline{\F}}Y)=\EE_{\F}(X,Y)
\end{equation}
for any $X,Y\in\F$. 
Thus the essential image of $\M$ under $\pi$ is precisely $\M_{\underline{\F}}$.
The kernel of $\pi$ is $\PR\subseteq\M$, and so $\underline{\pi}\colon\M/\PR\to\M_{\underline{\F}}$ is an equivalence of additive categories.

For any $X,Y\in\M$, we have
\[\EE_{\M/\PR}(X,Y)=\EE_{\M}(X,Y):=\EE_\F(X,Y)\]
by Proposition~\ref{p:partial-stab}, and
\[\EE_{\M_{\underline{\F}}}(\underline{\pi}X,\underline{\pi}Y)=\EE_{\M_{\underline{\F}}}(X,Y)=\EE_{\underline{\F}}(X,Y)=\EE_{\F}(X,Y)\]
as in \eqref{eq:stable-exts}.
Thus we have identity maps $\Gamma_{X,Y}\colon\EE_{\M/\PR}(X,Y)\to\EE_{\M_{\underline{\F}}}(\underline{\pi}X,\underline{\pi}Y)$
for each $X,Y\in\M$, and these define a natural transformation \[\Gamma\colon\EE_{\M/\PR}(-,-)\to\EE_{\M_{\underline{\F}}}(\underline{\pi}{-},\underline{\pi}{-}),\]
using the fact that $\underline{\pi}$ is tautological to check commutativity of the necessary diagrams.

Finally, we show that $\Gamma$ is compatible with the realisation maps, i.e. \[\underline{\pi}\circ\s_{\M/\PR}=\s_{\M_{\underline{\F}}}\circ\Gamma_{X,Y}=\s_{\M_{\underline{\F}}}\]
for all $X,Y\in\M$.
Let $X,Y\in\M$ and $\delta\in\EE_{\M/\PR}(X,Y)=\EE_\F(X,Y)=\EE_{\M_{\underline{\F}}}(X,Y)$. Then $\s_{\F}(\delta)$ is a conflation in $\F$ with all terms in $\M$, since $\M$ is extension-closed.
To obtain $\s_{\M_{\underline{\F}}}(\delta)$, we project this conflation to $\underline{\F}$ via $\pi$.
On the other hand, to obtain $\s_{\M/\PR}$ we apply the projection $\M\to\M/\PR$.
It then follows from the commutativity of \eqref{eq:induct} that $\underline{\pi}(\s_{\M/\PR}(\delta))=\s_{\M_{\underline{\F}}}(\delta)$, as required.

Since $\underline{\pi}\colon\M/\PR\to\M_{\underline{\F}}$ is an equivalence of additive categories and $\Gamma$ is a natural isomorphism of functors, the pair is an equivalence of extriangulated categories by \cite[Prop.~2.13]{NOS}.
\end{proof}

\begin{lemma}
\label{l:stab-fun}
Let $\F$ and $\F'$ be Frobenius extriangulated categories, and let $\varphi\colon\F\to\F'$ be an extriangulated functor taking projective-injective objects of $\F$ to projective-injective objects of $\F'$.
Consider the diagram
\begin{equation}
\label{eq:induced-tri-fun}
\begin{tikzcd}
\F\arrow{r}{\varphi}\arrow{d}[swap]{\pi}&\F'\arrow{d}{\pi'}\\
\underline\F\arrow{r}{\underline\varphi}&\underline{\F}',
\end{tikzcd}
\end{equation}
in which the vertical arrows are the projections, and $\underline{\varphi}$ is the unique functor making the diagram commutative.
Then there is a natural isomorphism $\gamma\colon\underline{\varphi}\Sigma_{\F}\to\Sigma_{\F'}\underline{\varphi}$ together with which $\underline{\varphi}$ is a triangle functor.
\end{lemma}

\begin{proof}
Since $\varphi$ is assumed to take projective-injective objects in $\F$ to those in $\F'$, any morphism in $\F$ factoring over a projective-injective object is in the kernel of $\pi'\varphi$, whence the existence and uniqueness of the functor $\underline{\varphi}$.

Recall from Remark~\ref{r:stabcat-well-def} that the suspension $\Sigma_{\F'}$ is only well-defined up to natural isomorphism.
In practice, we will show that there is a particular representative of this natural isomorphism class (depending on $\varphi$) for which we actually have an equality $\underline{\varphi}\Sigma_\F=\Sigma_{\F'}\underline{\varphi}$.

To make the triangulated structure on $\underline{\F}$ explicit, choose all of the necessary conflations \eqref{eq:inj-conf} and diagrams \eqref{eq:induced-morph} and \eqref{eq:induced-tri} in $\F$.
Then applying the extriangulated functor $\varphi$, with associated natural transformation $\Gamma\colon\EE_{\F}(-,-)\to\EE_{\F'}(\varphi{-},\varphi{-})$, to the conflation $E_\F(X)$ yields a conflation
\[\begin{tikzcd}
E_{\F'}(\varphi X)=\varphi E_{\F}(X)\colon\quad \varphi X \arrow[r,infl]{}{\varphi\alpha^{\F}_X} & \varphi I_{\F}(X) \arrow[r,defl]{}{\varphi\beta^{\F}_X} & \varphi\bSigma_{\F}(X) \arrow[r,confl]{}{\Gamma\delta_X^\F} & \phantom{}
 \end{tikzcd}\]
in $\F'$ in which, by our assumption on $\varphi$, the object $\varphi I_\F(X)$ is projective-injective.
For any object $Z$ which does not lie in the image of $\varphi$, we choose a conflation $E_{\F'}(Z)$ arbitrarily.

Similarly, applying $\varphi$ to the diagram \eqref{eq:induced-morph} yields the map of conflations
\begin{equation}
\label{eq:induced-morph-part}
\begin{tikzcd}[column sep=3.5pc]
\varphi X\arrow[infl]{r}{\varphi\alpha_X^{\F}}\arrow{d}{\varphi f}&\varphi I_{\F}(X)\arrow[defl]{r}{\varphi\beta_X^{\F}}\arrow{d}{}&\varphi\bSigma_{\F}(X)\arrow{d}{\varphi\bSigma_{\F}(f)}\arrow[confl]{r}{\Gamma\delta_X^{\F}}&\phantom{}\\
\varphi Y\arrow[infl]{r}{\varphi\alpha_Y^{\F}}&\varphi I_{\F}(Y)\arrow[defl]{r}{\varphi\beta_Y^{\M}}&\varphi\bSigma_{\F}(Y)\arrow[confl]{r}{\Gamma\delta_Y^{\F}}&\phantom{}
\end{tikzcd}
\end{equation}
in $\F'$, and applying $\varphi$ to the diagram \eqref{eq:induced-tri} yields the map of conflations
\begin{equation}
\label{eq:induced-tri-part}
\begin{tikzcd}[column sep=3.5pc]
\varphi X \arrow[r,infl]{}{\varphi\alpha} \arrow[d,equal] & \varphi Y \arrow[r,defl]{}{\varphi\beta} \arrow[d] & \varphi Z \arrow[r,confl]{}{\Gamma\delta} \arrow[d]{}{h} & \phantom{} \\
\varphi X \arrow[r,infl]{}{\alpha^{\F}_X} & \varphi I_{\F}(X) \arrow[r,defl]{}{\varphi\beta^{\F}_X} & \varphi\bSigma_{\F}(X) \arrow[r,confl]{}{\Gamma\delta_X^{\F}} & \phantom{}
\end{tikzcd}
\end{equation}
in $\F'$. The diagrams \eqref{eq:induced-morph-part} and \eqref{eq:induced-tri-part} are exactly of the form required to construct the triangulated structure on $\underline{\F}'$.
As before, we may choose arbitrary diagrams of this form for those morphisms and conflations from $\F'$ which are not in the image of $\varphi$.

With this set of choices, the map $\bSigma_{\F'}$ satisfies $\bSigma_{\F'}\varphi=\varphi\bSigma_{\F}$ on both objects and morphisms.
It follows that the suspension functor $\Sigma_{\F'}$ induced from these choices satisfies $\Sigma_{\F'}\pi'\varphi=\pi'\varphi\bSigma_{\F}$, and that we have distinguished triangles in $\F'$ given by the sextuples
\begin{equation}
\label{eq:standard-triangle-2}
\begin{tikzcd}[column sep=2.5pc]
\pi'\varphi X \arrow[r]{}{\pi'\varphi f} &  \pi'\varphi Y \arrow[r]{}{ \pi'\varphi g} &  \pi'\varphi Z \arrow[r]{}{ \pi'\varphi h} & \Sigma_{\F'}(\pi'\varphi X) \end{tikzcd}
\end{equation}
obtained by applying $\pi'$ to \eqref{eq:induced-tri-part}.

But now commutativity of the diagram \eqref{eq:induced-tri-fun} implies that
\[\underline{\varphi}\Sigma_{\F}\pi=\underline{\varphi}\pi\bSigma_{\F}=\pi'\varphi\bSigma_{\F}=\Sigma_{\F'}\pi'\varphi=\Sigma_{\F'}\underline{\varphi}\pi,\]
and so since $\pi$ is an epimorphic functor it follows that $\underline{\varphi}$ intertwines the suspensions.
Similarly, applying $\underline{\varphi}$ to the standard triangle \eqref{e:standardtriangles-extri} yields the standard triangle \eqref{eq:standard-triangle-2}, so $\underline{\varphi}$ is exact. This completes the proof.
\end{proof}

\begin{corollary}
\label{c:stab-equiv}
If $\F$ and $\F'$ are equivalent as Frobenius extriangulated categories, then the stable categories $\underline{\F}$ and $\underline{\F}'$ are equivalent as triangulated categories.
\end{corollary}
\begin{proof}
If the functor $\varphi$ from Lemma~\ref{l:stab-fun} is an equivalence of extriangulated categories then it is in particular an equivalence of additive categories, and it moreover preserves projective-injective objects.
We may thus check directly that $\underline{\varphi}$ is an equivalence of additive categories.
Since it is also a triangle functor, by Lemma~\ref{l:stab-fun}, it is an equivalence of triangulated categories; this is a special case of the result \cite[Prop.~2.13]{NOS} for extriangulated functors in general, but was also well-known for triangle functors prior to this (see, for example, \cite[Prop.~1.41]{Huybrechts06}).
\end{proof}

\begin{corollary}
\label{c:2-step-stab}
Let $\F$ be a Frobenius extriangulated category, and let $\PR\subseteq\F$ be a class of projective-injective objects in $\F$.
Then the stable categories of $\F$ and $\F/\PR$ are triangle equivalent.
\end{corollary}
\begin{proof}
Recall from Proposition~\ref{p:partial-stab} that in this situation $\F/\PR$ is itself a Frobenius extriangulated category.
We deduce the result by applying Lemma~\ref{l:stab-fun} to the tautological projection functor $\pi\colon\F\to\F/\PR$; this is an extriangulated functor as in Remark~\ref{r:extri-functor}, and preserves projective-injective objects since these coincide in $\F$ and $\F/\PR$.
We may check (using the third isomorphism theorem) that the induced functor $\underline{\pi}\colon\underline{\F}\to\underline{\smash{\F/\PR}}$ from \eqref{eq:induced-tri-fun} is an equivalence of additive categories, and thus conclude as in the proof of Corollary~\ref{c:stab-equiv}.
\end{proof}

We now have everything we need to prove the equivalence we want.

\begin{theorem}
\label{thm:triangleequivalence-extri}
Let $\F$ be a stably $2$-Calabi--Yau Frobenius extriangulated category, and let $\X\subseteq\F$ be rigid and functorially finite.
Then there is a triangle equivalence
\[\underline{\rperp{\X}} \simeq \X_{\underline{\mc{F}}}^{\perp_1}/\add(\X).\]
\end{theorem}

\begin{proof}
The category $\X_{\underline{\mc{F}}}^{\perp_1}$ is a Frobenius extriangulated category with stable category $\rperp{\X}_{\underline{\F}}/\add(\X)$ by Proposition~\ref{Prop:MperpisFrobenius-extri}, and $\rperp{\X}_{\underline{\F}}\simeq\rperp{\X}/\PR$ as extriangulated categories by Proposition~\ref{p:extri-equiv}, we have a triangle equivalence
\[\underline{\rperp{\X}/\PR}\simeq\rperp{\X}_{\underline{\F}}/\add(\X)\]
of stable categories by Corollary~\ref{c:stab-equiv}.
Moreover, there is a triangle equivalence $\underline{\rperp{\X}}\simeq\underline{\smash{\rperp{\X}/\PR}}$ by Corollary~\ref{c:2-step-stab}, and the result follows.
\end{proof}

\begin{remark}
\label{r:compare-extri}
We note that $\underline{\rperp{\X}}$ is the stable category of the reduction $\rperp{\X}$ of $\F$ at $\X$, whereas $\rperp{\X}_{\underline{\F}}/\add(\X)$ is the Iyama--Yoshino reduction of the stable category $\underline{\F}$ at $\X$.
Thus, the content of Theorem~\ref{thm:triangleequivalence-extri} is that taking the stable category intertwines the notion of reduction considered here for Frobenius extriangulated categories with Iyama--Yoshino's reduction for triangulated categories.
\end{remark}

\begin{remark}
In Theorem~\ref{thm:triangleequivalence-extri}, the category $\rperp{\X}$ is stably $2$-Calabi--Yau by Proposition~\ref{p:Mperp-2cy-extri}.
We can also see this by combining Theorem~\ref{thm:triangleequivalence-extri} with~\cite[Thm.~4.7]{IY08}, which gives that the category $\X_{\underline{\mc{F}}}^{\perp_1}/\add(\X)$ is triangulated $2$-Calabi--Yau.
\end{remark}

\begin{corollary}
Assume that $X\in\F$ is a rigid object, and $T\in\rperp{X}$ a cluster-tilting object. Then $\underline{A}'=\stabEnd_{\rperp{X}}(T)^{\opp}$ is related to $\underline{A}=\stabEnd_{\F}(T)^{\opp}$ by $\tau$-tilting reduction, in the sense of Jasso \cite{Jasso}, at the projective $\underline{A}$-module $P=\stabHom_{\F}(T,X)$.
In particular, support $\tau$-tilting modules for $\underline{A}'$ are in bijection with support $\tau$-tilting modules for $\underline{A}$ with $P$ in their additive closure.
\end{corollary}
\begin{proof}
By Theorem~\ref{thm:triangleequivalence-extri}, the algebra $\stabEnd_{\rperp{X}}(T)^{\opp}$ is isomorphic to the endomorphism algebra of $T$ in the Iyama--Yoshino reduction $\rperp{X}_{\underline{\F}}/\add(X)$ of $\underline{\F}$ (in which $T$ is also cluster-tilting by Proposition~\ref{p:ct-bij-extri}).
Thus the result follows from \cite[Thm.~4.24]{Jasso}.
\end{proof}

\section{Cluster characters}
\label{s:CC}

Let $\F$ be a Krull--Schmidt, stably $2$-Calabi--Yau Frobenius extriangulated category with a cluster tilting object $T$.
Then $\F$ is in particular idempotent complete \cite[Cor.~4.4]{KrauseKS}, hence weakly idempotent complete \cite[Lem.~A.6.2]{TT90}.
By work of Wang, Wei and Zhang \cite{WWZ2}, the category $\F$ admits a cluster character, a function taking each object of $\F$ to a Laurent polynomial and satisfying various multiplication formulae relating the product of cluster characters of objects $M$ and $N$ to the cluster characters of possible middle terms of extensions between them.
This builds on many precursors, notably Caldero--Chapoton \cite{CC06} for abelian categories, Palu \cite{Palu} for triangulated categories, and Fu--Keller \cite{FuKeller} for Frobenius exact categories.
We also refer to Keller--Wu \cite{KellerWu} for the special case of Higgs categories, certain stably $2$-Calabi--Yau Frobenius extriangulated categories associated to an ice quiver with potential by Wu \cite{Wu}.

Now if $M\in\add(T)$, then $\M$ is also a stably $2$-Calabi--Yau Frobenius extriangulated category, in which $T$ is again a cluster-tilting object, by Theorem~\ref{t:reduction}, and so it will often carry its own cluster character.
In this section, we demonstrate that, as expected, the cluster character on $\M$ is obtained by restricting that on $\F$.
We begin by recalling the ingredients of the cluster character formula.

\begin{definition}
Let $\F$ be a Frobenius extriangulated category and let $\T$ be a cluster-tilting subcategory. For each $X\in\F$, choose a conflation
\begin{equation}
\label{eq:index-confl}
\begin{tikzcd}
T^1_X\arrow[infl]{r}&T^0_X\arrow[defl]{r}{r}&X\arrow[confl]{r}&\phantom{}
\end{tikzcd}
\end{equation}
with $T^0_X,T^1_X\in\T$.
Then the \emph{index} of $X$ with respect to $\T$ is $\ind_{\T}(X)=[T^0_X]-[T^1_X]\in\K_0(\T)$, this Grothendieck group being the free group on the isoclasses of indecomposable objects of the rigid subcategory $\T$.
\end{definition}

We note that the choice of conflation \eqref{eq:index-confl} is equivalent to the choice of a right $\T$-approximation $r\colon T^0_X\to X$.
Such an approximation exists, and is a deflation, by Remark~\ref{r:weak-to-strong}, recalling that $\F$ is weakly idempotent complete.
Moreover, $\ind_{\T}(X)$ is independent of this choice by \cite[Lem.~4.36, Rem.~4.37]{PPPP}, and hence is well-defined.
As usual, if $\T=\add(T)$ for $T$ a cluster-tilting object, we write $\ind_{T}:=\ind_{\T}$.

Now assume $\F$ has a cluster-tilting object $T$, which we fix for the rest of the section.
Let $\underline{A}=\stabEnd_{\F}(T)^\opp$ be its stable endomorphism algebra.
Since $\F$ is stably $2$-Calabi--Yau, $\underline{\F}$ is in particular Hom-finite, and so $\underline{A}$ is a finite-dimensional algebra.
Thus the category $\module{\underline{A}}$ of finitely generated (equivalently, finite-dimensional) $\underline{A}$-modules is abelian, and we write $\K_0(\module{\underline{A}})$ for its Grothendieck group as an abelian category.

\begin{definition}
\label{def:Theta}
For each $Y\in\F$, choose a conflation
\begin{equation}
\label{eq:Theta-confl}
\begin{tikzcd}
Y\arrow[infl]{r}&P\arrow[defl]{r}&Y'\arrow[confl]{r}&\phantom{}
\end{tikzcd}
\end{equation}
with $P\in\PR$, and set $\Theta(Y)=\ind_{T}(Y)-[P]+\ind_{T}(Y')\in\K_0(\add(T))$.
\end{definition}
The function $\Theta$ is well-defined as in \cite[\S3]{WWZ2} (although loc.\ cit.\ write $\Theta(Y')$ in place of our $\Theta(Y)$).
Note that $\PR\subseteq\add(T)$ since $T$ is cluster-tilting, and so $[P]=\ind_{T}(P)$.

Write $G=\Ext^1(T,-)\colon\F\to\module{\underline{A}}$; this functor is essentially surjective by \cite[Cor.~4.4]{KoenigZhu} (see also \cite[Prop.~2.1]{KellerReiten07}, \cite[Thm.~A]{BMR}).
We remark that this is different from the functor denoted by $\mathbb{G}$ in \cite{WWZ2}, although both functors factor over the stable category $\underline{\F}$, on which $G=\mathbb{G}\circ\Sigma$; our preference for $G$ over $\mathbb{G}$ is also the reason for the change in convention in the definition of $\Theta$.
It nevertheless follows from \cite[Thm.~3.9]{WWZ2} that $\Theta(Y)$ depends only on $[GY]\in\K_0(\module{\underline{A}})$, and so we get an induced function $\theta\colon\K_0(\module{\underline{A}})\to\K_0(\add(T))$ defined on classes of objects by $\theta[GY]=\Theta(Y)$.

\begin{definition}[{\cite[Defn.~4.2]{WWZ2}}]
\label{d:clucha}
Let $\F$ be a Krull--Schmidt, stably $2$-Calabi--Yau Frobenius extriangulated category, and let $T\in\F$ be a cluster-tilting object.
Write $\underline{A}=\stabEnd_{\F}(T)^\opp$ and $G=\EE_\F(T,-)\colon\F\to\module{\underline{A}}$.
Then for each $X\in\F$, we define
\begin{equation}
\label{eq:WWZCC}
\clucha{\F}{T}(X)=x^{\ind_T(X)}\sum_d\chi(\QGr{d}{GX})x^{-\theta(d)}\in\QQ[\K_0(\add(T))].
\end{equation}
Here the sum is taken over possible dimension vectors $d$ of $\underline{A}$-modules, $\QGr{d}{GX}$ denotes the Grassmannian of submodules of $GX$ of dimension vector $d$,
and $\chi$ denotes the Euler--Poincaré characteristic.
\end{definition}

In order to rewrite \eqref{eq:WWZCC} as in \cite{WWZ2}, we may assume $T$ is basic, and choose a decomposition $T=\bigoplus_{i=1}^nT_i$ of $T$ into indecomposable direct summands.
This yields a basis (or, strictly speaking, a free generating set) $[T_i]$ of $\K_0(\add(T))$, and allows us to write $\clucha{\F}{T}(X)$ as a Laurent polynomial in the variables $x_i=x^{[T_i]}$.
Writing $v=\sum_{i=1}^n[v:T_i][T_i]$ for each $v\in\K_0(\add(T))$, the resulting expression is
\begin{equation}
\label{eq:WWZCC2}
\clucha{\F}{T}(X)=x_i^{[\ind_T(X):T_i]}\sum_d\chi(\QGr{d}{GX})x^{-[\theta(d):T_i]}\in\QQ[x_1^{\pm1},\dotsc,x_n^{\pm1}].
\end{equation}
Since $GT_i=0$ and $\ind_T(T_i)=[T_i]$, it follows immediately that $\clucha{\F}{T}(T_i)=x_i$.
While we have taken coefficients in $\QQ$ for compatibility with \eqref{eq:WWZCC} (and large parts of the cluster algebras literature), the Euler--Poincaré characteristics of quiver Grassmannians in fact lie in $\ZZ$.

By \cite[Thm.~4.4]{WWZ2}, the function $\clucha{\F}{T}$ is a cluster character in the sense of \cite[Defn.~4.1]{WWZ2} (following \cite[Defn.~1.2]{Palu}).
While this explains why the function is interesting from a cluster-theoretic point of view, we will not actually use this property here.
A comparison of Definition~\ref{d:clucha} with other formulae for cluster characters in the literature can be found in \cite[\S4.4]{WWZ2}.
In loc.\ cit., it is shown for exact categories that the cluster character from Definition~\ref{d:clucha} coincides with Fu--Keller's cluster character \cite{FuKeller} only under an additional technical assumption, which turns out to be redundant.

\begin{proposition}
\label{p:FKCC}
Assume that $\F$ is a Krull--Schmidt, stably $2$-Calabi--Yau Frobenius exact category, and assume that $T\in\F$ is a cluster-tilting object such that $A=\End_{\F}(T)^\opp$ is Noetherian.
Then the cluster character $\clucha{\F}{T}$ from Definition~\ref{d:clucha} coincides with Fu--Keller's \cite{FuKeller}.
\end{proposition}
\begin{proof}
First observe that assumptions on $\F$ and $T$ are sufficient for the existence of the Fu--Keller cluster character, despite the fact that we only assume that $\End_{\C}(T)^{\opp}$ is Noetherian, rather than that $\C$ is Hom-finite.
Indeed the proof of this statement in \cite[Thm.~3.3]{FuKeller} remains valid (cf.~\cite[Thm.~6.8]{Pre22}).

By \cite[Prop.~4.8]{WWZ2} (which is also valid under our slightly weaker hypotheses), it is sufficient to check that
\[[\theta[GX]:T_i]=\langle GX,S_i\rangle_3\]
whenever $T_i$ is a projective indecomposable summand of $T$, with $S_i$ the corresponding simple module of $A$, and where
\[\langle M,N\rangle_3=\dim\Hom_{\underline{A}}(M,N) - \dim\Ext^1_{\underline{A}}(M,N) + \dim\Ext^1_{\underline{A}}(N,M) - \dim\Hom_{\underline{A}}(N,M).\]
Note that we have adjusted \cite[Cond.~4.7]{WWZ2} to reflect our definition of $\Theta$ and of the functor $G$.

By \cite[Prop.~4(c)]{KellerReiten07} (and its proof), any $M\in\module{\underline{A}}$ has projective dimension at most $3$ as an $A$-module.
As a result, there is a well-defined Euler pairing $\langle-,-\rangle_{\textrm{Eul}}\colon\K_0(\module{\underline{A}})\times K_0(\module{A})\to\ZZ$ defined by
\[\langle M,N\rangle_{\textrm{Eul}}=\dim\Hom_A(M,N)-\dim\Ext^1_A(M,N)+\dim\Ext^2_A(M,N)-\dim\Ext^3_A(M,N)\]
for $M\in\module{\underline{A}}$ and $N\in\module{A}$.
We may moreover use the relative Calabi--Yau property of $A$ \cite[Prop.~4(c)]{KellerReiten07} to see that in this case $\dim\Ext^i_A(M,N)=\dim\Ext^{3-i}_A(N,M)$, and hence $\langle M,N\rangle_{\textrm{Eul}}=\langle M,N\rangle_3$.
Noting that $GX\in\module{\underline{A}}$, and $\theta[GX]=\Theta(X)$ by definition, we thus need to show that
\[[\Theta(X):T_i]=\langle GX,S_i\rangle_{\textrm{Eul}}.\]

To this end, let $F=\Hom_{\F}(T,-)$ be the covariant Yoneda functor, inducing an equivalence $F\colon\add(T)\stackrel{\sim}{\to}\projmod{A}$.
Applying $F$ to \eqref{eq:index-confl}, which is a short exact sequence since $\F$ is exact, we obtain the exact sequence
\[\begin{tikzcd}
0\arrow{r}&FT^1_X\arrow{r}&FT^0_X\arrow{r}&FX\arrow{r}&GT_1=0,
\end{tikzcd}\]
a projective resolution of $FX\in\module{A}$.
Similarly, applying $F$ to the exact sequence \eqref{eq:Theta-confl} for $X$ yields
\[\begin{tikzcd}
0\arrow{r}&FX\arrow{r}&FP\arrow{r}&FX'\arrow{r}&GX\arrow{r}&GP=0.
\end{tikzcd}\]
Now if $T'\in\add(T)$ then $FT'$ is projective, and so
\[\langle FT',S_i\rangle_{\textrm{Eul}}=\dim\Hom_{A}(FT',S_i)=[T':T_i].\]
It follows that
\begin{align*}
\langle GX,S_i\rangle_{\textrm{Eul}}&=
\langle FX',S_i\rangle_{\textrm{Eul}}-
\langle FP,S_i\rangle_{\textrm{Eul}}+
\langle FX,S_i\rangle_{\textrm{Eul}}
\\
&=
\langle FT^0_{X'},S_i\rangle_{\textrm{Eul}}-\langle FT^1_{X'},S_i\rangle_{\textrm{Eul}}-\langle FP,S_i\rangle_{\textrm{Eul}}+\langle FT^0_X,S_i\rangle_{\textrm{Eul}}-\langle FT^1_X,S_i\rangle_{\textrm{Eul}}\\
&=[T^0_{X'}:T_i]-[T^1_{X'}:T_i]-[P:T_i]+[T^0_X:T_i]-[T^1_X:T_i]\\
&=[\ind_T(X'):T_i]-[P:T_i]+[\ind_T(X):T_i]\\
&=[\Theta(X):T_i],
\end{align*}
as required.
\end{proof}

\begin{theorem}
\label{t:cc_restrict}
Let $\F$ be a Krull--Schmidt, stably $2$-Calabi--Yau Frobenius extriangulated category.
Let $T\in\F$ be a cluster-tilting object, and let $M\in\add(T)$.
Then $\M$ is again a Krull--Schmidt, stably $2$-Calabi--Yau Frobenius extriangulated category, and
\[\FKcc{\M}{T}=\FKcc{\F}{T}|_{\M}.\]
\end{theorem}
\begin{proof}
The required properties of $\M$ follow from Theorem~\ref{t:reduction} together with Krause's characterisation of Krull--Schmidt categories \cite[Cor.~4.4]{KrauseKS}; the fact that $\M$ is idempotent complete is proved analogously to Lemma~\ref{l:Mperpwic}.
Since $T$ is cluster-tilting in $\F$, $T\in \M$ and $M\in \add(T)$, it follows from Proposition~\ref{p:ct-bij-extri} that $T$ is cluster-tilting in
$\M$. So Definition~\ref{d:clucha} also applies to the object $T$ in $\M$, yielding the cluster character $\clucha{\M}{T}$.

Now let $X\in\M$; we aim to show that $\clucha{\M}{T}(X)=\clucha{\F}{T}(X)$, and will do this by a direct comparison of the formulae \eqref{eq:WWZCC} for these two cluster characters.
The definition of $\ind_T(X)$ is insensitive to whether we view $X\in\M$ or $X\in\F$, so the two leading factors agree.

Since every projective-injective object in $\F$ is also projective-injective in $\M$ by Proposition~\ref{Prop:MperpisFrobenius-extri}, there is a surjective algebra homomorphism from $\underline{A}=\stabEnd_{\F}(T)^{\opp}$ to $\underline{A}'=\stabEnd_{\M}(T)^{\opp}$.
We may thus use the induced fully-faithful restriction functor to view
\[\module{\underline{A}'}\subseteq\module{\underline{A}}.\]
Now for any $X\in\M$, we have $\EE(M,X)=0$ by definition, and so $GX\in\module{\underline{A'}}$.
Thus $\QGr{d}{GX}=\varnothing$ unless $d$ is a dimension vector for $\underline{A}'$, and so while the summation set in the expression for $\clucha{T}{\F}(X)$ is larger than that in the expression for $\clucha{T}{\M}(X)$, the additional terms are all zero.

To complete the proof, we show that if $Y\in\M$ and
\[\begin{tikzcd}
Y\arrow[infl]{r}&Q\arrow[defl]{r}&Y''\arrow[confl]{r}&\phantom{}
\end{tikzcd}\]
is a conflation in $\M$ with $Q\in\add(M,\PR)$, then $\Theta(Y)=\ind_T(Y)-[Q]+\ind_T(Y'')$.
Indeed, $\add(M,\PR)$ is the category of projective-injective objects in $\M$ by Lemma~\ref{l:enough-extri}, and so this will imply that the definition of $\Theta$ is insensitive to whether we view $Y\in\M$ or $Y\in\F$.

By the dual of \cite[Lem.~A.11]{ChenThesis} (stated explicitly in \cite[Prop.~1.19]{Palu-Extri}), there is a commutative diagram
\[\begin{tikzcd}
Y\arrow[infl]{r}\arrow[infl]{d}&P\arrow[defl]{r}\arrow[infl]{d}&Y'\arrow[confl]{r}\arrow[equal]{d}&\phantom{}\\
Q\arrow[infl]{r}\arrow[defl]{d}&P\oplus Y''\arrow[defl]{r}\arrow[defl]{d}&Y'\arrow[confl]{r}&\phantom{}\\
Y''\arrow[equal]{r}\arrow[confl]{d}&Y''\arrow[confl]{d}\\
\phantom{}&\phantom{}
\end{tikzcd}\]
in which the top row is a conflation as in Definition~\ref{def:Theta} and the middle vertical conflation splits since $P\in\PR$ is injective.
Since $Q\in\rperp{T}$, we have
\[[Q]=[P]+\ind_T(Y'')-\ind_T(Y')\]
by \cite[Lem.~3.8(3)]{WWZ1} (see also \cite[Prop.~2.2]{Palu}), noting that $\ind_T(P)=[P]$ and $\ind_T(Q)=[Q]$ since $P,Q\in\add(T)$.
It follows that
\[\Theta(Y)=\ind_T(Y)-[P]+\ind_T(Y')=\ind_T(Y)-[Q]+\ind_T(Y''),\]
as required.
\end{proof}

When $\F$ is triangulated, in which case the cluster character is due to Palu \cite{Palu}, Fu and Keller also give some results, most notably \cite[Thm.~6.3]{FuKeller}, concerning the restriction of the cluster character to $\M$ for a rigid object $M\in\F$.

\section{Internally Calabi--Yau algebras}
\label{s:icy}

We recall the notion of a bimodule internally $3$-Calabi--Yau algebra, and the main properties of such an algebra, from \cite{Pre17}.

\begin{definition}[cf.~{\cite[Defn.~2.4]{Pre17}}]
\label{def:bi3cy}
Let $A$ be a Noetherian algebra and $e\in A$ an idempotent.
We say that $(A,e)$ is bimodule internally $3$-Calabi--Yau if $A$ is perfect with projective dimension at most $3$ when considered as an $A$-bimodule, and fits into a triangle
\[A\longrightarrow\Omega_A[3]\longrightarrow C\longrightarrow\Sigma A\]
of such bimodules, where $\Omega_A=\RHom_{A\otimes A^{\opp}}(A,A\otimes A^{\opp})$
is the bimodule dual, and $\RHom_A(C,M)=0=\RHom_{A^{\opp}}(C,N)$ for any complexes $M$ and $N$ with finite-dimensional total cohomologies such that $eM=0$, respectively $Ne=0$.
\end{definition}

While this definition is somewhat technical, we will only need a few consequences.
For example, if $(A,e)$ is bimodule internally $3$-Calabi--Yau, then it follows that $A$ has global dimension at most $3$, and that there is a functorial duality
\[\Ext^i_A(N,M)=\mathrm{D}\Ext^{3-i}_A(M,N)\]
when $M$ is an arbitrary $A$-module and $N$ is finite-dimensional with $eN=0$, cf.~\cite[Cor.~2.9]{Pre17}.
Recall that $\mathrm{D}=\Hom_{\field}(-,\field)$ denotes duality over the ground field.
Moreover, we can use a bimodule internally Calabi--Yau algebra to construct a stably $2$-Calabi--Yau Frobenius exact category with a cluster-tilting object, as follows.

\begin{theorem}[{\cite[Thms.~4.1, 4.10]{Pre17}}]
\label{thm:Frobcat-constr}
Let $A$ be a Noetherian algebra and $e\in A$ an idempotent such that the quotient algebra $A/AeA$ is finite-dimensional.
Write $B=eAe$.
If $(A,e)$ is bimodule internally $3$-Calabi--Yau, then
\begin{enumerate}
\item the algebra $B$ is an Iwanaga--Gorenstein algebra and so the category
\[\GP(B)=\{X\in\module{A}:\text{$\Ext^i_B(X,B)=0$ for all $i>0$}\}\]
of Gorenstein projective $B$-modules is Frobenius exact,
\item the category $\GP(B)$ is stably $2$-Calabi--Yau,
\item the $B$-module $eA$ is a cluster-tilting object in $\GP(B)$ and
\item the natural map $A\to\End_B(eA)^{\opp}$ is an isomorphism, inducing an isomorphism $A/AeA\to\underline{\End}_B(eA)^{\opp}$.
\end{enumerate}
\end{theorem}

\begin{remark}
\label{r:GPBwic}
Since $\module{B}$ is idempotent complete and $\Ext$-vanishing conditions are closed under taking direct summands, it follows that $\GP(B)$ is also idempotent complete (and hence weakly idempotent complete; see e.g.~\cite[Lem.~A.6.2]{TT90}). %
\end{remark}

\begin{definition}
Let $A$ be an algebra.
We make the set of idempotent elements of $A$ into a poset by defining $e'\geq e$ if and only if $e'e=e=ee'$.
\end{definition}

\begin{remark}
We will usually work with algebras presented by a quiver $Q$ with relations.
In this context, it will be enough to consider idempotents of the form
\[e=\sum_{i\in S}e_i\]
for some set $S\subset Q_0$, where $e_i$ denotes the vertex idempotent at vertex $i$.
The poset structure on the idempotents of this form coincides precisely with the standard one on the power set of $Q_0$, given by inclusion of subsets.
\end{remark}

Assuming $(A,e)$ is bimodule internally $3$-Calabi--Yau with respect to $e$, choose an idempotent $e'\geq e$ and write $B=eAe$ and $B'=e'Ae'$.
It then follows directly from Definition~\ref{def:bi3cy} that $(A,e')$ is also bimodule internally $3$-Calabi--Yau.
Since $AeA\subset Ae'A$ in this situation, finite-dimensionality of $A/AeA$ also implies that of $A/Ae'A$.
Thus, when this finite-dimensionality holds, we may use Theorem~\ref{thm:Frobcat-constr} (and Remark~\ref{r:GPBwic}) to see that $\GP(B)$ and $\GP(B')$ are weakly idempotent complete stably $2$-Calabi--Yau Frobenius exact categories, with cluster-tilting objects $eA$ and $e'A$ respectively.
By Theorem~\ref{thm:Frobcat-constr} again, both of these cluster-tilting objects have endomorphism algebra isomorphic to $A$.

Our aim for the remainder of the section is to relate these two exact categories, by showing that $\GP(B')$ is equivalent to the reduction $\M\subseteq\GP(B)$ with respect to the rigid $B$-module $M=eAe'$.

\begin{lemma}
\label{lem:M-to-B'}
There is a natural isomorphism $B'\to\End_B(M)^{\opp}$.
\end{lemma}
\begin{proof}
The natural map $A\to\End_B(eA)^{\opp}$ induces a natural map
\[B'=e'Ae'\to e'\End_B(eA)^{\opp}e'=\End_B(eAe')^{\opp}=\End_B(M)^{\opp}.\]
By Theorem~\ref{thm:Frobcat-constr}, this map is an isomorphism.
\end{proof}

Note that Lemma~\ref{lem:M-to-B'} recovers the famous isomorphism $B\to\End_B(B)^{\opp}$ when $e=e'$, using that $B=eAe$.

As observed above, there are isomorphisms $A\cong\End_{B}(eA)^{\opp}$ and $A\cong\End_{B'}(e'A)^{\opp}$. This implies that $\add(eA)\subset\M$ and $\add(e'A)\subset\GP(B')$ are equivalent categories, and we claim that an explicit equivalence is given by $\Hom_B(M,-)$.

\begin{lemma}
\label{lem:Hom-M-eA}
$\Hom_B(M,eA)=e'A$.
\end{lemma}
\begin{proof}
Similar to the proof of Lemma~\ref{lem:M-to-B'}, the natural map $A\to\End_B(eA)^{\opp}$ induces a natural map
\[e'A\to e'\End_B(eA)^{\opp}=\Hom_B(eAe',eA)=\Hom_B(M,eA),\]
of vector spaces, which is an isomorphism by Theorem~\ref{thm:Frobcat-constr}.
(Note that the algebra structure on $\End_B(eA)^{\opp}$ is only used here to make sense of the subspace of elements divisible by $e'$ on the left.)
\end{proof}

Now consider the recollement
\[\begin{tikzcd}
\module{B'/B'eB'}\arrow{r}&\module{B'}\arrow{r}{e}\arrow[bend left]{l}\arrow[bend right]{l}&\module{B}\arrow[bend left]{l}{\Hom_B(eB',-)}\arrow[bend right,swap]{l}{B'e\otimes_B-}
\end{tikzcd}
\]
induced by the idempotent $e\in B'$.
Note that $eB'=eAe'=M$ so $\Hom_B(eB',-)=\Hom_B(M,-)$; here we use that $e'\geq e$.
By standard recollement theory, e.g.\ \cite{PV14}, the functor $\Hom(M,-)\colon\module(B)\rightarrow \module(B')$ is fully faithful.
The following is then immediate from this observation together with Lemma~\ref{lem:Hom-M-eA}.

\begin{lemma}
\label{l:clustertilting}
The functor $\Hom_B(M,-)$ induces an equivalence $\add eA\simeq \add e'A$.
\end{lemma}

Thus we see that $\Hom_B(M,-)$ is a fully faithful functor from $\module B$ to $\module B'$ taking the cluster-tilting object $eA\in\GP(B)$ to the cluster-tilting object $e'A\in\GP(B')$.
We restrict $\Hom_B(M,-)$ to $\M$, where $eA$ is still cluster-tilting by Proposition~\ref{p:ct-bij-extri}, since $M\in\add(eA)$, and we denote this restricted functor by $F$.

\begin{lemma}
\label{l:liesin}
If $X\in\M$, then $FX=\Hom_{B}(M,X)\in\GP(B')$.
\end{lemma}
\begin{proof}
Let $X$ be an object in $\M$.
Then, because $eA\in\M$ is cluster-tilting, there is a short exact sequence
\[\begin{tikzcd}
0\arrow{r}&X\arrow{r}&T_1\arrow{r}&T_2\arrow{r}&0
\end{tikzcd}\]
in $\M$, such that $T_i\in \add(eA)$ for $i=1,2$ (obtained by choosing the map $X\to T_1$ to be a left $\add(T)$-approximation of $X$).
Because $X\in\M$, applying $F$ yields a short exact sequence
\[\begin{tikzcd}
0\arrow{r}&FX\arrow{r}&FT_1\arrow{r}&FT_2\arrow{r}&0.
\end{tikzcd}\]
Applying $\Hom_{B'}(-,B')$ to this sequence gives exact sequences
\[\begin{tikzcd}
\Ext^j_{B'}(FT_1,B')\arrow{r}&\Ext^j_{B'}(FX,B')\arrow{r}&\Ext^{j+1}_{B'}(FT_2,B')
\end{tikzcd}\]
for all $j>0$. Since $e'A\in\GP(B')$ and $FT_i\in\add{e'A}$ for $i=1,2$, we have that $\Ext^j_{B'}(FT_1,B')=0=\Ext^{j+1}_{B'}(FT_2,B')$, and hence
$\Ext^j_{B'}(FX,B')=0$, for all $j>0$.
Thus $FX\in\GP(B')$.
\end{proof}

Thus $F$ is an exact and fully faithful functor from $\M$ to $\GP(B')$ taking the cluster-tilting object $eA$ in $\M$ to the cluster-tilting object $e'A$ in $\GP(B')$. This turns out to imply that $F$ is an equivalence.

For the following proposition, recall that a cluster-tilting subcategory (of some category $\C$) is a full and functorially finite subcategory $\T$ such that ${}^{\perp_1}\T=\T=\T^{\perp_1}$. In particular, if $T$ is a cluster-tilting object then $\add(T)$ is a cluster-tilting subcategory.

\begin{proposition}
\label{p:ctenough}
Let $\C$ and $\C'$ be stably $2$-Calabi--Yau Frobenius exact categories, and assume that $\C$ is weakly idempotent complete.
Suppose that there are cluster-tilting subcategories $\T\subset\C$ and $\T'\subset\C'$, and that $G\colon\C\rightarrow \C'$ is an exact functor restricting to an equivalence $\T\to\T'$.
Then $G$ is an equivalence.
\end{proposition}
\begin{proof}
We adapt the argument from \cite[Lem.~4.5]{KellerReiten08}.
Let $X$ be an object in $\C$.
Then, since $\T$ is cluster-tilting, there is an exact sequence
\begin{equation}
\label{eq:X-approximation}
\begin{tikzcd}
0\arrow{r}&T_2\arrow{r}&T_1\arrow{r}&X\arrow{r}&0,
\end{tikzcd}
\end{equation}
with $T_1,T_2\in\T$ (obtained by taking $T_1\to X$ to be a right $\T$-approximation of $X$).
Choosing $T\in\T$ and applying $\Hom_{\C}(T,-)$ to this sequence we get
\[\begin{tikzcd}
0\arrow{r}&\Hom_{\C}(T,T_2)\arrow{r}&\Hom_{\C}(T,T_1)\arrow{r}&\Hom_{\C}(T,X)\arrow{r}&0,
\end{tikzcd}\]
noting that $\Ext^1_{\C}(T,T_2)=0$ since $\T$ is cluster-tilting.
Since $G$ is exact we also have an exact sequence
\[\begin{tikzcd}
0\arrow{r}&GT_2\arrow{r}&GT_1\arrow{r}&GX\arrow{r}&0,
\end{tikzcd}\]
to which we may apply $\Hom_{\C'}(GT,-)$ to obtain
\[\begin{tikzcd}
0\arrow{r}&\Hom_{\C'}(GT,GT_2)\arrow{r}&\Hom_{\C'}(GT,GT_1)\arrow{r}&\Hom_{\C'}(GT,GX)\arrow{r}&0,
\end{tikzcd}\]
observing that $\Ext^1_{\C'}(GT,GT_2)=0$ since $G(\T)=\T'$ is cluster-tilting in $\C'$.

We now consider the commutative diagram
\begin{equation}
\label{eq:h}
\begin{tikzcd}[column sep=1.5pc]
0\arrow{r}&\Hom_{\C}(T,T_2)\arrow{r}\arrow{d}{f}&\Hom_{\C}(T,T_1)\arrow{r}\arrow{d}{g}&\Hom_{\C}(T,X)\arrow{r}\arrow{d}{h}&0\\
0\arrow{r}&\Hom_{\C'}(GT,GT_2)\arrow{r}&\Hom_{\C'}(GT,GT_1)\arrow{r}&\Hom_{\C'}(GT,GX)\arrow{r}&0
\end{tikzcd}
\end{equation}
with exact rows, in which each vertical map is induced by the functor $G$.
Since $G$ restricts to an equivalence $\T\to\T'$, both $f$ and $g$ are isomorphisms, and hence so is $h$.

Now let $X$ and $Y$ be arbitrary objects in $\C$ and choose a sequence \eqref{eq:X-approximation} with $T_1,T_2\in\add{T}$.
Applying $\Hom_{\C}(-,Y)$ yields an exact sequence
\[\begin{tikzcd}
0\arrow{r}&\Hom_{\C}(X,Y)\arrow{r}&\Hom_{\C}(T_1,Y)\arrow{r}&\Hom_{\C}(T_2,Y),
\end{tikzcd}\]
to which we further apply the left exact functors $G$ and $\Hom_{\C'}(-,GY)$ to obtain
\[\begin{tikzcd}
0\arrow{r}&\Hom_{\C'}(GX,GY)\arrow{r}&\Hom_{\C'}(GT_1,GY)\arrow{r}&\Hom_{\C'}(GT_2,GY).
\end{tikzcd}\]
We may thus construct a commutative diagram
\[\begin{tikzcd}
0\arrow{r}&\Hom_{\C}(X,Y)\arrow{r}\arrow{d}{p}&\Hom_{\C}(T_1,Y)\arrow{r}\arrow{d}{q}&\Hom_{\C}(T_2,Y)\arrow{d}{r}\\
0\arrow{r}&\Hom_{\C'}(GX,GY)\arrow{r}&\Hom_{\C'}(GT_1,GY)\arrow{r}&\Hom_{\C'}(GT_2,GY)
\end{tikzcd}\]
in which the vertical maps are once again induced from $G$. Both $q$ and $r$ are isomorphisms for the same reason as $h$ in \eqref{eq:h}.
Thus $p$ is also an isomorphism, and so $G$ is fully faithful.

To see that $G$ is dense, let $Y$ be an arbitrary object in $\C'$.
Then, since $\T'$ is cluster-tilting in $\C'$ and $G$ induces an equivalence from $\T$ to $\T'$, there are objects $T_1,T_2\in\T$ and a map $\varphi\colon T_1\rightarrow T_2$ from which we may form a short exact sequence
\begin{equation}
\begin{tikzcd}
0\arrow{r}&GT_1\arrow{r}{G\varphi}&GT_2\arrow{r}&Y\arrow{r}&0
\end{tikzcd}
\label{e:Ysequence}
\end{equation}
in $\C'$. 

First we show that $\varphi$ is an inflation in $\C$.
Let $i\colon T_2\rightarrow Q$ be an injective envelope of $T_2$.
Then $GQ$ is injective because $G$ is exact, and $G\varphi$ is an inflation, so there exists a map $\psi'\colon GT_1\to GQ$ such that the diagram
\[\begin{tikzcd}
GT_2\arrow{r}{G\varphi}\arrow[swap]{d}{Gi}&GT_1\arrow{dl}{\psi'}\\GQ
\end{tikzcd}\]
commutes. 
Since $F$ induces an equivalence between $\add T$ and $\add T'$, and $Q$ is a direct summand of $T$, there is a map $\psi:T_1\rightarrow Q$ such that $G\psi=\psi'$.
Furthermore, the preceding diagram is the image under $G$ of the commuting diagram
\[\begin{tikzcd}
T_2\arrow{r}{\varphi}\arrow[swap]{d}{i}&T_1\arrow{dl}{\psi}\\Q
\end{tikzcd}\]
in $\C$, from which we see that $\psi\varphi=i$ is an inflation.
Since $\C$ is weakly idempotent complete, we may therefore apply (the dual of) \cite[Prop.~7.6]{Buehler} to see that $\varphi$ is an inflation as required. 

Thus there is a short exact sequence
\[\begin{tikzcd}
0\arrow{r}&T_1\arrow{r}{\varphi}&T_2\arrow{r}&X\arrow{r}&0
\end{tikzcd}\]
in $\C$.
Since $G$ is exact, the image of this sequence under $G$ is isomorphic to the exact sequence~\eqref{e:Ysequence}, hence in particular $GX\cong Y$ and $G$ is dense as required.
\end{proof}

\begin{theorem}
\label{t:GP-reduction}
The functor $\Hom_B(M,-)$ induces an exact equivalence between the Frobenius exact categories $\M$ and $\GP(B')$ taking the cluster-tilting object $eA$ in $\M$ to the cluster-tilting object $e'A$ in $\GP(B')$.
\end{theorem}

\begin{proof}
Since $\GP(B)$ is idempotent complete, so is $\M$. By Lemma~\ref{l:liesin}, the restriction $F$ of $\Hom_B(M,-)$ to $\M$ is a exact functor from $\M$ to $\GP(B')$. By Lemma~\ref{l:clustertilting}, $F$ induces an equivalence from $\add(eA)$ to $\add(e'A)$. Hence, by Proposition~\ref{p:ctenough}, $F$ is an equivalence as required.
\end{proof}

\section{Examples}
\label{s:examples}

In this section, we apply the reduction method developed above to some Grassmannian cluster categories from~\cite{JKS16}.
This will allow us to make a connection to frieze patterns in Section~\ref{s:friezes}.  %

Let $n$ and $k$ be positive integers such that $1\leq k\leq n-1$, and let $\Gr(k,n)$ denote the Grassmannian of $k$-dimensional subspaces of $\mathbb{C}^n$. By~\cite{Scott06}, the homogeneous coordinate ring $\mathbb{C}[\Grcone(k,n)]$ of $\Gr(k,n)$ (i.e.\ the coordinate ring of the affine cone $\Grcone(k,n)$) is a cluster algebra.
The article~\cite{JKS16} defines a category $\CM(C_{k,n})$ which, by~\cite{JKS16} and~\cite{BKM16}, categorifies this cluster algebra. We recall the definition and properties of this category from~\cite[\S3]{JKS16}. We will actually follow the set-up in~\cite{CKP}, adapting the results from~\cite{JKS16} appropriately. Thus we consider a cyclic quiver with $n$ vertices $Q_0=\{1,2,\ldots, n\}$ and arrows $x_i$ (respectively, $y_i$), $i\in C_1=\{1,2,\ldots ,n\}$ (taken modulo $n$), joining adjacent vertices clockwise (respectively, anticlockwise). The quiver for $n=6$ is shown in Figure~\ref{f:quiver6}.
We choose labels for the vertices here, as in~\cite{JKS16}, since we will need to refer explicitly to the corresponding vector spaces when working with modules.

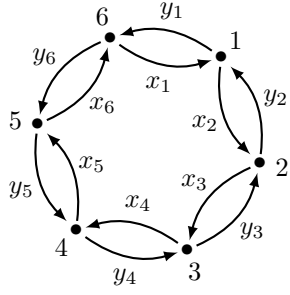
\begin{figure}
\begin{tikzpicture} [scale=1,
qarrow/.style={-latex, thick}]  
\foreach \j in {1,...,6}
{ \path (110-60*\j:1.5) node (w\j) {};
  \path (170-60*\j:1.5) node (v\j) {};
  \draw (w\j) node [black] {$\bullet$};
  \path [qarrow] (v\j) edge [bend right=26] (w\j);
  \path [qarrow] (w\j) edge [bend right=26] (v\j);
  \draw (140-60*\j:1.8) node[black] {$y_{\j}$};
  \draw (140-60*\j:0.8) node[black] {$x_{\j}$}; 
  \draw (110-60*\j:1.8) node[black] {$\j$}; 
}
\end{tikzpicture}
\caption{The cyclic quiver $Q$ for $n=6$.}
\label{f:quiver6}
\end{figure}

Consider the algebra with quiver $Q$ and relations, starting at any vertex, given by $xy=yx$ and $y^k=x^{n-k}$, where we interpret $x$ and $y$ as the appropriate arrows $x_i$ and $y_i$. Let $C_{k,n}$ be the completion of this algebra at the ideal generated by the arrows.
The article~\cite{JKS16} considers the algebra $\widetilde{C}_{k,n}$ with the relation $y^k=x^{n-k}$ replaced with $x^k=y^{n-k}$; there is an isomorphism
$\alpha\colon C_{k,n}\rightarrow \widetilde{C}_{k,n}$ sending $x_i$ to $y_{n+1-i}$ and $y_i$ to $x_{n+1-i}$.

Then $\C_{k,n}=\CM(C_{k,n})$ is the category of (maximal) Cohen--Macaulay $C_{k,n}$-modules; by \cite[Cor.~3.7]{JKS16} this coincides with the category $\GP(C_{k,n})$ of Gorenstein projective $C_{k,n}$-modules, as appearing in Section~\ref{s:icy}.
The isomorphism $\alpha$ induces an equivalence $F$ from $\C_{k,n}$ to the category $\widetilde{\C}_{k,n}$ of (maximal) Cohen--Macaulay $\widetilde{C}_{k,n}$-modules.
By~\cite[Cor.~3.7]{JKS16}, $\widetilde{\C}_{k,n}$ is a Frobenius exact category and, as pointed out in~\cite[Rem.~3.3]{JKS16}, $\widetilde{\C}_{k,n}$ has Auslander--Reiten sequences and an Auslander--Reiten quiver by~\cite{Auslander86}; hence $\C_{k,n}$ has these properties also.

Each object in $\C_{k,n}$ is, by definition, free as a $Z$-module, where $Z=\powser{\CC}{t}$ is the centre of $C_{k,n}$, generated by $t=xy$. Let $K$ denote the field of fractions of $Z$; then the rank of an object in $\C_{k,n}$ is defined to be the length of $M\otimes_Z K$ as a $B\otimes_Z K$-module, noting that $B\otimes_Z K$ is isomorphic to the simple algebra $M_n(K)$ of $n\times n$ matrices over $K$.

If $I$ is a $k$-subset of $C_1$ (i.e.\ a subset of cardinality $k$), then, as in~\cite[Defn.~3.2]{CKP}
(adapted from~\cite[Defn.\ 5.1]{JKS16}), we define a $C_{k,n}$-module $M_I$ as follows. For $j\in Q_0$, set $V_j=Z$ and, for $a\in C_1$, define
$$x_a=\begin{cases} \text{multiplication by $t$,} & a\in I; \\ \text{multiplication by $1$,} & a\not\in I;\end{cases} \quad\quad
y_a=\begin{cases} \text{multiplication by $1$,} & a\in I; \\ \text{multiplication by $t$,} & a\not\in I.\end{cases}$$
Note that $F(M_I)\cong \widetilde{M}_{I'}$, where
$\widetilde{M}_I$ is the $\widetilde{C}_{k,n}$-module defined in~\cite[Defn.\ 5.1]{JKS16} and $I'$ is the $k$-subset of $[1,n]$ obtained by replacing each element $i$ of $I$ with $n+1-i$.
In terms of the combinatorics of profiles as in \cite[\S6]{JKS16}, our notational convention means that the set $I$ records the upward steps in the profile of $M_I$. We take this difference into account below.

By~\cite[Prop.\ 5.2]{JKS16}, every rank $1$ $C_{k,n}$-module is isomorphic to a module $M_I$ for a unique $k$-subset $I$.

\begin{definition}
\label{d:morphisms}
By~\cite[Rem.\ 5.4]{JKS16}, $\Hom_{C_{k,n}}(M_I,M_J)$ is a free rank $1$ $Z$-module. 
A \emph{monomial} morphism (see~\cite[Defn.\ 7.3]{JKS16}) $f\in \Hom_{C_{k,n}}(M_I,M_J)$ is given by a tuple $(f_j)_{j\in Q_0}$
of $Z$-maps, where each $f_j$ is given by multiplication by $t^{\alpha_j}$
for nonnegative integers $\alpha_j$ satisfying
\[\alpha_{j}-\alpha_{j-1}=\begin{cases} 1, & j\in J\setminus I; \\ -1, & j\in I\setminus J; \\ 0, & \text{otherwise}.\end{cases}\]
A generator $\varphi^I_J$ of $\Hom_{C_{k,n}}(M_I,M_J)$ as a $Z$-module is given by the unique solution to this equation for which the tuple $\alpha=(\alpha_j)_{j\in Q_0}$ has at least one zero component.
\end{definition}

Let $I$ and $J$ be $k$-subsets of $[1,n]$. Then $I$ and $J$ are said to be \emph{non-crossing}~\cite{Scott06} (or \emph{weakly separated}~\cite{LeclercZelevinsky98}) if there do not exist $a,b,c,d$, cyclically ordered, such that $a,c\in I\setminus J$ and $b,d\in J\setminus I$. Otherwise, $I$ and $J$ are said to be \emph{crossing}.
By~\cite[Prop.~5.6]{JKS16}, $\Ext^1_{C_{k,n}}(M_I,M_J)=0$ if and only if $I$ and $J$ are non-crossing.

By~\cite[Rem.~3.3]{JKS16}), $\C_{k,n}$ has Auslander--Reiten sequences.
The Auslander--Reiten quiver, $\Gamma_{2,n}$, of $\C_{2,n}$ (shown in Figure~\ref{f:ARquiver2n}) was already given in~\cite[Eg.~5.3]{JKS16}, but the irreducible maps were not explicitly described, so we describe them here.

\begin{lemma}
\label{lem:ARsequence}
For any $2$-subset $\{i,j\}$ of $[1,n]$ with $j\not=i\pm 1$, the Auslander--Reiten sequence with first
term $M_{i,j}$ is, up to equivalence of short exact sequences, the sequence
\begin{equation}
\begin{tikzcd}
0 \arrow{r}{} & M_{i+1,j+1} \arrow{r}{f} & M_{i+1,j}\oplus M_{i,j+1}
\arrow{r}{g} & M_{i,j} \arrow{r}{} & 0,
\end{tikzcd}
\label{e:ARsequence}
\end{equation}
where $f=\begin{pmatrix} \varphi^{i+1,j+1}_{i+1,j} \\[5pt] \varphi^{i+1,j+1}_{i,j+1} \end{pmatrix}$
and
$g=\begin{pmatrix} \varphi^{i+1,j}_{i,j} &
-\varphi^{i,j+1}_{i,j} \end{pmatrix}$.
\end{lemma}

\begin{proof}
Since $k=2$, we have $\Ext^1_{C_{k,n}}(M_I,M_J)\cong \mathbb{C}$ whenever $I$ and $J$ are crossing \cite[Eg.~3.7]{BBGE20} (see also \cite{BBGE20C}). So it is enough to show that the sequence~\eqref{e:ARsequence} is exact (since then it is clearly not split).

From the description in Definition~\ref{d:morphisms}, the morphisms $\varphi^{i+1,j+1}_{i,j+1}$ and $\varphi^{i+1,j}_{i,j}$ multiply the component corresponding to vertex $i$ by $t$ and all other components by $1$, while
the morphisms $\varphi^{i+1,j+1}_{i+1,j}$ and $\varphi^{i,j+1}_{i,j}$ multiply the component corresponding to vertex $i$ by $t$ and all other components by $1$.

Since all of the maps $\varphi^I_J$ are injective, $f$ is injective. If $m=(z_a)_{a\in Q_0}$ is an arbitrary element of $M_{i,j}$, then we see that
$m=g((x_a)_{a\in C_1},(y_a)_{a\in Q_0})$, where $x_a=z_a$ for $a\not=j$, $x_j=0$, and $y_a=0$ for $a\not=j$, $y_j=z_j$, so $g$ is surjective.

The image of $f$ is the set of pairs
$(t^{\delta_{ai}}z_a)_{a\in Q_0},t^{\delta_{aj}}(z_a)_{a\in Q_0})$, where $z_a\in Z$ for all $a\in Q_0$. The kernel of $g$ is the set of
pairs
$((x_a)_{a\in Q_0},(y_a)_{a\in Q_0})$ satisfying $t^{\delta_{aj}}x_a=t^{\delta_{ai}}y_a$.
If $m=((t^{\delta_{ai}}z_a)_{a\in Q_0},t^{\delta_{aj}}(z_a)_{a\in Q_0})$ lies in the
image of $f$, then we have
$t^{\delta_{aj}}t^{\delta_{ai}}z_a=t^{\delta_{ai}}t^{\delta_{aj}}z_a$, so $m$ lies in the kernel of $g$.
Conversely, if
$m=((x_a)_{a\in Q_0},(y_a)_{a\in Q_0})$ lies in the kernel of $g$, then it satisfies $t^{\delta_{aj}}x_a=t^{\delta_{ai}}y_a$. We have
$$((x_a)_{a\in Q_0},(y_a)_{a\in Q_0})=
f((z_a)_{a\in Q_0}),$$ where $z_a=x_a=y_a$ if $a\not\in \{i,j\}$, $z_i=y_i$ and $z_j=x_j$, so $m$ lies in the image of $f$. We have shown that the sequence~\eqref{e:ARsequence} is a non-split exact sequence and the result follows.
\end{proof}
\begin{figure}
\[\begin{tikzpicture}
\foreach \x/\y/\a/\b/\e in {0/6/1/n/, 2/6/n-1/n/, 4/6/n-2/n-1/, 10/6/2/3/, 12/6/1/2/r, 1/5/1/n-1/, 3/5/n-2/n/, 9/5/2/4/, 11/5/1/3/r, 2/4/1/n-2/, 8/4/2/5/, 10/4/1/4/r, 4/2/1/3/, 6/2/2/n/, 8/2/1/n-1/r, 5/1/1/2/, 7/1/1/n/r}
{\draw (\x,\y) node (M\a\b\e) {$M_{\a,\b}$};}
\foreach \s/\t in {1n/1n-1, 1n-1/n-1n, 1n-1/1n-2, n-1n/n-2n, 1n-2/n-2n, n-2n/n-2n-1, 23/13r, 13r/12r, 24/23, 24/14r, 14r/13r, 25/24, 13/12, 12/2n, 2n/1nr, 1nr/1n-1r}
{\path [->] (M\s) edge (M\t);}
\foreach \s/\t in {1n-1/n-2n, 24/13r, 25/14r, 13/2n, 2n/1n-1r}
{\path [dashed] (M\s) edge (M\t);}
\foreach \x/\y/\a/\b in {6/6/n-3/n-2, 8/6/3/4, 5/5/n-3/n-1, 7/5/3/5, 4/4/n-3/n, 3/3/1/n-3, 5/3/n-4/n, 7/3/2/6, 9/3/1/5}
{\draw (\x,\y) node (M\a\b) {};}
\foreach \s/\t in {n-2n-1/n-3n-1, n-2n/n-3n, 1n-2/1n-3, 34/24, 35/25, 26/25, 15/14r}
{\path [->] (M\s) edge (M\t);}
\foreach \s/\t in {1n-3/13, n-3n/n-4n, 2n/26, 1n-1r/15, n-2n-1/34, n-3n-1/35}
{\path [dotted] (M\s) edge (M\t);}
\end{tikzpicture}\]
\caption{The Auslander--Reiten quiver of $\C_{2,n}$.}
\label{f:ARquiver2n}
\end{figure}
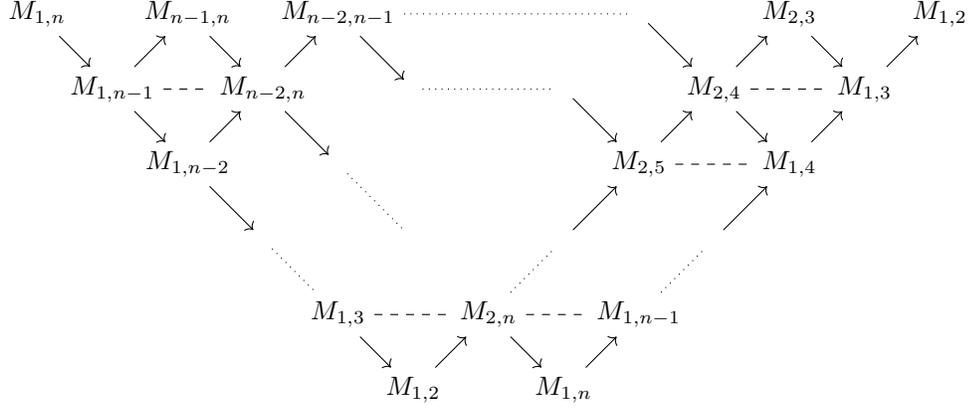
\begin{theorem}
\label{thm:ARexistence}
\begin{enumerate}
\item\label{thm:ARexistance-RvdB} \cite[Thm.\ 1.2.4]{RvdB02} Let $\mathcal{C}$ be a Hom-finite $k$-linear Krull--Schmidt triangulated category, where $k$ is a field.
Then $\mathcal{C}$ has a Serre functor if and only if $\mathcal{C}$ has Auslander--Reiten triangles.
\item\label{thm:ARexistanec-AR} Let $\mathcal{F}$ be a Frobenius exact category for which $\underline{\mathcal{F}}$ is Krull--Schmidt.
If $\underline{\mathcal{F}}$ has a Serre functor, then $\mathcal{F}$ has Auslander--Reiten sequences.
\end{enumerate}
\end{theorem}
\begin{proof}
For \ref{thm:ARexistanec-AR}, the category $\underline{\mathcal{F}}$ has Auslander--Reiten triangles by \ref{thm:ARexistance-RvdB}.
Therefore $\mathcal{F}$ has Auslander--Reiten sequences, by~\cite[Lem.~3]{roggenkamp96}.
\end{proof}

\begin{proposition}
\label{p:CMproperties}
\begin{enumerate}
\item\label{p:CMproperties-KS} \cite[Rem.\ 3.3]{JKS16} The categories $\C_{k,n}$ and $\underline{\C_{k,n}}$ are Krull--Schmidt.
\item\label{p:CMproperties-2CY} \cite[Prop.\ 2.11]{BBGEL},~\cite[Cor.\ 4.6]{JKS16},~\cite[Prop.\ 3.4]{GLS08}
The category $\underline{\C_{k,n}}$ is $2$-Calabi--Yau.
\end{enumerate}
\end{proposition}

We add some remarks on part \ref{p:CMproperties-2CY}. Let $Q_k$ denote the indecomposable injective module over the preprojective algebra of type $A_{n-1}$ on vertices $1,2,\ldots ,k-1$. By~\cite[Cor. 4.6]{JKS16}
(see Theorem~\ref{t:clusterstructure}\ref{t:clusterstructure-JKS}), $\Sub Q_k$ and $\C_{k,n}$ are stably equivalent, so $\underline{\C_{k,n}}$ is Hom-finite (as remarked in~\cite[Prop.\ 2.11]{BBGEL}).
The remaining $2$-Calabi--Yau property is then shown by combining~\cite[Cor.\ 4.6]{JKS16} with~\cite[Prop.\ 3.4]{GLS08}.

\begin{corollary}
\label{c:ARexistence}
Let $M$ be a rigid object in $\C_{k,n}$.
Then $\M$ has Auslander--Reiten sequences.
\end{corollary}
\begin{proof}
By Proposition~\ref{p:CMproperties}\ref{p:CMproperties-2CY},
$\C_{k,n}$ is stably $2$-Calabi--Yau.
Hence, by Proposition~\ref{p:Mperp-2cy-extri}, $\M$ is also stably $2$-Calabi--Yau; in particular, $\underline{\M}$ has a Serre functor.

By Proposition~\ref{p:CMproperties}\ref{p:CMproperties-KS}, $\C_{k,n}$ is Krull--Schmidt, and hence so is $\M$.
By Proposition~\ref{p:CMproperties}\ref{p:CMproperties-2CY},
$\underline{\C_{k,n}}$ is Hom-finite.
For objects $X,Y$ in $\M$, $\stabHom_{\M}(X,Y)$ is a quotient of $\stabHom_{\C_{k,n}}(X,Y)$ by Proposition~\ref{Prop:MperpisFrobenius-extri}, so $\underline{\M}$ is also Hom-finite.
Hence, by Theorem~\ref{thm:ARexistence}\ref{thm:ARexistanec-AR}, $\M$ has Auslander--Reiten sequences.
\end{proof}

By~\cite[Rem.~3.3]{JKS16}, $\C_{k,n}$ itself has Auslander--Reiten sequences. Taking $M=0$, Corollary~\ref{c:ARexistence} gives an alternative proof of this fact.

\subsection{Example: the Grassmannian \texorpdfstring{$\Gr(2,6)$}{Gr(2,6)}} \label{ss:Gr(2,6)}
We now focus on the example $\C_{2,6}$.
The Auslander--Reiten quiver of $\C_{2,6}$ is shown in Figure~\ref{f:ARquiver26} (with some objects drawn twice to give a better picture).
We take $M$ to be the indecomposable object $M_{14}$.
By Corollary~\ref{c:ARexistence}, $M_{14}^{\perp_1}$ has Auslander--Reiten sequences. Its indecomposable objects are shown in bold in Figure~\ref{f:ARquiver26}.

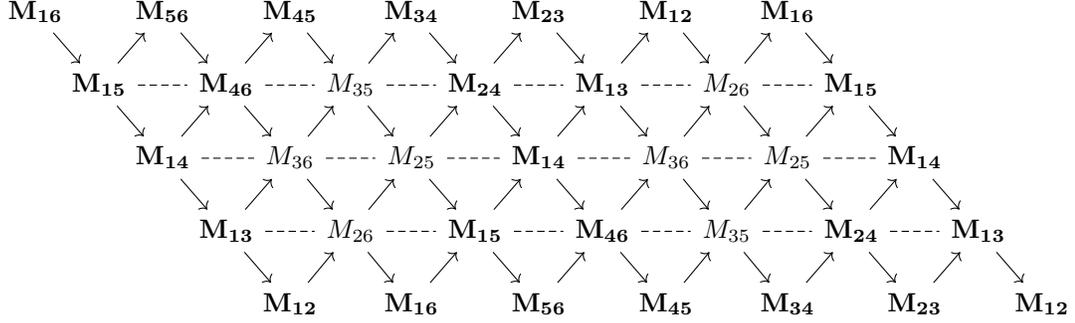
\begin{figure}
\makebox[\textwidth][c]{
\begin{tikzcd}[row sep=1.2em, column sep=-0.5em,ampersand replacement=\&]
\mathbf{M_{16}} \ar[dr] 
\&\& \mathbf{M_{56}}  \ar[dr] \&\& \mathbf{M_{45}}  \ar[dr] \&\& \mathbf{M_{34}} \ar[dr] \&\& \mathbf{M_{23}}  \ar[dr] \&\& \mathbf{M_{12}}  \ar[dr] \&\& \mathbf{M_{16}} \ar[dr] \\
\& \mathbf{M_{15}} \arrow[rr,dashed,no head] \ar[ur] \ar[dr] \&\& \mathbf{M_{46}} \arrow[rr,dashed,no head] \ar[ur] \ar[dr] \&\& M_{35} \arrow[rr,dashed,no head] \ar[ur] \ar[dr] \&\& \mathbf{M_{24}} \arrow[rr,dashed,no head] \ar[ur] \ar[dr] \&\& \mathbf{M_{13}} \arrow[rr,dashed,no head] \ar[ur] \ar[dr] \&\& M_{26} \arrow[rr,dashed,no head] \ar[ur] \ar[dr] \&\& \mathbf{M_{15}} \ar[dr] \\
\&\& \mathbf{M_{14}} \arrow[rr,dashed,no head] \ar[ur] \ar[dr] \&\& M_{36} \arrow[rr,dashed,no head] \ar[ur] \ar[dr] \&\& M_{25} \arrow[rr,dashed,no head] \ar[ur] \ar[dr] \&\& \mathbf{M_{14}} \arrow[rr,dashed,no head] \ar[ur] \ar[dr] \&\& M_{36} \arrow[rr,dashed,no head] \ar[ur] \ar[dr] \&\& M_{25} \arrow[rr,dashed,no head] \ar[ur] \ar[dr] \&\& \mathbf{M_{14}} \ar[dr] \\
\&\&\& \mathbf{M_{13}} \arrow[rr,dashed,no head] \ar[ur] \ar[dr] \&\& M_{26} \arrow[rr,dashed,no head] \ar[ur] \ar[dr] \&\& \mathbf{M_{15}} \arrow[rr,dashed,no head] \ar[ur] \ar[dr] \&\& \mathbf{M_{46}} \arrow[rr,dashed,no head] \ar[ur] \ar[dr] \&\& M_{35} \arrow[rr,dashed,no head] \ar[ur] \ar[dr] \&\& \mathbf{M_{24}} \arrow[rr,dashed,no head] \ar[ur] \ar[dr] \&\& \mathbf{M_{13}} \ar[dr] \\
\&\&\&\& \mathbf{M_{12}}  \ar[ur] \&\& \mathbf{M_{16}} \ar[ur] \&\& \mathbf{M_{56}}  \ar[ur] \&\& \mathbf{M_{45}}  \ar[ur] \&\& \mathbf{M_{34}}  \ar[ur] \&\& \mathbf{M_{23}} \ar[ur] \&\& \mathbf{M_{12}}  
\end{tikzcd}
}
\caption{The Auslander--Reiten quiver of $\C_{2,6}$.}
\label{f:ARquiver26}
\end{figure}

\begin{proposition}
The Auslander--Reiten quiver of $M_{14}^{\perp_1}$ in
$\C_{2,6}$ is as shown in Figure~\ref{f:ARquiverM14perpredrawn}.
\end{proposition}
\begin{proof}
The Auslander--Reiten sequences ending in $M_{46}$ and $M_{13}$ in $\C_{2,6}$ are entirely in $M_{14}^{\perp_1}$. Since $\Ext^1_{C_{2,6}}(M_{46},M_{15})$ and $\Ext^1_{C_{2,6}}(M_{13},M_{24})$ are $1$-dimensional, these sequences must be the Auslander--Reiten sequences in $M_{14}^{\perp_1}$. We also have the sequences
\begin{equation}
\begin{tikzcd}
0 \arrow{r}{} & M_{13} \arrow{r}{f} & M_{34}\oplus M_{12}
\arrow{r}{g} & M_{24} \arrow{r}{} & 0,
\end{tikzcd}
\label{e:ARMperp}
\end{equation}
where $f=\begin{pmatrix} \phi^{13}_{34} \\[3pt] \phi^{13}_{12} \end{pmatrix}$
and
$g=\begin{pmatrix} \phi^{34}_{24} &
-\phi^{12}_{24} \end{pmatrix}$,
and
\begin{equation}
\begin{tikzcd}
0 \arrow{r}{} & M_{46} \arrow{r}{h} & M_{45}\oplus M_{26}
\arrow{r}{k} & M_{15} \arrow{r}{} & 0,
\end{tikzcd}
\label{e:ARMperp2}
\end{equation}
where $h=\begin{pmatrix} \phi^{46}_{45} \\[3pt] \phi^{46}_{26} \end{pmatrix}$
and
$k=\begin{pmatrix} \phi^{45}_{15} &
-\phi^{26}_{15} \end{pmatrix}$.
Arguing as in the proof of Lemma~\ref{lem:ARsequence}, we can check that the sequences~\eqref{e:ARMperp} and~\eqref{e:ARMperp2} are exact and non-split.
Since
$\Ext^1_{C_{2,6}}(M_{24},M_{13})$ and
$\Ext^1_{C_{2,6}}(M_{15},M_{46})$ are $1$-dimensional, we see that \eqref{e:ARMperp} and \eqref{e:ARMperp2} must be Auslander--Reiten sequences in $M_{14}^{\perp_1}$.

We can also observe that there are irreducible maps $M_{45}\rightarrow M_{34}$ and $M_{12}\rightarrow M_{16}$ between projective-injective objects in $M_{14}^{\perp_1}$ which do not appear in any Auslander--Reiten sequence.
We thus have the Auslander--Reiten quiver of $M_{14}^{\perp_1}$ as shown in Figure~\ref{f:ARquiverM14perp}.
We have redrawn this in a nicer way: see Figure~\ref{f:ARquiverM14perpredrawn}.
\end{proof}

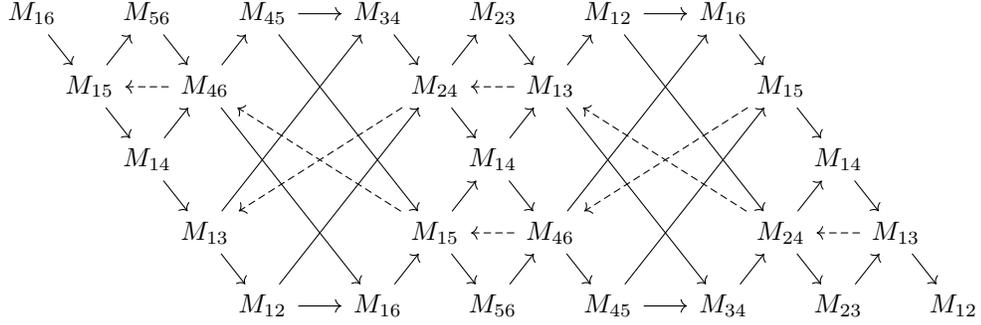
\begin{figure}
\begin{tikzcd}[row sep=1.2em, column sep=-0.5em]
M_{16} \ar[dr]  
&& M_{56}  \ar[dr] && M_{45} \ar[rr] \ar[dddrrr] && M_{34}  \ar[dr] && M_{23}  \ar[dr] && M_{12} \ar[rr] \ar[rrrddd] && M_{16} \ar[dr] \\
& M_{15} \ar[ur] \ar[dr] && M_{46}  \ar[dddrrr] \arrow[ll,dashed]  \ar[ur] &&  && M_{24} \arrow[ddllll,dashed] \ar[ur] \ar[dr] && M_{13} \ar[dddrrr] \arrow[ll,dashed] \ar[ur] &&  && M_{15} \arrow[ddllll,dashed] \ar[dr] \\
&& M_{14} \ar[ur] \ar[dr] && && && M_{14} \ar[ur] \ar[dr] &&  &&  && M_{14} \ar[dr] \\
&&& M_{13} \ar[uuurrr] \ar[dr] && \phantom{M_{26}} && M_{15} \arrow[lllluu,dashed]  \ar[ur] \ar[dr] && M_{46} \arrow[rrruuu]\arrow[ll,dashed]  \ar[dr] && \phantom{M_{35}} && M_{24} \arrow[uullll,dashed] \ar[ur] \ar[dr] && M_{13} \ar[dr] \arrow[ll,dashed] \\
&&&& M_{12} \ar[rr] \ar[uuurrr] && M_{16}  \ar[ur] && M_{56}  \ar[ur] && M_{45} \ar[rr] \ar[rrruuu] && M_{34}  \ar[ur] && M_{23}  \ar[ur] && M_{12}
\end{tikzcd}
\caption{The Auslander--Reiten quiver of $M_{14}^{\perp_1}$ in $\C_{2,6}$.}
\label{f:ARquiverM14perp}
\end{figure}

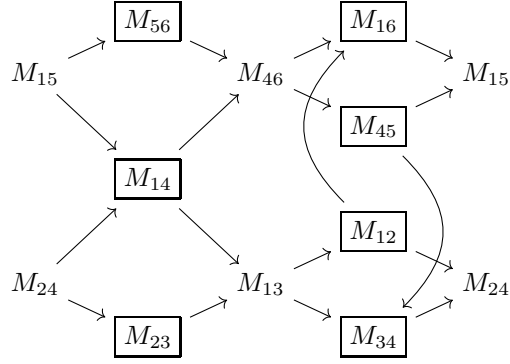
\begin{figure}
\[\begin{tikzpicture}[yscale=0.7,xscale=1.5]
\foreach \x/\y/\n/\l in {0/2/15l/15, 0/-2/24l/24, 2/2/46/46, 2/-2/13/13, 4/2/15r/15, 4/-2/24r/24}
{\draw (\x,\y) node (\n) {$M_{\l}$};}

\foreach \x/\y/\l in {1/3/56, 1/0/14, 1/-3/23,
3/3/16, 3/1/45, 3/-1/12, 3/-3/34}
{\draw (\x,\y) node (\l) {$\boxed{M_{\l}}$};}

\foreach \s/\t in {15l/56, 15l/14, 24l/14, 24l/23, 56/46, 14/46, 14/13, 23/13, 46/16, 46/45, 13/12, 13/34, 16/15r, 45/15r, 12/24r, 34/24r}
{\path [->] (\s) edge (\t);}
\path [->, bend left=25] (45) edge (34);
\path [->, bend left=25] (12) edge (16);

\end{tikzpicture}\]
\caption{The Auslander--Reiten quiver of $M_{14}^{\perp_1}$ in $\C_{2,6}$, redrawn.}
\label{f:ARquiverM14perpredrawn}
\end{figure}

\subsection{Example: the Grassmannian \texorpdfstring{$\Gr(3,6)$}{Gr(3,6)}} \label{ss:Gr(3,6)}
The subcategory $M_{236}^{\perp_1}$ has Auslander--
Reiten sequences (by an argument similar to that used 
for the $\Gr(2,6)$ example above).
In Figure~
\ref{f:ARquiverp145}, we show the Auslander--Reiten 
quiver of $M_{236}^{\perp_1}$.

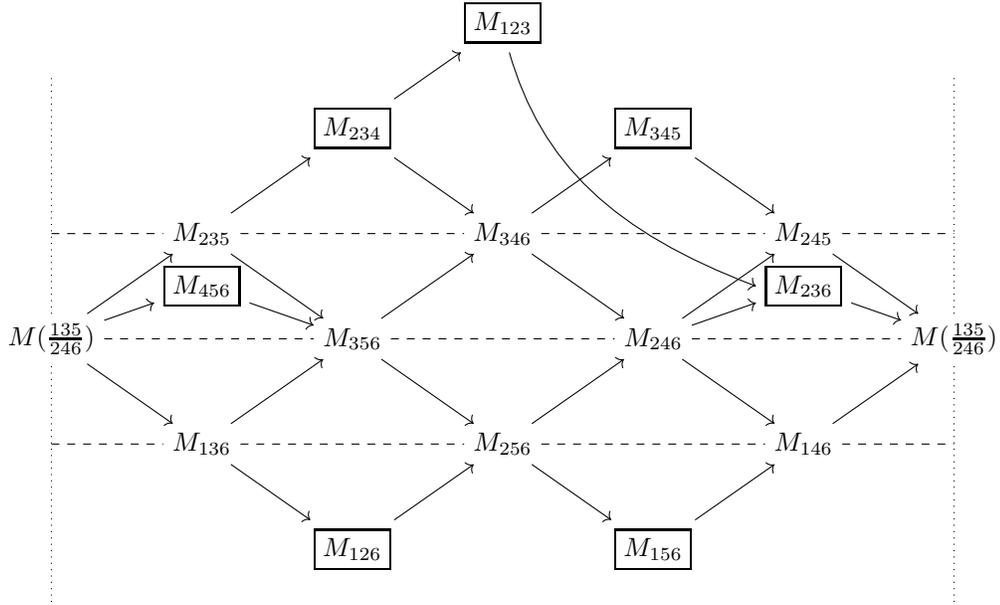
\begin{figure}
\[\begin{tikzpicture}[yscale=1.4,xscale=2]
\foreach \x/\y/\l in {1/1/235, 1/-1/136, 2/0/356, 3/1/346, 3/-1/256, 4/0/246, 5/1/245, 5/-1/146}
{\draw (\x,\y) node (\l) {$M_{\l}$};}
\draw (0,0) node (2l) {$M(\frac{135}{246})$};
\draw (6,0) node (2r) {$M(\frac{135}{246})$};

\foreach \x/\y/\l in {1/0.5/456, 2/2/234, 2/-2/126, 3/3/123, 4/2/345, 4/-2/156, 5/0.5/236}
{\draw (\x,\y) node (\l) {$\boxed{M_{\l}}$};}

\foreach \s/\t in {2l/235, 2l/456, 2l/136, 235/234, 235/356, 456/356, 136/356, 136/126, 234/123, 234/346, 356/346, 356/256, 126/256, 346/345, 346/246, 256/246, 256/156, 345/245, 246/245, 246/236, 246/146, 156/146, 245/2r, 236/2r, 146/2r}
{\path [->] (\s) edge (\t);}
\path [->, bend right=25] (123) edge (236.west);

\foreach \l/\r in {2l/356, 235/346, 136/256, 356/246, 346/245, 256/146, 246/2r}
{\path [dashed] (\l) edge (\r);}
\path [dashed] (0,1) edge (235);
\path [dashed] (0,-1) edge (136);
\path [dashed] (245) edge (6,1);
\path [dashed] (146) edge (6,-1);

\draw [dotted] (0,-2.5) -- (2l) -- (0,2.5);
\draw [dotted] (6,-2.5) -- (2r) -- (6,2.5);
\end{tikzpicture}\]
\caption{The Auslander--Reiten quiver of $M_{236}^{\perp_1}\subseteq\C_{3,6}$.}
\label{f:ARquiverp145}
\end{figure}

\section{Reduction of friezes from \texorpdfstring{$\C_{k,n}$}{\mathcal{C}\textunderscore k,n}}
\label{s:friezes}

One of our motivations for the study of reductions of Frobenius extriangulated categories was to study the effect on friezes associated to $\C_{k,n}$, in particular for the finite types $(2,n)$, for $n\geq 4$, and $(3,n)$, $n \in \{6,7,8\}$.
In these cases one can associate a \emph{mesh frieze} $\mc{F}_{k,n}$ to $\C_{k,n}$ \cite[Defn.~4.9]{BFGST}, and also a \emph{Pl\"ucker frieze} $\mc{P}_{k,n}$ and \emph{specialized Pl\"ucker frieze} $s\mc{P}_{k,n}$ \cite[Defn.~2.5, Defn.~2.9]{BFGST}.
Note that in the $k=2$ case the mesh frieze and the Pl\"ucker frieze coincide \cite[Eg.~4.10]{BFGST}.

Recall that we say that $\C_{k,n}$ has \emph{finite type} if it has finitely many indecomposable objects up to isomorphism.
Assuming without loss of generality that $k\leq n/2$, the category $\C_{k,n}$ has finite type if and only if $k=2$, or $k=3$ and $n=6$, $7$ or $8$, by~\cite[\S3]{JKS16}.

Let $\Lambda$ denote the preprojective algebra of type $A_{n-1}$ on vertices $1,2,\dots ,n-1$. For $1\leq k\leq n-1$, let $Q_k$ denote the indecomposable injective module over the preprojective algebra of type $A_{n-1}$.
In~\cite{GLS08}, the authors studied the Frobenius exact category $\Sub Q_k$ consisting of modules with socle supported only at vertex $k$ (and used it to categorify the cluster structure on the unipotent cell in the Grassmannian).

We will need the following results.

\begin{theorem}
\label{t:clusterstructure}
\begin{enumerate}
\item\label{t:clusterstructure-JKS} \cite[Prop.\ 4.6]{JKS16}
The Frobenius exact categories $\Sub Q_k$ and $\C_{k,n}$ are stably equivalent.
\item\label{t:clusterstructure-Pre22} \cite[Thm.\ 6.11]{Pre22} (see also~\cite{GLS06},~\cite{GLS08} and~\cite{BIRS09}).
The cluster-tilting objects in $\underline{\C_{k,n}}\simeq \stabSub Q_k$ form a cluster structure (in the sense of~\cite[\S II.1]{BIRS09}).
\end{enumerate}
\end{theorem}
\begin{note} A key part of checking \ref{t:clusterstructure-Pre22} includes verifying that the endomorphism algebras of cluster-tilting objects have no loops or $2$-cycles; this was done for $\C_{k,n}$ in~\cite[Thm.\ 4.2]{JKS2} and~\cite[Prop.\ 5.17]{Pre20} (see also \cite{Pre21C}).
\end{note}

We see in~\cite{JKS16} that, in the finite type cases, the stable Auslander--Reiten quiver of $\C_{k,n}$ is the same as that of the cluster category of the same cluster type as $\Gr(k,n)$. In fact, more is true.

\begin{theorem}
\label{t:stablecluster}
Let $\C_{k,n}$ be a Grassmannian cluster category of finite type, where $k\leq n/2$. Then the stable category $\underline {\C_{k,n}}$ is triangle equivalent to the cluster category (in the sense of~\cite{BMRRT06}) of the same cluster type as $\Gr(k,n)$.
\end{theorem}
\begin{proof}
A $(k,n)$-alternating strand diagram (or Postnikov diagram) $D$ (in the sense of~\cite[\S 14]{Postnikov}) gives rise to a cluster for $\Gr(k,n)$ by~\cite[Thm.\ 3]{Scott06} and hence a cluster-tilting object $T$ in $\C_{k,n}$ by~\cite[Rem.\ 9.6]{JKS16}. By~\cite[Thm.\ 11.2]{BKM16}, we obtain a description of $\End_{C_{k,n}}(T)^\opp$ as a completed dimer algebra given by the quiver $Q(D)$ of $D$ together with potential $W$ which is the sum of the minimal anticlockwise cycles in $Q$ minus the sum of the minimal clockwise cycles.

By~\cite[Cor.~4.4]{Pasquali20}, the stable endomorphism algebra $\stabEnd_{C_{k,n}}(T)^\opp$ is isomorphic to the uncompleted Jacobian algebra defined by the quiver $\underline{Q}$, obtained by removing the boundary vertices from $Q$, and potential $\underline{W}$, the image of $W$ under the quotient map $\mathbb{C}Q\rightarrow \mathbb{C}\underline{Q}$.
Thus $\underline{W}$ is just $W$ with the cycles containing a boundary vertex removed.

Assume first that $k=2$, so $n\geq 4$.
A Postnikov diagram for $\Gr(2,n)$ can be constructed using Scott's construction~\cite[\S 3]{Scott06} applied to a fan triangulation of a regular $(n+3)$-sided polygon (i.e.\ a triangulation in which all of the diagonals are incident with a fixed vertex).
In this case, the quiver $\underline{Q}$ is the linearly oriented quiver of type $A_{n-3}$.

Thus $\underline{\C_{2,n}}$ is a triangulated $2$-Calabi--Yau category with a cluster-tilting object whose endomorphism algebra is a quiver of Dynkin type $A_{n-3}$.
By Proposition~\ref{p:CMproperties}, $\underline{\C_{2,n}}$ is also Hom-finite and Krull--Schmidt. Hence it is equivalent to the cluster category of type $A_{n-3}$ by~\cite[Cor.\ 2.1]{KellerReiten08}.

For $\Gr(3,6)$, an alternating strand diagram is given in~\cite[Fig.\ 4]{MR20}. In this case, the quiver $\underline{Q}$ is an orientation of the $D_4$ quiver, giving the result in this case by arguing as above.

For $\Gr(3,7)$, an alternating strand diagram is given 
in~\cite[Fig.\ 13]{MR20}. In this case, the quiver $\underline{Q}$ is shown in Figure~\ref{f:Quiver37}.
Mutating this quiver at the vertex $4$ gives an orientation of the Dynkin quiver of type $E_6$. Since the cluster-tilting objects in $\underline{\C_{3,7}}$ form a cluster structure (Theorem~\ref{t:clusterstructure}\ref{t:clusterstructure-Pre22}), there is a cluster-tilting object in $\underline{\C_{3,7}}$ whose endomorphism algebra has quiver of type $E_6$.
The result in this case now follows from~\cite[Cor.\ 2.1]{KellerReiten08} as above.

For $\Gr(3,8)$, an alternating strand diagram is given 
in~\cite[Fig.\ 11]{Scott06}. In this case, the quiver $\underline{Q}$ is shown in Figure~\ref{f:Quiver38}, and Scott gives a sequence of mutations ($2$, $6$, $3$, $4$, $8$, $1$, $7$, $6$, $5$, $3$, $4$, $5$ in our numbering) resulting in an orientation of the Dynkin quiver of type $E_8$.
Since the cluster-tilting objects in $\underline{\C_{3,8}}$ form a cluster structure (Theorem~\ref{t:clusterstructure}\ref{t:clusterstructure-Pre22}), there is a cluster-tilting object in $\underline{\C_{3,8}}$ whose endomorphism algebra has quiver of type $E_8$.
The result in this case now follows from~\cite[Cor.\ 2.1]{KellerReiten08} as above.
\end{proof}

\begin{figure}
\[\begin{tikzpicture}[yscale=0.5,xscale=0.7]
\foreach \x/\y/\n in {0/2/1, 2/2/2, 4/2/3, 0/0/4, 2/0/5, 4/0/6}
{\draw (\x,\y) node (\n) {${\n}$};}

\foreach \s/\t in {2/1, 2/3, 5/4, 6/5, 1/4, 2/5, 4/2}
{\path [->] (\s) edge (\t);}
\end{tikzpicture}\]
\caption{The quiver $\underline{Q}$ for the alternating strand diagram for $\Gr(3,7)$ in~\cite[Fig.\ 13]{MR20}.}
\label{f:Quiver37}
\end{figure}
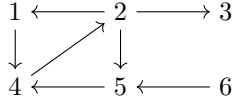

\begin{figure}
\[\begin{tikzpicture}[yscale=0.5,xscale=0.7]
\foreach \x/\y/\n in {0/2/1, 2/2/2, 4/2/3, 6/2/4, 0/0/5, 2/0/6, 4/0/7, 6/0/8}
{\draw (\x,\y) node (\n) {${\n}$};}

\foreach \s/\t in {1/2, 2/6, 6/5, 5/1, 3/2, 6/7, 7/3, 3/4, 4/8, 8/7}
{\path [->] (\s) edge (\t);}
\end{tikzpicture}\]
\caption{The quiver $\underline{Q}$ for the alternating strand diagram for $\Gr(3,8)$ in~\cite[Fig.\ 11]{Scott06}.}
\label{f:Quiver38}
\end{figure}

\begin{proposition}
\label{p:brick}
Let $\C(Q)$ be the cluster category of a Dynkin quiver $Q$, and let $M\in\C(Q)$ be an indecomposable object. Then $\Hom_{\C(Q)}(M,M)\cong\CC$. In particular, $\stabHom_{\C_{k,n}}(M,M)\cong\CC$ whenever $\C_{k,n}$ is a Grassmannian cluster category of finite type and $M\in\C_{k,n}$ is indecomposable and non-projective.
\end{proposition}
\begin{proof}
Let $\mathcal{C}(Q)$ be the cluster category of the path algebra $\mathbb{C}Q$, where $Q$ is a Dynkin quiver, and let $X$ be an indecomposable $\mathbb{C}Q$-module.
Then $\Hom_{\mathbb{C}Q}(X,X)\cong \mathbb{C}$ (this follows, for example, from~\cite[VII.1.5(b), IV.2.15(c)]{ASS06}, since every indecomposable module is preprojective). By~\cite[Prop.\ 1.5]{BMRRT06}, denoting by $[1]$ the shift in the bounded derived category $D^\bdd(\mathbb{C}Q)$ of $\mathbb{C}Q$, we have
\begin{align*}
\Hom_{\mathcal{C}(Q)}(X,X) &\cong \Hom_{\mathbb{C}Q}(X,X)\oplus \Hom_{\mathbb{C}Q}(X,\tau^{-1}X[1]) \\
&\cong \Hom_{\mathbb{C}Q}(X,X)\oplus \Hom_{\mathbb{C}Q}(X,\tau^2 X) \\
&\cong \Hom_{\mathbb{C}Q}(X,X)\\
&\cong\CC,
\end{align*}
using $\Ext^1_{\mathbb{C}Q}(X,Y)\cong \Hom_{D^\bdd(\mathbb{C}Q)}(X,Y[1])$ and the Auslander--Reiten formula.

Now if $\C_{k,n}$ is a Grassmannian cluster category of finite type, and $M\in\C_{k,n}$ is indecomposable and non-projective, then $M$ is also indecomposable in the stable category $\underline{\C_{k,n}}$.
By Theorem~\ref{t:stablecluster}, there is a triangle equivalence $\underline{\C_{k,n}}\simeq\C(Q)$ for a Dynkin quiver $Q$, and the result follows.
\end{proof}

\begin{remark}
One can in fact show for any $k$ and $n$ that $\underline {\C_{k,n}}$ is equivalent to Amiot's generalised cluster category \cite{Amiot} associated to the quiver with potential $(\underline{Q},\underline{W})$.
Since this combines a range of results in the literature, involving concepts we will not otherwise need, we give only a brief sketch.
The relative Ginzburg algebra of the ice quiver with potential associated to a $(k,n)$-Postnikov diagram is concentrated in degree $0$ by \cite[Thm.~3.7]{Pre22} and \cite[Lem.~5.11]{Wu-Mutation}.
Thus its Higgs category coincides with the category of Gorenstein projective modules over its boundary algebra by \cite[Thm.~4.17]{KellerWu}, and the stable Higgs category with the relevant generalised cluster category by \cite[Cor.~4.9, Thm.~4.15]{KellerWu}. Note here that the relevant quiver with potential $(\underline{Q},\underline{W})$ is Jacobi-finite by \cite[Prop.~4.4]{Pre22}. 
Finally, this boundary algebra is isomorphic to $C_{k,n}$ by \cite[Thm.~11.2]{BKM16}, and $\GP(C_{k,n})=\C_{k,n}$ by \cite[Cor.~3.7]{JKS16}.
\end{remark}

Recall that $\underline{\C_{k,n}}$ is $2$-Calabi--Yau, so in particular has a Serre functor $S=[2]$. By~\cite[I.2.3, I.2.4]{RvdB02}, we have $S=\tau\circ [1]$ in $\underline{\C_{k,n}}$, where $\tau$ is the Auslander--Reiten translate. Hence
\begin{equation}
\label{e:ARformula}
\stabHom_{C_{k,n}}(X,Y)\cong \Ext^1_{C_{k,n}}(Y,\tau X)
\end{equation}
for all objects $X,Y$ in $\underline{\C_{k,n}}$.
By Proposition~\ref{p:Mperp-2cy-extri}, the category $\underline{\M}$ also has a Serre functor (since it is also $2$-Calabi--Yau), giving us
\begin{equation}
\label{e:ARformula2}
\stabHom_{\M}(X,Y)\cong \Ext^1_{\M}(Y,\tau_{\M} X).
\end{equation}
Note that while $\Ext^1_{\M}(Y,\tau_{\M} X)=\Ext^1_{C_{k,n}}(Y,\tau_{\M} X)$ by Lemma~\ref{l:closed-extri}, the objects $\tau X$ and $\tau_{\M}X$ may be non-isomorphic. Similarly, $\stabHom_{C_{k,n}}(X,Y)$ and $\stabHom_{\M}(X,Y)$ may differ since $\M$ typically has more projective-injective objects than $\underline{\C_{k,n}}$. Nevertheless, we have the following result.

\begin{proposition}
\label{p:Extonedim}
Suppose that $\C_{k,n}$ has finite type.
Let $X$ be a non-projective indecomposable object in $\underline{\C_{k,n}}$.
Then $\Ext^1_{C_{k,n}}(X,\tau X)\cong \mathbb{C}$. Furthermore, if $M$ is a rigid object in $\C_{k,n}$, then $\Ext^1_{\M}(X,\tau_{\M}X)\cong \mathbb{C}$.
\end{proposition}
\begin{proof}
We have $\stabHom_{C_{k,n}}(X,X)\cong \mathbb{C}$ by Proposition~\ref{p:brick}, so the result for
$\C_{k,n}$ follows from~\eqref{e:ARformula}.
Moreover, $\stabHom_{\M}(X,X)$ is a quotient of $\stabHom_{C_{k,n}}(X,X)\cong \mathbb{C}$ by Proposition~\ref{Prop:MperpisFrobenius-extri}, and so, since it cannot be zero, we must have $\stabHom_{\M}(X,X)\cong \mathbb{C}$.
Hence $\Ext^1_{\M}(X,\tau_{\M} X)\cong \mathbb{C}$ by~\eqref{e:ARformula2}.
\end{proof}

\begin{definition} \label{d:frieze} \cite[Defn.\ 4.9]{BFGST}
Suppose that $\C_{k,n}$ is a Grassmannian cluster category of finite type. A \emph{mesh frieze} $F$ for $\C_{k,n}$ is a positive integer $F(M)$ for each indecomposable object in $\C_{k,n}$ such that $F(P)=1$ for every projective object $P$, and for every Auslander--Reiten sequence
\[\begin{tikzcd}
0\arrow{r}& \tau X\arrow{r}& \bigoplus_{i=1}^l E_i \arrow{r}& X\arrow{r}& 0
\end{tikzcd}\]
in $\C_{k,n}$, we have
$$F(X)F(\tau X)=\prod_{i=1}^l F(E_i)+1.$$
\end{definition}

This definition extends verbatim to finite type exact categories with Auslander--Reiten sequences, such as $\M$ for $M\in\C_{k,n}$ a rigid object.
We can now give an alternative proof of~\cite[Prop.\ 5.3]{BFGST}.
For this we use the cluster character as discussed in Section~\ref{s:CC}; note that $\C_{k,n}$ is Krull--Schmidt by \cite[Rem.~3.3]{JKS16}, since we took took the completion when defining the algebra $C_{k,n}$.

\begin{theorem} \cite[Prop. 5.3]{BFGST} \label{Thm:friezered}
Suppose that $\C_{k,n}$ is of finite type.
Let $F$ be a mesh frieze on $\C_{k,n}$.
Let $M$ be a rigid indecomposable non-projective object in $\C_{k,n}$, and suppose that $F(M)=1$.
Then $F\vert_{\M}$ is a mesh frieze for $\M$.
\end{theorem}
\begin{proof}
Let $T=\bigoplus_{i=1}^NT_i$ be a cluster-tilting object in $\C_{k,n}$ such that the non-projective summands form a slice in the Auslander--Reiten quiver of $\C_{k,n}$.
Let $\Phi_{\C_{k,n}}^T$ be the corresponding cluster character, as defined in~\eqref{eq:WWZCC}; this coincides with the Fu--Keller cluster character by Proposition~\ref{p:FKCC}.
Let $\chi$ be the function on indecomposable objects in $\C_{k,n}$ obtained by specialising $x_i=F(T_i)$ in the expression~\eqref{eq:WWZCC2} for $\Phi_{\C_{k,n}}$.

By Proposition~\ref{p:Extonedim}, $\Ext^1_{C_{k,n}}(X,\tau X)\cong \mathbb{C}$ for every non-projective indecomposable object in $\C_{k,n}$, which means that $\chi$ is a mesh frieze on $\C_{k,n}$, by the definition of a cluster character~\cite[Defn.\ 1.2]{Palu} (see~\cite[\S 2.5]{FuKeller} for the exact case). Since the values of a mesh frieze on $\C_{k,n}$ are determined by the values on a slice, we see that $\chi$ and $F$ coincide.

By Theorem~\ref{t:cc_restrict}, the function
$\Phi_{\C_{k,n}}^T\vert_{\M}$ is a cluster character on $\M$. Note also that $F(M)=1$, so $F$ takes the value $1$ on the projective indecomposable objects of $\M$ (see Proposition~\ref{Prop:MperpisFrobenius-extri}). By Proposition~\ref{p:Extonedim} it follows that $F\vert_{\M}=\chi\vert_{\M}$ is a mesh frieze on $\M$ as required.
\end{proof}

\begin{example} \label{Ex:CCfriezered}  Consider the example $\C_{2,6}$ from Section \ref{ss:Gr(2,6)}. The Auslander--Reiten quiver is shown in Figure~\ref{f:ARquiver26}. Choose the cluster tilting object  $T$ with non-projective direct summands $M_{13} \oplus M_{14} \oplus M_{15}$, that is $T=M_{13} \oplus M_{14} \oplus M_{15} \oplus \big(\bigoplus_{i=1}^6 M_{i,i+1}\big)$ in  $\C_{2,6}$. As above, by specialising $x_i=1$ in $\Phi_{\C_{2,6}}^T$ one obtains the mesh frieze (which is the same as a specialised Pl\"ucker frieze $s\mc{P}_{2,6}$) in Figure~\ref{f:Frieze26}.
\begin{figure}
\begin{tikzcd}[row sep=1.2em, column sep=-0.5em]
\col{\mathbf{1}}&&\col{\mathbf{1}} &&\col{\mathbf{1}} && \mathbf{1} && \mathbf{1}   && \mathbf{1}   && \mathbf{1}  && \mathbf{1}   && \mathbf{1}   && \mathbf{1} &&\col{\mathbf{1}} && \col{\mathbf{1}} && \col{\mathbf{1}}\\
&\col{\mathbf{2}}&& \col{\mathbf{1}}&& \col{4} && \mathbf{1}     && \mathbf{2}     && 2     && \mathbf{2}     && \mathbf{1}     && 4     && \mathbf{1} && \col{\mathbf{2}} && \col{2}  && \col{\mathbf{2}}\\
&& \col{\mathbf{1}}&&\col{3}&&\col{3}&& \mathbf{1}     && 3     && 3     && \mathbf{1}     && 3    && 3     && \mathbf{1} &&\col{3} &&\col{3} &&\col{\mathbf{1}} \\
&&& \col{\mathbf{2}} &&\col{2}&&\col{\mathbf{2}}&& \mathbf{1}     && 4     && \mathbf{1}     && \mathbf{2}     && 2     && \mathbf{2}     && \mathbf{1} && \col{4} && \col{\mathbf{1}} && \col{\mathbf{2}}\\
&&&& \col{\mathbf{1}} &&\col{\mathbf{1}}&&\col{\mathbf{1}} && \mathbf{1}   && \mathbf{1}  && \mathbf{1}   && \mathbf{1}   && \mathbf{1}   && \mathbf{1}  && \mathbf{1}  && \col{\mathbf{1}} &&\col{\mathbf{1}} && \col{\mathbf{1}}
\end{tikzcd}
\caption{A mesh frieze on $\C_{2,6}$.}
\label{f:Frieze26}
\end{figure}
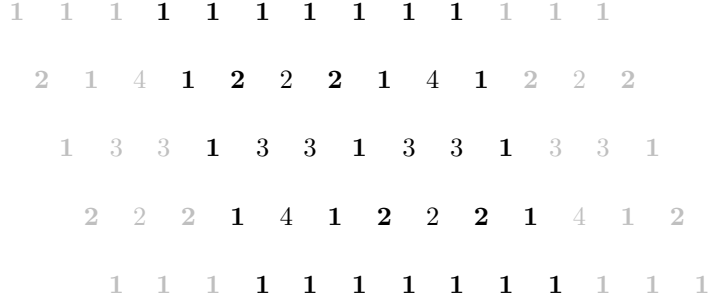

This can be seen combinatorially as follows: there is a correspondence between the modules $M_{ij}$ of $\C_{2,6}$ and the arcs $(i,j)$ of a hexagon with vertices $1, \ldots, 6$, as shown in Figure~\ref{f:26arcs}.
The boundary arcs $(i,i+1)$ thus correspond to the projective modules $M_{i,i+1}$, and the cluster-tilting object $T$ corresponds to the triangulation highlighted in Figure~\ref{f:26arcs} (left).
The entries in the frieze satisfy the Ptolemy relations, see e.g.~\cite[Defn.~3.1]{CuntzHolmJorgensen}.  

On the other hand, the category $M^{\perp_1}_{14}$ is obtained from $\C_{2,6}$ by deleting all $M \in \C_{2,6}$ such that $\Ext^1(M,M_{14}) \neq 0$, making $M_{14}$ projective.
The reduction, shown in bold in Figure~\ref{f:Frieze26}, satisfies the frieze relations, by Theorem~\ref{Thm:friezered}.
The colours in Figure~\ref{f:redFrieze26} indicate how one can see the reduced frieze as two Conway--Coxeter friezes of rank $1$ glued together at the entry corresponding to the module $M_{14}$.
In the combinatorial picture this corresponds to ``freezing'' the arc $(1,4)$, so the hexagon becomes two quadrilaterals glued along this arc $(1,4)$.
Further, one removes all arcs that intersect $(1,4)$ from the hexagon; see Figure~\ref{f:26arcs} (right).
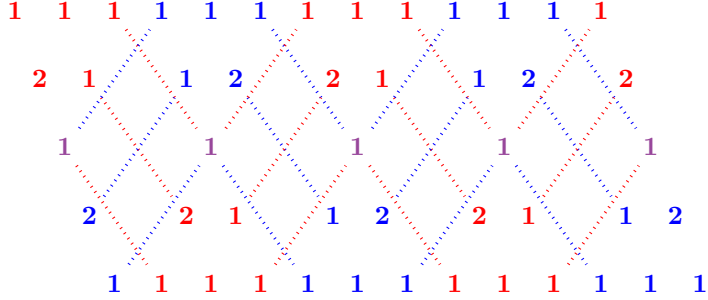
\begin{figure}
\begin{tikzcd}[row sep=1.2em, column sep=-0.5em]
\ye{\mathbf{1}}&&\ye{\mathbf{1}} &&\ye{\mathbf{1}} \arrow[ddrrrr,dotted, very thick, -,color=red]&& \bl{\mathbf{1}} && \bl{\mathbf{1}}   && \bl{\mathbf{1}} \arrow[ddrrrr,dotted, very thick, -,color=blue]  && \ye{\mathbf{1}}  && \ye{\mathbf{1}}   && \ye{\mathbf{1}} \arrow[ddrrrr,dotted, very thick, -,color=red]  && \bl{\mathbf{1}} &&\bl{\bl{\mathbf{1}}} && \bl{\mathbf{1}} \arrow[ddrrrr,dotted, very thick, -,color=blue] && \ye{\mathbf{1}}\\
&\ye{\mathbf{2}}&& \ye{\mathbf{1}} \arrow[ddrrrr,dotted, very thick, -,color=red]&& \phantom{\bl{4}} && \bl{\mathbf{1}}     && \bl{\mathbf{2}} \arrow[ddrrrr,dotted, very thick, -,color=blue]    && \phantom{2}     && \ye{\mathbf{2}}     && \ye{\mathbf{1}}   \arrow[ddrrrr,dotted, very thick, -,color=red]  && \phantom{4}     && \bl{\mathbf{1}} && \bl{\mathbf{2}} \arrow[ddrrrr,dotted, very thick, -,color=blue]&& \phantom{\bl{2}}  && \ye{\mathbf{2}}\\
&& \gr{\mathbf{1}} \arrow[ddrrrr,dotted, very thick, -,color=red]  \arrow[uurrrr,dotted, very thick, -,color=blue] &&\phantom{\col{3}}&&\phantom{\col{3}}&& \gr{\mathbf{1}}  \arrow[ddrrrr,dotted, very thick, -,color=blue]  \arrow[uurrrr,dotted, very thick, -,color=red]   && \phantom{3}     && \phantom{3}     && \gr{\mathbf{1}}  \arrow[ddrrrr,dotted, very thick, -,color=red]  \arrow[uurrrr,dotted, very thick, -,color=blue]   && \phantom{3}    && \phantom{3}     && \gr{\mathbf{1}} \arrow[ddrrrr,dotted, very thick, -,color=blue]  \arrow[uurrrr,dotted, very thick, -,color=red] &&\phantom{\col{3}} &&\phantom{\col{3}} &&\col{\gr{\mathbf{1}}} \\
&&& \bl{\mathbf{2}} \arrow[uurrrr,dotted, very thick, -,color=blue] &&\phantom{\ye{2}}&&\ye{\mathbf{2}}&& \ye{\mathbf{1}}   \arrow[uurrrr,dotted, very thick, -,color=red]    && \phantom{4}     && \bl{\mathbf{1}}     && \bl{\mathbf{2}}    \arrow[uurrrr,dotted, very thick, -,color=blue]   && \phantom{2}     && \ye{\mathbf{2}}     && \ye{\mathbf{1}}  \arrow[uurrrr,dotted, very thick, -,color=red] && \phantom{\ye{4}} && \bl{\mathbf{1}} && \bl{\mathbf{2}}\\
&& && \bl{\mathbf{1}} \arrow[uurrrr,dotted, very thick, -,color=blue]  &&\ye{\ye{\mathbf{1}}}&&\ye{\ye{\mathbf{1}}} && \ye{\mathbf{1}}   \arrow[uurrrr,dotted, very thick, -,color=red]  && \bl{\mathbf{1}}  && \bl{\mathbf{1}}   && \bl{\mathbf{1}}  \arrow[uurrrr,dotted, very thick, -,color=blue]   && \ye{\mathbf{1}}   && \ye{\mathbf{1}}  && \ye{\mathbf{1}}  \arrow[uurrrr,dotted, very thick, -,color=red]  && \bl{\mathbf{1}} &&\bl{\mathbf{1}} && \bl{\mathbf{1}}
\end{tikzcd}
\caption{Reduction of the frieze from Figure~\ref{f:Frieze26} to $M^{\perp_1}_{14}$.}
\label{f:redFrieze26}
\end{figure}
\begin{figure}%
    \begin{tikzpicture}[scale=3,cap=round,>=latex]

    \node  at (90:0.65) {$1$};
        \node   at (30:0.65) {$2$};
    \node  at (330:0.65) {$3$};
        \node at (270:0.65) {$4$};
    \node at (210:0.65) {$5$};
        \node at (150:0.65) {$6$};
         
    \coordinate  (p1) at (90:0.5) {}; 
        \coordinate  (p2) at (30:0.5){};
   \coordinate (p3) at (330:0.5){};
    \coordinate (p4) at (270:0.5){};
   \coordinate (p5) at (210:0.5){}; 
    \coordinate (p6) at (150:0.5){};

 \draw[fill=black!5] (p1)--(p2)--(p3)--(p4)--(p5)--(p6)--(p1); 

\draw[-,thick] (p1)--(p3); 
\draw[-,thick] (p1)--(p4); 
\draw[-,thick] (p1)--(p5); 
  
   \draw[-,black,very thick] (p1)--(p2); 
      \draw[-,very thick] (p2)--(p3); 
      \draw[-,very thick] (p3)--(p4); 
            \draw[-,very thick] (p4)--(p5); 
       \draw[-,very thick] (p5)--(p6); 
         \draw[-,very thick] (p6)--(p1); 

             \draw[dotted] (p2)--(p4); 
           \draw[dotted] (p2)--(p5);
           \draw[dotted] (p2)--(p6); 
             \draw[dotted] (p2)--(p5);             
              \draw[dotted] (p3)--(p5); 
             \draw[dotted] (p3)--(p6);  
              \draw[dotted] (p4)--(p6);

	\draw (p1) node[fill=black,circle,inner sep=0.039cm] {} circle (0.01cm);	        
     \draw (p2) node[fill=black,circle,inner sep=0.039cm] {} circle (0.01cm);
     \draw (p3) node[fill=black,circle,inner sep=0.039cm] {} circle (0.01cm);
     \draw (p4) node[fill=black,circle,inner sep=0.039cm] {} circle (0.01cm);
     \draw (p5) node[fill=black,circle,inner sep=0.039cm] {} circle (0.01cm);
     \draw (p6) node[fill=black,circle,inner sep=0.039cm] {} circle (0.01cm);
    \end{tikzpicture}
    \ \hspace{2cm} \
     \begin{tikzpicture}[scale=3,cap=round,>=latex]

    \node  at (90:0.65) {$1$};
        \node   at (30:0.65) {$2$};
    \node  at (330:0.65) {$3$};
        \node at (270:0.65) {$4$};
    \node at (210:0.65) {$5$};
        \node at (150:0.65) {$6$};
         
    \coordinate  (p1) at (90:0.5) {}; 
        \coordinate  (p2) at (30:0.5){};
   \coordinate (p3) at (330:0.5){};
    \coordinate (p4) at (270:0.5){};
   \coordinate (p5) at (210:0.5){}; 
    \coordinate (p6) at (150:0.5){};

 \draw[fill=black!5] (p1)--(p2)--(p3)--(p4)--(p5)--(p6)--(p1); 

\draw[-,thick, red] (p1)--(p3); 
\draw[-,thick, blue] (p1)--(p5);   

   \draw[-,red,very thick] (p1)--(p2); 
      \draw[-,red,very thick] (p2)--(p3); 
      \draw[-,red,very thick] (p3)--(p4); 
            \draw[-,blue,very thick] (p4)--(p5); 
       \draw[-,blue,very thick] (p5)--(p6); 
         \draw[-,blue,very thick] (p6)--(p1); 
             \draw[-,Purple,very thick] (p1)--(p4);  %

             \draw[dotted,red] (p2)--(p4); 
              \draw[dotted,blue] (p4)--(p6);

	\draw (p1) node[fill=black,circle,inner sep=0.039cm] {} circle (0.01cm);	        
     \draw (p2) node[fill=black,circle,inner sep=0.039cm] {} circle (0.01cm);
     \draw (p3) node[fill=black,circle,inner sep=0.039cm] {} circle (0.01cm);
     \draw (p4) node[fill=black,circle,inner sep=0.039cm] {} circle (0.01cm);
     \draw (p5) node[fill=black,circle,inner sep=0.039cm] {} circle (0.01cm);
     \draw (p6) node[fill=black,circle,inner sep=0.039cm] {} circle (0.01cm);
    \end{tikzpicture}
\caption{Combinatorial picture for $\C_{2,6}$, with triangulation corresponding to $M_{13} \oplus M_{14} \oplus M_{15}$ (left), reduced to $M^{\perp}_{14}$ (right). The colours of the arcs on the right-hand side correspond to the colours in the frieze in Figure~\ref{f:redFrieze26}.}
\label{f:26arcs}
\end{figure}
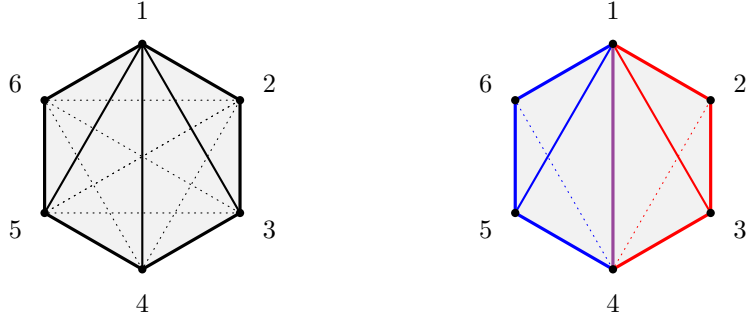

\end{example}

\begin{example} \label{Ex:redfriezecoeff} Consider again $\C_{2,6}$, but now take the cluster tilting object $T'=M_{26}\oplus M_{36} \oplus M_{46}\oplus  \big(\bigoplus_{i=1}^6 M_{i,i+1}\big)$.
The Conway--Coxeter frieze (Figure~\ref{f:Frieze26differentT}) contains the same numbers as the one from Example \ref{Ex:CCfriezered}, but all diagonals are shifted to the right.
\begin{figure}%
\begin{tikzcd}[row sep=1.2em, column sep=-0.5em]
\col{\mathbf{1}}&&\col{\mathbf{1}} &&\col{\mathbf{1}} && \mathbf{1} && \mathbf{1}   && \mathbf{1}   && \mathbf{1}  && \mathbf{1}   && \mathbf{1}   && \mathbf{1} &&\col{\mathbf{1}} && \col{\mathbf{1}} && \col{\mathbf{1}}\\
&\col{\mathbf{2}}&& \col{\mathbf{2}}&& \col{1} && \mathbf{4}     && \mathbf{1}     && 2     && \mathbf{2}     && \mathbf{2}     && 1     && \mathbf{4} && \col{\mathbf{1}} && \col{2}  && \col{\mathbf{2}}\\
&& \col{\mathbf{3}}&&\col{1}&&\col{3}&& \mathbf{3}     && 1     && 3     && \mathbf{3}     && 1    && 3     && \mathbf{3} &&\col{1} &&\col{3} &&\col{\mathbf{3}} \\
&&& \col{\mathbf{1}} &&\col{2}&&\col{\mathbf{2}}&& \mathbf{2}     && 1    && \mathbf{4}     && \mathbf{1}     && 2     && \mathbf{2}     && \mathbf{2} && \col{1} && \col{\mathbf{4}} && \col{\mathbf{1}}\\
&&&& \col{\mathbf{1}} &&\col{\mathbf{1}}&&\col{\mathbf{1}} && \mathbf{1}   && \mathbf{1}  && \mathbf{1}   && \mathbf{1}   && \mathbf{1}   && \mathbf{1}  && \mathbf{1}  && \col{\mathbf{1}} &&\col{\mathbf{1}} && \col{\mathbf{1}}
\end{tikzcd}
\caption{The frieze obtained from $\C_{2,6}$ with $T'=M_{26} \oplus M_{36} \oplus M_{46}$. }
\label{f:Frieze26differentT}
\end{figure}
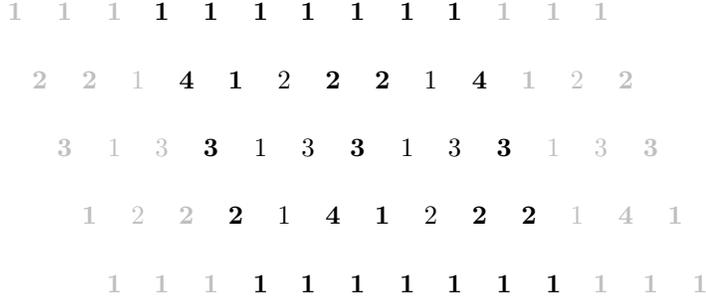

One can again reduce the frieze to the subcategory $M^{\perp_1}_{14}$, but now Theorem~\ref{Thm:friezered} does not apply, since $\Phi^{T'}_{\C_{2,6}}(M_{14}) =3 \neq 1$.
However, the reduction of (Figure~\ref{f:redFrieze26differentT}) is a frieze with coefficients in the sense of Cuntz--Holm--J{\o}rgensen \cite[Defn.~2.1]{CuntzHolmJorgensen}, see also \cite{CuntzHolm}.
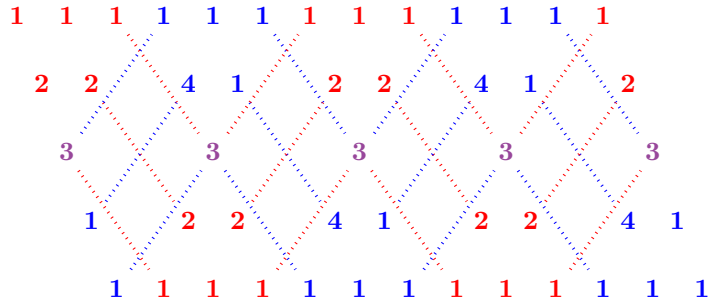
\begin{figure}%
\begin{tikzcd}[row sep=1.2em, column sep=-0.5em]
\ye{\mathbf{1}}&&\ye{\mathbf{1}} &&\ye{\mathbf{1}} \arrow[ddrrrr,dotted, very thick, -,color=red]&& \bl{\mathbf{1}} && \bl{\mathbf{1}}   && \bl{\mathbf{1}} \arrow[ddrrrr,dotted, very thick, -,color=blue]  && \ye{\mathbf{1}}  && \ye{\mathbf{1}}   && \ye{\mathbf{1}} \arrow[ddrrrr,dotted, very thick, -,color=red]  && \bl{\mathbf{1}} &&\bl{\bl{\mathbf{1}}} && \bl{\mathbf{1}} \arrow[ddrrrr,dotted, very thick, -,color=blue] && \ye{\mathbf{1}}\\
&\ye{\mathbf{2}}&& \ye{\mathbf{2}} \arrow[ddrrrr,dotted, very thick, -, color=red]&& \phantom{\bl{1}} && \bl{\mathbf{4}}     && \bl{\mathbf{1}} \arrow[ddrrrr,dotted, very thick, -,color=blue]    && \phantom{2}     && \ye{\mathbf{2}}     && \ye{\mathbf{2}}   \arrow[ddrrrr,dotted, very thick, -,color=red]  && \phantom{1}     && \bl{\mathbf{4}} && \bl{\mathbf{1}} \arrow[ddrrrr,dotted, very thick, -,color=blue]&& \phantom{\bl{2}}  && \ye{\mathbf{2}}\\
&& \gr{\mathbf{3}} \arrow[ddrrrr,dotted, very thick, -,color=red]  \arrow[uurrrr,dotted, very thick, -,color=blue] &&\phantom{\col{1}}&&\phantom{\col{3}}&& \gr{\mathbf{3}}  \arrow[ddrrrr,dotted, very thick, -,color=blue]  \arrow[uurrrr,dotted, very thick, -,color=red]   && \phantom{1}     && \phantom{3}     && \gr{\mathbf{3}}  \arrow[ddrrrr,dotted, very thick, -,color=red]  \arrow[uurrrr,dotted, very thick, -,color=blue]   && \phantom{1}    && \phantom{3}     && \gr{\mathbf{3}} \arrow[ddrrrr,dotted, very thick, -,color=blue]  \arrow[uurrrr,dotted, very thick, -,color=red] &&\phantom{\col{1}} &&\phantom{\col{3}} &&\col{\gr{\mathbf{3}}} \\
&&& \bl{\mathbf{1}} \arrow[uurrrr,dotted, very thick, -,color=blue] &&\phantom{\ye{2}}&&\ye{\mathbf{2}}&& \ye{\mathbf{2}}   \arrow[uurrrr,dotted, very thick, -,color=red]    && \phantom{1}     && \bl{\mathbf{4}}     && \bl{\mathbf{1}}    \arrow[uurrrr,dotted, very thick, -,color=blue]   && \phantom{2}     && \ye{\mathbf{2}}     && \ye{\mathbf{2}}  \arrow[uurrrr,dotted, very thick, -,color=red] && \phantom{\ye{1}} && \bl{\mathbf{4}} && \bl{\mathbf{1}}\\
&& && \bl{\mathbf{1}} \arrow[uurrrr,dotted, very thick, -,color=blue]  &&\ye{\ye{\mathbf{1}}}&&\ye{\ye{\mathbf{1}}} && \ye{\mathbf{1}}   \arrow[uurrrr,dotted, very thick, -,color=red]  && \bl{\mathbf{1}}  && \bl{\mathbf{1}}   && \bl{\mathbf{1}}  \arrow[uurrrr,dotted, very thick, -,color=blue]   && \ye{\mathbf{1}}   && \ye{\mathbf{1}}  && \ye{\mathbf{1}}  \arrow[uurrrr,dotted, very thick, -,color=red]  && \bl{\mathbf{1}} &&\bl{\mathbf{1}} && \bl{\mathbf{1}}
\end{tikzcd}
\caption{Reduction of the frieze to $M^{\perp_1}_{14}$.}
\label{f:redFrieze26differentT}
\end{figure}

In particular, one can see that the modified frieze rule is implied by the multiplication formula for the cluster character \cite[Defn.~3.1]{FuKeller}.
In the combinatorial picture (Figure~\ref{f:redFrieze26differentTcombinatorial}), one again gets two quadrilaterals glued together at the arc $(1,4)$, but now the quadrilateral with vertices $1,2,3,4$ is not triangulated.
 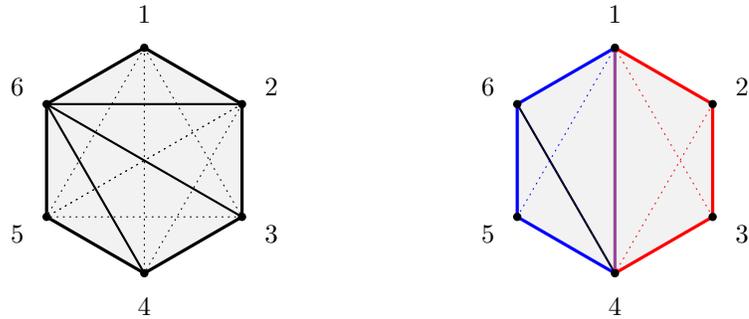
\begin{figure}%
    \begin{tikzpicture}[scale=3,cap=round,>=latex]

    \node  at (90:0.65) {$1$};
        \node   at (30:0.65) {$2$};
    \node  at (330:0.65) {$3$};
        \node at (270:0.65) {$4$};
    \node at (210:0.65) {$5$};
        \node at (150:0.65) {$6$};
         
    \coordinate  (p1) at (90:0.5) {}; 
        \coordinate  (p2) at (30:0.5){};
   \coordinate (p3) at (330:0.5){};
    \coordinate (p4) at (270:0.5){};
   \coordinate (p5) at (210:0.5){}; 
    \coordinate (p6) at (150:0.5){};

 \draw[fill=black!5] (p1)--(p2)--(p3)--(p4)--(p5)--(p6)--(p1); 

\draw[-,thick] (p2)--(p6); 
\draw[-,thick] (p3)--(p6); 
\draw[-,thick] (p4)--(p6); 
  
   \draw[-,very thick] (p1)--(p2); 
      \draw[-,very thick] (p2)--(p3); 
      \draw[-,very thick] (p3)--(p4); 
            \draw[-,very thick] (p4)--(p5); 
       \draw[-,very thick] (p5)--(p6); 
         \draw[-,very thick] (p6)--(p1); 

             \draw[dotted] (p2)--(p4); 
           \draw[dotted] (p2)--(p5);
           \draw[dotted] (p1)--(p4); 
             \draw[dotted] (p2)--(p5);             
              \draw[dotted] (p3)--(p5); 
             \draw[dotted] (p1)--(p5);  
              \draw[dotted] (p1)--(p3);

	\draw (p1) node[fill=black,circle,inner sep=0.039cm] {} circle (0.01cm);	        
     \draw (p2) node[fill=black,circle,inner sep=0.039cm] {} circle (0.01cm);
     \draw (p3) node[fill=black,circle,inner sep=0.039cm] {} circle (0.01cm);
     \draw (p4) node[fill=black,circle,inner sep=0.039cm] {} circle (0.01cm);
     \draw (p5) node[fill=black,circle,inner sep=0.039cm] {} circle (0.01cm);
     \draw (p6) node[fill=black,circle,inner sep=0.039cm] {} circle (0.01cm);
    \end{tikzpicture}
       \ \hspace{2cm} \
     \begin{tikzpicture}[scale=3,cap=round,>=latex]

    \node  at (90:0.65) {$1$};
        \node   at (30:0.65) {$2$};
    \node  at (330:0.65) {$3$};
        \node at (270:0.65) {$4$};
    \node at (210:0.65) {$5$};
        \node at (150:0.65) {$6$};
         
    \coordinate  (p1) at (90:0.5) {}; 
        \coordinate  (p2) at (30:0.5){};
   \coordinate (p3) at (330:0.5){};
    \coordinate (p4) at (270:0.5){};
   \coordinate (p5) at (210:0.5){}; 
    \coordinate (p6) at (150:0.5){};

 \draw[fill=black!5] (p1)--(p2)--(p3)--(p4)--(p5)--(p6)--(p1); 

\draw[-,thick] (p4)--(p6);   

   \draw[-,red,very thick] (p1)--(p2); 
      \draw[-,red,very thick] (p2)--(p3); 
      \draw[-,red,very thick] (p3)--(p4); 
            \draw[-,blue,very thick] (p4)--(p5); 
       \draw[-,blue,very thick] (p5)--(p6); 
         \draw[-,blue,very thick] (p6)--(p1); 
             \draw[-,Purple,very thick] (p1)--(p4);  %

             \draw[dotted,red] (p2)--(p4); 
           \draw[dotted,red] (p3)--(p1);
           \draw[dotted,blue] (p5)--(p1);            
              \draw[dotted,blue] (p4)--(p6);

	\draw (p1) node[fill=black,circle,inner sep=0.039cm] {} circle (0.01cm);	        
     \draw (p2) node[fill=black,circle,inner sep=0.039cm] {} circle (0.01cm);
     \draw (p3) node[fill=black,circle,inner sep=0.039cm] {} circle (0.01cm);
     \draw (p4) node[fill=black,circle,inner sep=0.039cm] {} circle (0.01cm);
     \draw (p5) node[fill=black,circle,inner sep=0.039cm] {} circle (0.01cm);
     \draw (p6) node[fill=black,circle,inner sep=0.039cm] {} circle (0.01cm);
    \end{tikzpicture}
    \caption{Combinatorial picture for the triangulation corresponding to $T'=M_{26}\oplus M_{36}\oplus M_{46}$, reduced at $M_{14}$.}
\label{f:redFrieze26differentTcombinatorial}
     \end{figure}
\end{example}

\begin{remark}
Similar to Examples \ref{Ex:CCfriezered} and \ref{Ex:redfriezecoeff}, for any $\C_{2,n}$ one obtains a Conway--Coxeter frieze of width $n-3$ by first considering the cluster character $\Phi_{\C_{2,n}}^T$, where $T$ is a cluster tilting module, and then specialising the values of $\Phi_{\C_{2,n}}^T(T_i)$ to $1$ for all $i=1, \ldots, 2n-3$.
Then one also gets a cluster character $\Phi_{\C_{2,n}}^T|_{M^{\perp_1}_{ij}}$ on any subcategory $M^{\perp_1}_{ij}$ for $i \neq j+1$ by Theorem \ref{t:cc_restrict}.
If $M_{ij}$ is a direct summand of $T$, then the combinatorial picture is as in Example \ref{Ex:CCfriezered} and the reduced frieze can be thought of as two smaller Conway--Coxeter friezes glued together at the entry corresponding to $M_{ij}$.
If $M_{ij}$ is not a direct summand of $T$, then the reduction of the frieze will be two friezes with coefficients glued together at the entry corresponding to $M_{ij}$ as in Example \ref{Ex:redfriezecoeff}.
This gives a more representation-theoretic interpretation of friezes with coefficients \cite{CuntzHolmJorgensen}.
\end{remark}


\newcommand{\etalchar}[1]{$^{#1}$}

\end{document}